\documentclass{amsart}
\usepackage{amsmath,amsthm}
\usepackage{graphics}
\usepackage{color}
\usepackage{epsfig}

\theoremstyle{plain}
\newtheorem{thm}{Theorem}[section]
\newtheorem{cor}[thm]{Corollary}
\newtheorem{conj}[thm]{Conjecture}
\newtheorem{prop}[thm]{Proposition}
\newtheorem{lemma}[thm]{Lemma}
\newtheorem{claim}[thm]{Claim}
\newtheorem{addendum}[thm]{Addendum}
\newtheorem{question}{Question}

\newtheorem*{thm-eight}{Theorem 8.1}
\newtheorem*{thmapp}{Theorem~\ref{thm:appendix}}

\newcounter{prp}

\theoremstyle{definition}
\newtheorem{defn}[thm]{Definition}
\newtheorem{remark}[thm]{Remark}
\newtheorem{property}[prp]{Property}

\newcommand{\arc}[1]{\ensuremath{\overline{#1}}}
\newcommand{\edge}[1]{\ensuremath{(#1)}}
\newcommand{\hatF}{\ensuremath{{\widehat{F}}}}
\newcommand{\hatQ}{\ensuremath{{\widehat{Q}}}}

\newcommand{\calA}{\ensuremath{{\mathcal A}}}
\newcommand{\calB}{\ensuremath{{\mathcal B}}}
\newcommand{\calN}{\ensuremath{{\mathcal N}}}
\newcommand{\cala}{\ensuremath{{\mathfrak a}}}    
\newcommand{\calL}{\ensuremath{{\mathcal L}}}
\newcommand{\calT}{\ensuremath{{\mathcal T}}}
\newcommand{\calC}{\ensuremath{{\mathcal C}}}
\newcommand{\calD}{\ensuremath{{\mathcal D}}}
\newcommand{\calE}{\ensuremath{{\mathcal E}}}
\newcommand{\calF}{\ensuremath{{\mathcal F}}}

\newcommand{\B}{\ensuremath{{\mathfrak B}}}
\newcommand{\trunk}{\ensuremath{{\mathfrak t}}}

\newcommand{\comment}[1]{}

\newcommand{\bdry}{\ensuremath{\partial}}
\DeclareMathOperator{\Int}{Int}

\DeclareMathOperator{\nbhd}{Nbhd}
\newcommand{\N}{\ensuremath{\mathbb{N}}}

\newcommand{\R}{\ensuremath{\mathbb{R}}}

\newcommand{\C}{\ensuremath{\mathbb{C}}}
\newcommand{\I}{{\mbox{{\rm \scriptsize{I}}}}}
\newcommand{\II}{{\mbox{{\rm \scriptsize{I\!I}}}}}
\newcommand{\III}{{\mbox{{\rm \scriptsize{I\!I\!I}}}}}
\newcommand{\SII}{\ensuremath{\Sigma_{\II}}}

\newcommand{\g}{\ensuremath{{\mathfrak g}}}
\newcommand{\G}{\ensuremath{{\mathfrak G}}}
\newcommand{\mobius}{M\"{o}bius }
\newcommand{\AO}{A\Theta}
\newcommand{\ao}{a\theta}

\def\smallcoprod{\raise.3ex\hbox{$\,\scriptstyle\coprod\,$}}
\def\C{{\mathcal C}}

\begin{document}
\title{Bridge number, Heegaard genus and non-integral Dehn surgery}
\author{Kenneth L.\ Baker}
\address{Department of Mathematics \\ University of Miami \\ Coral Gables, FL 33146}
\email{kb@math.miami.edu}

\author{Cameron Gordon}
\address{Department of Mathematics \\ The University of Texas at Austin \\ Austin, TX 78712}
\email{gordon@math.utexas.edu}

\author{John Luecke}
\address{Department of Mathematics \\ The University of Texas at Austin \\ Austin, TX 78712}
\email{luecke@math.utexas.edu}
\begin{abstract}
We show there exists a linear function $w \colon \N \to \N$ with the following
property.
Let $K$ be a hyperbolic knot in a hyperbolic $3$--manifold $M$ admitting a 
non-longitudinal $S^3$ surgery. If $K$ is put into thin position with respect to a strongly irreducible, genus $g$ Heegaard splitting of $M$ then $K$ intersects a thick level at most $2w(g)$ times. Typically, this shows that the bridge 
number of $K$ with respect to this Heegaard splitting is at most $w(g)$,
and the tunnel number of $K$ is at most $w(g) + g-1$.   
\end{abstract}
\maketitle


\section{Introduction}

Let $M = K'(p/q)$ be the manifold obtained by $p/q$--Dehn surgery on a knot 
$K'$ in $S^3$. 
(Here $p/q \in \mathbb Q$, $q > 0$, parametrizes the slope on the boundary of 
the exterior of $K'$ that runs $p$ times meridionally and $q$ times 
longitudinally around $K'$.) 
We denote the core of the attached solid torus in $K'(p/q)$ by $K$. 
Thus $K$ is a knot in $M$ with a Dehn surgery that yields $S^3$, and it is 
natural to investigate what properties of $K$ this entails. 
In the present paper we are interested in the relationship between $K$ and 
the Heegaard splittings of $M$; more specifically, if $S$ is a Heegaard 
surface for $M$, what can we say about the bridge number 
$br(K) = br_S(K)$ of $K$ with respect to $S$? 
First note that if $S$ is a Heegaard surface for $X = S^3 - N(K')$, the 
exterior of $K'$, i.e. $S$ separates $X$ into a handlebody and a compression 
body, then $S$ becomes a Heegaard surface for $M$ such that 
$br_S(K) = 0$. 
In general not every Heegaard surface for $M$ arises in this way; for 
example the Heegaard genus of $M$ may be smaller than that of $X$. 
Nevertheless, it is shown in \cite{rs1} (see also \cite{moriahrubenstein} and
\cite{rieck}) that, given $K'$, for all but finitely 
many slopes $p/q$ we have $br_S(K) = 0$ for any Heegaard surface $S$ of $M$.

Recall that if $\gamma,\tau$ are two isotopy classes of simple closed
curves on a 2-torus, then the {\em distance} between $\gamma$ and $\tau$,
denoted $\Delta(\gamma,\tau)$ is defined to be the absolute value of the
intersection number between $\gamma$ and $\tau$.
Note that $q = \Delta(p/q,1/0)$, the distance of $p/q$ from the meridian 
of $K'$ on $\partial X$. 
Since the trivial Dehn surgery $K'(1/0) = S^3$ represents the maximal 
possible degeneration of Heegaard genus, one would expect the Heegaard 
splittings of $K'(p/q)$ to behave better as $q$ gets large. 
Indeed, it follows from \cite{rieck} that if $K'$ is hyperbolic and $S$ is a 
Heegaard surface of genus $g$ for $K'(p/q)$, then $q \ge 18(g+1)$ 
implies $br_S(K) = 0$.

Here we consider the question whether there is an upper bound on 
$br_S(K)$ that depends only on the genus of $S$:

\begin{question}
Is there a function $w : \mathbb N \to \mathbb N$ such that if $K'$ is a 
knot in $S^3$ and $S$ is a Heegaard surface of genus $g$ for $M = K'(p/q)$, 
where $q > 0$, then $br_S(K) \le w(g)$?
\end{question}

It turns out that the answer to Question~1 is ``no'' in general when $q = 1$; 
see Remark~\ref{rem1-6}  below. 
However, Corollary~\ref{cor1-1} below says 
that the answer is ``yes'' if $q \ge 2$, provided that $K'$ is 
hyperbolic, $M$ is not a special kind of 
Seifert fibered space, and $p/q$ is not a boundary slope for $X$.
Corollary~\ref{cor1-1} follows from our main result, Theorem~8.1, 
where we consider more generally the number of intersections of 
$K$ with a thick level surface in a Heegaard splitting of $M$ with 
respect to which $K$ is in thin position.

\begin{thm-eight}    
There is a linear function $w : \mathbb N \to \mathbb N$ with the following 
property. 
Let $K'$ be a hyperbolic knot in $S^3$, $M = K'(p/q)$ where $q \ge 2$, 
and $K$ the core of the attached solid torus in $M$. 
Suppose $K$ is in thin position with respect to a genus $g$ Heegaard 
splitting of $M$ and let $S$ be a corresponding thick level surface. 
If $S$ is a strongly irreducible Heegaard surface for $M$
then either
\begin{itemize}
\item[(1)] $|K\cap S| \le 2w(g)$; or

\item[(2)] $M$ is toroidal; or

\item[(3)] $M$ is a Seifert fibered space over the 2-sphere with exactly 
three exceptional fibers, at least one of which has order 2 or 3. 
\end{itemize}
Furthermore, in cases (2) and (3) $M$ has a genus~2 
Heegaard splitting with respect to 
which $K$ has bridge number~$0$ in case (2) and at most $w(2)$ in case (3). 
\end{thm-eight}

In \S\ref{sec:proofofmain} we show that  
$w(g)=10,581(g-1)+394$ works.
The technical assumption in Theorem~\ref{thm:main} that $S$ is 
strongly irreducible is likely unnecessary. It is only used in the proof of 
Lemma~\ref{lem:parallelesc}. 

If $X$ is a compact, orientable 3-manifold with torus boundary 
let us say that a slope $r$ on $\partial X$ is a {\it $(g,b)$--boundary slope 
for $X$} if there is a compact, connected, orientable essential surface 
$(F,\partial F) \subset (X,\partial X)$ with boundary slope $r$, such that 
$F$ has genus at most $g$ and at most $b$ boundary components. We say
that $r$ is a {\em $g$--boundary slope} if it is a $(g,b)$--boundary slope for
some $b$.

\begin{cor}\label{cor1-1} 
There is a linear function $w : \mathbb N \to \mathbb N$ with the 
following property. 
Let $K'$ be a hyperbolic knot in $S^3$,  $M = K'(p/q)$ where $q \ge 2$, 
and $K$ the core of the attached solid torus in $M$. 
Let $S$ be a genus $g$, strongly irreducible Heegaard surface in $M$. 
If $p/q$ is not a $(g,2w(g)-2)$--boundary slope for the exterior of $K'$ then 
either $br_S(K) \le w(g)$, or conclusion (3) of Theorem~8.1 holds. 
\end{cor}

\begin{proof}
For terminology on thin presentations, see section~\ref{sec:basics}.
Let $w$ be the function of Theorem~\ref{thm:main}. Assume $K'(p/q)$ has 
a genus $g$ splitting which is strongly irreducible, and that $p/q$ is not a 
$(g,2w(g)-2)$--boundary slope. In particular, $M$ cannot be toroidal since, otherwise,
by \cite{gl:nitds} 
$p/q$ is a $(1,2)$--boundary slope (and $g > 1$).
Put $K$ in thin position with respect to this splitting
and let $S$ be a thick level surface.
Theorem~\ref{thm:main} says that either $M$ is the Seifert
fibered space of its conclusion $(3)$ and that $K$ has bridge number at most $w(2)$
with respect to a genus $2$ splitting of $M$,
or $|K \cap S| \le 2w(g)$. This thin presentation of $K$ must be a bridge
presentation of $K$, otherwise there will be a thin level surface in the thin 
presentation which intersects $K$ fewer times than does $S$, and
Lemma~\ref{lem:thin=bridge} would contradict that $p/q$ is not a 
$(g,2w(g)-2)$--boundary slope. Thus $K$ is bridge with respect to $S$ and
 $br_S(K) \le w(g)$.
\end{proof}

\begin{remark}\label{rem1-2}
By \cite{hatcher}, at most finitely many slopes of the exterior of $K'$ are boundary 
slopes. 
Note that if $M$ is non-Haken, then $p/q$ cannot be a boundary slope, 
by Theorem~2.0.3 of \cite{cgls:dsok} ($q\ge2$ and $M$ is irreducible by 
\cite{gl:oidscyrm}).
If $M$ is a Seifert fibered space over the 2-sphere with at most three 
exceptional fibers, then, since $p \neq 0$,  $M$ is non-Haken.
\end{remark}

\begin{remark}\label{rem1-3}
It is conjectured that a Seifert fibered space can never be obtained by 
non-integral Dehn surgery on a hyperbolic knot in $S^3$ --- that is, that 
(3) of Theorem~8.1 never occurs. 
Theorem~2.4 of \cite{bgl:og2hsfs} implies that if $M$ is Seifert fibered over 
the 2-sphere 
with three exceptional fibers and $q>2$ then $K$ has bridge number at most 
1 with respect to some genus~2 Heegaard splitting of $M$, though possibly 
a 1-sided splitting (which are not considered in this paper). 
Note that since a 1-sided splitting corresponds to a non-orientable surface, 
if $p$ is odd then $K'(p/q)$ does not have a 1-sided splitting. 
Thus when $q>2$ and $p$ is odd, Theorem~2.4 of \cite{bgl:og2hsfs} 
sharpens conclusion (3) above.
\end{remark}
\begin{remark}\label{rem1-4}
If $M$ is toroidal, then by \cite{gl:nitds} ($q = 2$ and) $K'$ belongs to 
the family of knots described by Eudave-Mu\~noz in \cite{em:nhmobdsohk}. 
In \cite{em:otkklmnp} Eudave-Mu\~noz shows that these knots have tunnel number~1, and 
hence $br(K) = 0$ with respect to a minimal genus (genus~2) splitting of $M$.
\end{remark}

\begin{remark}\label{rem1-5}
We conjecture that the hypotheses that $K'$ is hyperbolic in 
Theorem~8.1  and 
Corollary~\ref{cor1-1}, and that $p/q$ is not a 
$(g,2w(g)-2)$--boundary slope in Corollary~\ref{cor1-1}, 
are redundant.
\end{remark}

\begin{remark}\label{rem1-6}
When $q = 1$ the answer to Question 1 is ``no'' in general. 
In \cite{teragaito} Teragaito 
constructs infinitely many hyperbolic knots $K'_n$ 
in $S^3$ such that 4-surgery on $K'_n$ gives the same 3-manifold $M$ 
for all $n$, where $M$ is a Seifert fibered space of type $S^2(2,6,7)$. 
In particular $M$ has Heegaard genus~2. 
Let $K_n$ be the core of the attached solid torus in $K'_n(4) = M$, and define 
$b(K_n) = \min \{br_S(K_n) : S$ a genus~2 Heegaard surface for $M\}$. 
It is shown in \cite{bgl:bnaids} that $b(K_n)$ is unbounded.
\end{remark}

\begin{remark}\label{rem1-7}
For small values of $g$ the bridge number $br(K)$ is either known or 
conjectured to be very small. 
For example, the impossibility of getting $S^3$ by non-trivial Dehn 
surgery on a non-trivial knot \cite{gl:kadbtc} 
can be expressed as saying that if $g=0$ and $q>0$ then $br (K)=0$.
When $g=1$, $K'(p/q)$ is a lens space and here the Cyclic Surgery 
Theorem \cite{cgls:dsok}
says that if $q>1$ then $K'$ is a torus knot, which is easily seen to 
imply $br(K)=0$, while if $q=1$ and $K'$ is hyperbolic the Berge Conjecture \cite{berge}
is equivalent to the assertion that $br(K)=1$.
In another paper (\cite[Theorem~2.4]{bgl:og2hsfs}), we consider the case 
$g=2$ and 
show that if $q>2$ then, generically, $br(K) \le 1$ (with respect to some 
genus~2 splitting, possibly 1-sided). 
\end{remark}

\begin{remark}
Strongly irreducible Heegaard surfaces and closed incompressible surfaces may exhibit similar behavior.  With this theme in mind, in the Appendix we adapt the proof of Theorem~\ref{thm:main} to bound the intersection number of $K$ with an incompressible surface rather than a thick Heegaard surface in terms of the genus of the surface:
%
\begin{thmapp} 
There is a linear function $w_I \colon \mathbb{N} \to \mathbb{N}$ with the following
property.
Let $K'$ be a hyperbolic knot in $S^3$, $M = K'(p/q)$ where $q \ge 2$,
and $K$ the core of the attached solid torus in $M$.
Let $S$ be an orientable, incompressible surface in $M$ of
genus $g$.
Then $K$ can be isotoped to intersect $S$ at most $w_I(g)$ times.
\end{thmapp}
%
\end{remark}

\begin{remark}
Osoinach shows that for integral surgeries the above does not hold.
In \cite{osoinach}, he gives examples of infinitely many different
knots in the 3-sphere on which $0$--surgery gives the same manifold. 
This manifold has
an essential torus, and in \cite{osoinachdiss} he shows 
that the set of 
minimal intersection numbers in $M$ of this torus with the corresponding 
cores of the attached
solid tori must be infinite. For a sharpening of Osoinach's result see \cite{teragaitotoroidal}.
\end{remark}

Our results give information on the relationship between the Heegaard genus 
of $M$ and that of $X$, the exterior of $K'$. 
Recall that a Heegaard splitting of $X$ is a decomposition $X = V\cup_S W$,
where $V$ is a handlebody with $\partial V = S$ and $W$ is a compression 
body with $\partial W = S\sqcup \partial X$. 
The {\em Heegaard genus} $g(X)$ of $X$ is the minimal genus of 
$S$ over all such decompositions. 
In this context one often talks about the {\em tunnel number} $t(K')$ 
of $K'$, the minimum number of arcs (``tunnels'') that need to be attached to 
$K'$ so that the complement of an open regular neighborhood of the resulting 
1-complex is a handlebody.
It is easy to see that $g(X) = t(K')+1$.
For any slope $p/q$, $V\cup_S W(p/q)$ is a Heegaard splitting of 
$M= K'(p/q)$; in particular $g(M) \le g(X)$. 
The question arises as to what extent Heegaard genus can decrease under 
Dehn filling, i.e. how bad the inequality $g(M) \le g(X) = t(K') + 1$ 
can be. 
Since it is easy to see that if $K$ has bridge number $br(K)$ with respect 
to a genus~$g$ splitting of $M$ then $t(K) (=t (K')) \le g+br (K)-1$, 
Corollary~\ref{cor1-1} 
gives a linear bound on the extent to which Heegaard genus 
can decrease under a non-integral Dehn surgery.

\begin{cor}\label{cor1-8}
There is a linear function $w_{TN} :\N \to \N$ with the following 
property. 
Let $K'$ be a hyperbolic knot in $S^3$ and $M= K(p/q)$. 
If $M$ has Heegaard genus~$g$, $p/q$ is not a $(g,2(w_{TN}(g) -g))$--boundary 
slope for the exterior of $K'$, and $q\ge 2$, then the tunnel number of 
$K'$ is 
at most $w_{TN}(g)$.
\end{cor}

We conjecture that there is always such a universal bound on the tunnel number
of a knot in terms of the Heegaard genus of 
any of its (non-trivial Dehn surgeries):

\begin{conj}\label{conj1-9}
There is a function $w_{TN} :\N \to \N$ such that if $K'$ is a knot in 
$S^3$ and $M = K(p/q)$ is a non-trivial Dehn surgery on $K'$ with Heegaard 
genus~$g$, then the tunnel number of $K'$ is at most $w_{TN}(g)$.
\end{conj} 

Information on the question of degeneration of Heegaard genus under Dehn 
filling in provided by Rieck and Sedgwick in \cite{rs1} and \cite{rs2}. 
As mentioned above, it is shown in \cite{rs1} that for all but finitely many 
slopes $p/q$, $ br(K) = 0$ with respect to
any Heegaard surface $S$ of $M$.
Taking $S$ to have minimal genus, it is easy to see that 
$br(K) =0$ implies that either 
$g(M) = g(X) = t (K') +1$ or $g(M) = g(X) -1 = t(K')$.
See \cite{rieck} for details.
By \cite{rs2}, the second possibility can happen for only a finite number of 
lines of slopes (where a {\em line\/} is a set of slopes $r$ such that 
$\Delta (r,r_0)=1$ for some fixed slope $r_0$). 

Regarding Corollary~\ref{cor1-8} and Conjecture~\ref{conj1-9}, 
in fact we know no examples 
where the Heegaard genus of $K'(p/q)$ ($q>0$) is less than $t(K')$. 
So we ask 

\begin{question}
Is $t(K') \le g(K'(p/q))$ for all $q>0$?
\end{question}

For example, the answer to Question~2 is known to be ``yes'' when 
\begin{itemize}
\item $g(K'(p/q)) =0$ (\cite{gl:kadbtc})
\item $g(K'(p/q)) =1$ and $q >1$ (\cite{cgls:dsok})
\item $g(K'(p/q)) =2$, $q>2$, and $K'(p/q)$ does not contain an incompressible
surface of genus~2. (\cite{bgl:og2hsfs})
\end{itemize}

Finally, the bound on bridge number in Corollary~\ref{cor1-1}
allows us to use a result of Tomova \cite{tomova} to get a statement about the 
distance of splittings of exteriors of knots with genus~$g$ Dehn surgeries. 
If $S$ is a Heegaard surface for some 3-manifold, we denote by $d(S)$ the 
{\em distance\/} of the corresponding splitting; see \cite{hempel}. 

\begin{cor}\label{cor1-10}
There is a linear function $w_{HD} : \N \to\N$ with the following property.
Let $K'$ be a hyperbolic knot in $S^3$ whose exterior has a genus~$g$ 
Heegaard surface $S$ with $d(S) > w_{HD}(g)$.
If $p/q$ is not a $(g,w_{HD}(g)-2g-2)$--boundary slope for the exterior of 
$K'$ and $q\ge2$, then $g(K'(p/q))=g$.
\end{cor}

Thus the distance of a splitting of a knot exterior is putting a limit 
on the degeneration of Heegaard genus under Dehn filling.
For instance, this applies to the examples of \cite{mms}.
Recall that a knot has only finitely many boundary slopes (\cite{hatcher}). 
By \cite{mms}, for any $g\ge 2$, there are knots $K'$ in $S^3$ whose exteriors 
have genus~$g$ Heegaard splittings $S$ with $d(S) > w_{HD}(g)$, in fact 
with $d(S)$ arbitrarily large (such knots are necessarily hyperbolic). 
(The case $g=2$ was first done in \cite{johnson}.) 
Corollary~\ref{cor1-10}
says that for such a knot $K'$, if $q\ge2$ and $p/q$ is not 
a $(g,w_{HD}(g)-2g-2)$--boundary slope, then $K'(p/q)$ has Heegaard genus $g$.  For results in a similar vein see \cite{bcjtt}.

\begin{proof}[Proof of Corollary~\ref{cor1-10}] 
Let $K'$, $p/q$, $S$ be as in the hypothesis.
Set $w_{HD}(g) = 2g + 2w(g)$.
Assume for contradiction that $g(K'(p/q)) = g' < g$.
By Corollary~\ref{cor1-1}, the bridge number of $K$ with respect to some 
genus~$g'$ Heegaard surface $\widehat F$ of $K'(p/q)$ is at most $w(g')$.
Thus $K$ can be put in bridge position with respect to $\widehat F$ so that 
$2-\chi (\widehat F - K) = 2- (2-2g' - 2w(g')) \le 2g+2w(g) = w_{HD}(g)$.
Then $d(S) >2 - \chi (\widehat F - K)$ by assumption, and the main result 
of \cite{tomova} implies that, in $K'(p/q)$, $\widehat F$ is isotopic to a 
stabilization of $S$.
But $\widehat F$ has smaller genus than $S$.
\end{proof}

Throughout this article $M$ will be the manifold $K'(p/q)$ obtained by 
$p/q$--Dehn surgery on a hyperbolic knot $K'$ in $S^3$, with $q \ge 2$. 
It follows that
\begin{itemize}
\item[(1)] $M$ is irreducible \cite{gl:oidscyrm};

\item[(2)] $M$ is not a lens space (or $S^3$) \cite{cgls:dsok};

\item[(3)] $M$ does not contain a Klein bottle \cite{gl:dsokcetI}, \cite{boyerzhang}.
\end{itemize}

Note that (1) and (2) together imply that $M$ does not contain a projective plane.
We also assume  (see Remark~\ref{rem1-4})

\begin{itemize}
\item[(4)] $M$ is atoroidal.
\end{itemize}

\subsection{Overview of the proof of Theorem~\ref{thm:main}}

The proof of Theorem~\ref{thm:main} occupies the remainder of this article and culminates in \S\ref{sec:proofofmain}.  We give an overview, taking the
notation from Theorem~\ref{thm:main}. 

We begin in  \S\ref{sec:fatvertexedgraphs} by briefly reviewing and setting up notation for the notions of thin position and fat-vertexed graphs of intersection between (punctured) Heegaard spheres of $S^3$ and Heegaard 
surfaces of $M$. In particular, $Q$ and $F$ are surfaces in the exterior of $K'$ coming from
a Heegaard sphere, $\hatQ$, in $S^3$ and the given thick level Heegaard surface,
$\hatF=S$, in $M$. $G_Q,G_F$ are the graphs of intersection between $Q,F$.
Let $t=|K \cap {\hatF}|=|K \cap S|$ be
the number of components of $\partial F$. Then $t$ is also the number of
vertices of $G_F$ and consquently the number of different labels on the 
vertices of $G_Q$. 
Our goal is
to bound $t$ by some linear function of $g$.

Lemma~\ref{lem:2} shows the existence of a special subgraph, $\Lambda$, of 
$G_Q$ called a great $g$--web. For each label $x$ of this graph, 
$\Lambda_x$ is the
subgraph of $\Lambda$ consisting of those edges with label $x$. We consider the
set $\calL$ of all labels, $x$, for which $\Lambda_x$ has a bigon or
trigon face. In Proposition~\ref{prop:webcount} we make a key estimate
showing that up to an additive linear function in $g$, $|\calL| \geq
(3/4)t$. 

A bigon or trigon of $\Lambda_x$ gives rise to a bigon or trigon subgraph
in $\Lambda$. We classify these subgraphs as either 
extended Scharlemann cycles
or as trigons of Type I, I\!I, I\!I\!I. The extended Scharlemann cycles give
rise to {\em long \mobius bands} and {\em long twisted $\theta$--bands}
embedded in $M$ which intersect $\hatF$ in essential curves 
(or essential $\theta$--curves). 
We choose a minimal collection, $\Sigma$, of extended Scharlemann cycles
so that each label of $\calL$ that corresponds to an extended Scharlemann
cycle appears as a label in some element of $\Sigma$. Let $L(\Sigma)$ be
the set of all labels appearing in the extended Scharlemann cycles of $\Sigma$.
Those labels of $\calL$ that do not appear in $L(\Sigma)$ then correspond
to trigons of Type I, I\!I, or I\!I\!I. Section~\ref{sec:typeIorII} 
analyzes trigons of Type I or
I\!I and culminates in Theorem~\ref{thm:Type12boundN}, which bounds, in terms of a linear function in $g$,
those labels 
appearing in trigons of Type I or I\!I but not in $L(\Sigma)$ (the collection of these labels
is called $\calL_\II$). Section~\ref{sec:typeIII} 
shows that those 
labels appearing in Type I\!I\!I trigons but not in $L(\Sigma)$,
up to a linear function in $g$ (Lemma~\ref{lem:boundwrapping}), 
correspond to vertices in trivial curves 
on $\hatF$ called
simple gnarls. The collection of such labels is
called $\G$. That is, up to an added linear function in $g$, we have that
$|\calL|$ is $|L(\Sigma)| + |\G|$. 

Let $-\calL$ be those labels of $G_Q$ that are not in $\calL$.
The statements of this paragraph {\bf will all be up to an added linear
function of $g$}. The
curves of intersection, denoted $\ao(\Sigma)$, of  
$\hatF$ with the long \mobius bands or long twisted $\theta$--bands coming 
from $\Sigma$,
bound a collection of disjoint annuli of $\hatF$. Each curve
of $\ao(\Sigma)$ corresponds to two elements of $L(\Sigma)$, so 
there are $|L(\Sigma)|/2$ such annuli. Furthermore,
half of these annuli must contain in their interiors vertices of $G_F$
which correspond to labels in $G_Q$ which are in either
$-\calL$ or $\G$. This is Lemma~\ref{lem:buildA}, whose main
ingredient is Lemma~\ref{lem:triplets}. This sub-collection of annuli is called
$\calA$, and from the above we get $|\calA| \geq |L(\Sigma)|/4$. 
On the other hand, it
follows from Lemma~\ref{lem:nonemptygnarl} and Corollary~\ref{3ingnarl} that 
each simple gnarl 
must contain at least three vertices of $G_F$ corresponding to labels in 
$-\calL$. 
Counting those elements of $-\calL$ coming from from the interiors of
the annuli in $\calA$ and the gnarls in $\G$ we estimate (equation sequence
(**) in section~\ref{sec:proofofmain}) that $|-\calL| \ge (3/7)
(|L(\Sigma)|+|\G|) \ge (3/7) |\calL| \ge (3/7) (3/4) t$.

Thus $t=|-\calL| + |\calL| \geq (3/4)t + (9/28)t - c(g)$, 
where $c(g)$ is a linear function of $g$, thereby bounding $t$ by $14 c(g)$.

The full proof of Theorem~\ref{thm:main} is given at the beginning of
\S\ref{sec:proofofmain}.  In the course of this proof, we keep track of and explicitly determine suitable linear bounds.




\subsection{Acknowledgements}
In the course of this work KB was partially supported by NSF Grant DMS-0239600, by the University of Miami 2011 Provost Research Award, and by a grant from the Simons Foundation (\#209184 to Kenneth Baker).  KB would also like to thank the Department of Mathematics at the University of Texas at Austin for its hospitality during his visits.  These visits were supported in part by NSF RTG Grant DMS-0636643.  CG was partially supported by NSF Grant DMS-0906276. The authors would like to thank the referee for detailed and helpful comments.

\section{Fat vertexed graphs}\label{sec:fatvertexedgraphs}

\subsection{Heegaard splittings, thin position, and bridge position}\label{sec:basics}
In this paper, a Heegaard splitting will always be a 2-sided Heegaard 
splitting. Given such a Heegaard surface $S$ of a closed $3$--manifold $Y$ there is a product $S \times \R \subset Y$ so that $S = S \times \{0\}$ and the complement of the product is the union of the cores of the 
two handlebodies.  This defines a height function on the complement of the
the cores of the handlebodies.  Consider all the circles $C$ embedded in the product that are Morse with respect to the height function and represent the knot type of a particular knot $J$.  The following terms are all understood to be taken with respect to the Heegaard splitting.

Following \cite{gabai:fatto3mIII} (see also \cite{thompson:tpabnfkit3s}), the {\em width} of an embedded circle $C$ is the sum of the number of intersections $|C \cap S \times \{y_i\}|$ where one regular value $y_i$ is chosen between each pair of consecutive critical values.  The {\em width} of a knot $J$ is the minimum width of all such embeddings.  An embedding realizing the width of $J$ is a {\em thin position} of $J$, and $J$ is said to be {\em thin}.  If the critical point immediately below $y_i$ is a local minimum and the critical point immediately above $y_i$ is a local maximum, then the level $S \times \{y_i\}$ is a {\em thick level}. If the critical point immediately below $y_i$ is a local maximum 
and the critical point immediately above $y_i$ is a local minimum, then the 
level $S \times \{y_i\}$ is a {\em thin level}.

The minimal number of local maxima among Morse embeddings of $C$ is the 
{\em bridge number} of $J$, and denoted $br_S(J)$, or, if $S$ is understood, 
$br(J)$.  
An embedding realizing the bridge number of $J$ 
may be ambient isotoped so that all local maxima lie above all 
local minima, without introducing any more extrema.  
The resulting embedding is a {\em bridge position} of $J$, and 
$J$ is said to be {\em bridge}.

With $J$ in bridge position, the arcs of $J$ intersecting a Heegaard 
handlebody are collectively $\bdry$--parallel.  There is an embedded collection of disks in the handlebody such that the boundary of each is formed of one arc on $S$ and one arc on $J$.  A single such disk is called a {\em bridge disk} for that arc of $J$, and the arc is said to be {\em bridge}.

A thin position for a knot may have smaller width than that of its 
bridge position,
with respect to the same Heegaard splitting. That is, thin position may not
be bridge position.  
However, this only happens when the meridian of the knot in the ambient 
manifold is a boundary slope of the knot
exterior. Recall

\begin{defn}\label{def:2bdryslope}
If $E$ is a compact, orientable 3-manifold with torus boundary,
a slope $r$ on $\partial E$ is a {\it $(g,b)$--boundary slope 
for $E$} if there is a compact, connected, orientable essential surface 
$(F,\partial F) \subset (E,\partial E)$ with boundary slope $r$, such that 
$F$ has genus at most $g$ and at most $b$ boundary components. We say
that $r$ is a $g$--boundary slope if it is a $(g,b)$--boundary slope for
some $b$.
\end{defn}

\smallskip
\begin{lemma}~\label{lem:thin=bridge}
Assume $J$ is a knot in a 3-manifold $M$. If $J$ has a thin position which 
is not a bridge position with
respect to a genus $g$ Heegaard splitting of $M$, then the meridian of $J$ 
is a $(g,b)$--boundary
slope for the exterior of $J$ where $b$ is the number of intersections
of $J$ with any thin level surface in this thin presentation of $J$.
\end{lemma}

\begin{proof}
This is proved in \cite{thompson:tpabnfkit3s} when $g=0$. 
The same proof works here.
We sketch it for the convenience of the reader. 

Let $S$ be the Heegaard surface of a genus $g$ splitting of $M$ with
respect to which $J$ is in thin position but not bridge position.
Then this thin presentation must have a thin level, $y$. 
Let $b$ the number of intersections of $J$ with the thin level 
surface $S \times \{y\}$. There can be no 
bridge disks for $J$ to the thin level surface, else such a disk 
would give rise to a thinner
presentation of $J$. Maximally compress $(S \times \{y\}) - \nbhd(J)$
in the exterior of $J$. Either some component of the result is an 
incompressible, $\partial$--incompressible surface of genus at most $g$ 
with at most $b$ boundary components each of which is a meridian of $J$, 
or the result is a non-empty 
collection of
boundary parallel annuli along with some closed surfaces. But each boundary
parallel annulus gives rise to a bridge disk of $J$ onto $S \times \{y\}$,
which is not possible. Thus the meridian is a $(g,b)$--boundary slope for the 
exterior of $J$.
\end{proof}

\begin{lemma}\label{lem:FandQ}
Assume $K'$ is a non-trivial knot in $S^3$ with exterior $X=S^3-\nbhd(K')$.
Let $M = K'(p/q)$ and $K$ be the core of the attached solid torus in $M$. 
Suppose $M$ has a genus $g$ Heegaard surface and that $K$ cannot be
isotoped to lie on this surface.
Assume $K$ is in thin position with respect to this genus $g$ Heegaard 
splitting of $M$ and let $\hatF$ be a corresponding thick level surface. 
Then there is a 
punctured $2$--sphere $Q$ and a punctured genus $g$ surface $F$
properly embedded and transverse in $X$ satisfying the following:
\begin{enumerate}
\item Each component of $\partial Q$ is a meridian of $\partial X$ and each
component of $\partial F$ is a $p/q$ curve on $\partial X$.
\item Each arc of $F \cap Q$ is essential in each of 
$F$ and $Q$.
\item There are no simple closed curves of $F \cap Q$ 
trivial in
both $F$ and $Q$.
\item Capping off $F$ with disks in $K'(p/q)$ gives $\hatF$. 
Capping off $Q$ with meridians of $K'$ gives a 2-sphere, $\hatQ$, 
in $S^3$.
\end{enumerate}
\end{lemma}

\begin{proof}
Let $K',X,M,{\hatF}$ be as stated. Let $K$ be
the core of the attached solid torus in $M=K'(p/q)$. 
Isotop $K$ to be in thin position with respect to the given genus $g$
splitting of $M$ so that $\hatF$ is a thick level surface. 
In $S^3$, put $K'$ into thin position with respect to the genus $0$ 
Heegaard splitting.  
By Theorem 6.2 of \cite{rieck} (by assumption $K,K'$ cannot be isotoped onto 
their Heegaard surfaces), there exists a thick level surface
$\hatQ$ of $S^3$ such that each arc of 
$F \cap Q$ is essential in each of 
$F = \hatF - \nbhd(K)$ and $Q = \hatQ-\nbhd(K')$. Note that the argument
of Theorem 6.2 allows us to choose the fat layer with which we 
want to work.  Within the chosen fat layer, the intersection of any thick 
level surface with
$X$ is properly isotopic in $X$ to the intersection of any other thick
level surface with $X$.
As the exterior of $K'$ is irreducible, after an isotopy in $X$
we may assume there are no simple closed curves of $F \cap Q$ 
trivial in both $F$ and $Q$.
\end{proof}

\subsection{The graphs $G_F, G_Q$ and their combinatorics}\label{sec:graphcombo}
In this section we describe the graphs $G_F,G_Q$ and the great web
$\Lambda$, which will be the context of the rest of the paper.
Let $F,Q$ be as in Lemma~\ref{lem:FandQ}.
On $\hatQ$ and $\hatF$ form the {\em fat vertexed graphs of intersection},
$G_Q$ and $G_F$, respectively, consisting of the {\em fat vertices} that are the disks $\nbhd(K') \cap \hatQ$ and $\nbhd(K) \cap \hatF$ and {\em edges} that are the arcs of $F \cap Q$.

Choosing an orientation on $K \subset M$, we may number the intersections of $K$ with $\hatF$, and hence the vertices of $G_F$, from $1$ to $t=|K \cap \hatF|$ in order around $K$.  Similarly, if $|K' \cap \hatQ| = u$, by choosing an orientation on $K' \subset S^3$ we may number the intersections of $K'$ with 
$\hatQ$ and hence the vertices of $G_Q$ from $1$ to $u$ in order around $K'$.

Each component of $\bdry F$ intersects each component of $\bdry Q$ a total of 
$q$ times.  Thus a vertex of $G_Q$ has valence $q t$ and a vertex of $G_F$ has valence $q u$.  Since each component of $\bdry F \cap \bdry Q$ is an endpoint of an arc of $F \cap Q$, each endpoint of an edge in $G_Q$ may be labeled with the vertex of $G_F$ whose boundary contains the endpoint.  Thus around the boundary of each vertex of $G_Q$ the labels $\{1, \dots, t\}$ appear in order $q$ times.  
Similarly around the boundary of each vertex of $G_F$ the labels $\{1, \dots, u\}$ appear in order $q$ times.

See the expository article \cite{gordon:cmids} for a more thorough discourse on such fat-vertexed graphs and standard techniques in their use.

Endow each vertex of $G_Q$ (and $G_F$) with a sign of $+$ or $-$ according to whether or not the corresponding component of $\bdry Q$ ($\bdry F$) with its induced orientation is parallel or anti-parallel on $\bdry X$ to a chosen component of $\bdry Q$ ($\bdry F$).  Two vertices on the same graph are {\em parallel} if they have the same sign, {\em anti-parallel} if they have opposite signs.

The orientability of $F$ and $Q$ and the knot exterior gives the
following useful property of these graphs

\medskip

\noindent {\bf Parity Rule:} 
{\em An edge connects parallel vertices on one of $G_F, G_Q$
iff it connects anti-parallel vertices on the other.}

\medskip

\begin{defn} 
For a subgraph, $\Lambda$, of $G_Q$, a {\em ghost edge} for $\Lambda$ is an
edge of $G_Q$ which does not belong to $\Lambda$ but is incident to a vertex 
of $\Lambda$. The incidence of a ghost edge for $\Lambda$ with a fat vertex
of $\Lambda$ is called a {\em ghost label} of $\Lambda$. 
A connected subgraph, $\Lambda$, of $G_Q$ is called a {\em $g$--web}
if its vertices are parallel and if $\Lambda$ has at most $t+2g-2$ ghost
labels. If $U$ is a component of $\hatQ - \Lambda$ then we say $D = \hatQ-U$ is a {\em disk bounded by $\Lambda$}.    A {\em great $g$--web} is a $g$--web 
with the property that there is a disk, $D$, bounded by $\Lambda$ such that
$\Lambda=G_Q \cap D$. Note that as long as $t>2g-2$ and $q \geq 2$ then
a great $g$--web must contain at least two vertices ($G_Q$ has no 1-sided faces).
\end{defn}

\begin{lemma}\label{lem:2}
Let $G_Q,G_F$ be the graphs of intersection coming from $Q,F$ of 
Lemma~\ref{lem:FandQ} as described above. If $t > 2g-2$ and $q \geq 2$, 
then $G_Q$ has a great $g$--web, $\Lambda$.
\end{lemma}

\begin{proof} This is Theorem 6.1 of \cite{gordon:cmids} where $\hatF,t$ play
the role of $\widehat{P},p$, where the slope $\beta$ there is the meridional
slope of $K'$.  
  In \cite{gordon:cmids},
a great $g$--web is not necessarily connected, but by restricting to a connected 
component we may take it to be.
\end{proof}

The goal of Theorem~\ref{thm:main} is to bound $K \cap \hatF$
in terms of the genus of $\hatF$. Thus, after
Lemma~\ref{lem:2}, we will hereafter assume the existence in $G_Q$ of
the great $g$--web, $\Lambda$.

\begin{defn}
For each label $x \in \{1, \dots, t\}$, the subgraph of a great $g$--web $\Lambda$ consisting of all edges with an endpoint labeled $x$ and the vertices to which these edges are incident is denoted $\Lambda_x$. We think of $\Lambda_x$ as a graph in the disk
bounded by $\Lambda$. A ghost edge of $\Lambda$ which is incident to a vertex
of $\Lambda$ with label $x$ is called a {\em ghost $x$--edge} for $\Lambda$.
By the Parity Rule, each ghost $x$--edge has at most one endpoint at a fat
vertex of $\Lambda$ with label $x$. We refer to this endpoint as a 
{\em ghost $x$--label}.
Let $\alpha_x$ denote the number of ghost $x$--labels for $\Lambda$. 
\end{defn}

Note that, taken over all labels, $\sum \alpha_x \leq t+2g-2$.                     

\begin{defn}\label{def:calL}
A bigon (trigon) face of $\Lambda_x$ is referred to as an {\em $x$--bigon}
({\em $x$--trigon}, resp.) of $\Lambda$.
Let $\calL$ be the set of labels $x$ of $\Lambda$
for which $\Lambda$ has 
an $x$--bigon or $x$--trigon.  By $-\calL$ we denote the complement of $\calL$ in the set of all labels of $G_Q$.
\end{defn}

\begin{lemma}\label{lem:q2xinL0}
If $q \geq 2, t > 2g-2$, and $\alpha_x \leq 3$, then $x \in \calL$.
\end{lemma}

\begin{proof}
Assume that $q \geq 2, t > 2g-2$ and  $\alpha_x \leq 3$. Assume for 
contradiction that $x \notin \calL$.
Take a connected component, $\Lambda_x'$, of $\Lambda_x$ 
which is innermost in the disk $D$ bounded by $\Lambda$ and which has the least number of ghost $x$--labels among such components.
 Let $V,E,F$ be the number of vertices, edges, and faces of $\Lambda_x'$ in the disk $D$. 
Let $n$ be the
number of outside edges of $\Lambda_x'$, counted with multiplicity.
By assumption $\Lambda_x'$ can contain no bigon or trigon face, so
counting edges gives: $2E \geq 4F+n$. Note that $E>0,n>0$ by our assumptions, 
and $\Lambda$ has at least two vertices.
As $\alpha_x \leq 3$ and as each edge of $\Lambda$ has at most one label
$x$ (Parity Rule), we have
($\ast$) $E \geq q V-3$. Along with the Euler characteristic equation $V-E+F=1$,
these imply that $(E+3)/q - E + (2E-n)/4 \geq 1$. That is, 
$(2-q)E \geq (n+4) q /2 - 6$. This implies that $q =2$ and $n \leq 2$.
If $\Lambda_x'$ has at most two ghost $x$--labels then $E \geq q V-2$. Using
this in place of ($\ast$), the calculations above show that 
$(2-q)E \geq (n+4) q/2 - 4$ -- a contradiction. Thus we have that 
$\Lambda_x'$ has exactly $3$ ghost $x$--labels. This implies that 
$\alpha_x=3$ and $\Lambda_x'$ contains all the vertices of $\Lambda$.
This along with the facts that $n \leq 2, q =2$ implies that $\Lambda_x'$
consists of a single edge connecting two vertices, and that one of these
vertices has two ghost $x$--labels. Thus $\Lambda$ consists of
two vertices connected by fewer than $t$ parallel edges ($G_Q$ contains 
no $1$--sided faces). As $q =2$, $\Lambda$ has more than $2t$ ghost labels.
Thus $2t < t+2g-2$ -- a contradiction.
\end{proof}

\begin{prop}\label{prop:webcount}
If $t > 2g-2$ and $q \geq 2$ then $|\calL| \geq (3/4) t - (g-1)/2$.
\end{prop}

\begin{proof}
Assume $t> 2g-2$ and $x$ is a label which is not in $\calL$.  Then by 
Lemma~\ref{lem:q2xinL0}, $\alpha_x \geq 4$. Hence

\[ t+2g-2 \geq \sum_x \alpha_x = \sum_{x \in \calL} \alpha_x + \sum_{x \not\in\calL} \alpha_x \geq \sum_{x \not\in\calL} \alpha_x \geq 4 | -\calL| = 4(t-|\calL|)\]
 Hence $|\calL| \geq t-(t+2g-2)/4 = (3/4) t - (g-1)/2$.  
\end{proof}

\subsection{Bigons and Trigons of  $\Lambda_x$}
\begin{defn}
A {\em Scharlemann cycle (of length $n$)} is a disk face of $G_Q$ or $G_F$
with $n$ edges, all edges having the same pair of labels and all connecting 
parallel vertices of the graph. We use the same term for the set of edges
defining the face. The Scharlemann cycles considered in this paper
are typically on $G_Q$ and of length $2$ or $3$.
\end{defn}

\begin{remark}
Note that by the Parity Rule, any edge of the great $g$--web $\Lambda$ must have 
different labels in $G_Q$.
\end{remark}

\begin{defn}
The subgraph of $\Lambda$ bounded by a bigon face in $\Lambda_x$ contains a Scharlemann cycle of length $2$.  This subgraph is called an {\em extended Scharlemann cycle of length $2$} and the Scharlemann cycle it contains is referred to as its {\em core} Scharlemann cycle.
\end{defn}

\begin{defn}\label{def:trigonarms}
A trigon in $\Lambda_x$ is a {\em cycle} trigon if its 
edges can be oriented consistently around the trigon so that the $x$--label 
is always at the tail end of an edge (i.e.\ it forms an $x$--cycle). 
See Figure~\ref{fig:xcycletrigon}. Otherwise
it is a {\em non-cycle} trigon. See Figure~\ref{fig:noncycletrigon}.
\end{defn} 

\begin{defn}\label{def:trigonarms2} If $\sigma$ is a trigon face of $\Lambda_x$ then the subgraph
of $\Lambda$ bounded by $\sigma$ consists of a trigon face of $\Lambda$
together with three {\em arms}, each consisting of a (possibly empty) string
of bigon faces of $\Lambda$; see for example Figure~\ref{fig:xcycletrigon}.
\end{defn} 

\begin{defn}
The subgraph of $\Lambda$ bounded by a trigon face, $\sigma$, in $\Lambda_x$, either contains a single Scharlemann cycle of length $3$ or contains one or two Scharlemann cycles of length $2$ (e.g. see the proof of Lemma~\ref{lem:existsS22}). 
If it contains a Scharlemann cycle of length $3$, this subgraph is called an {\em extended Scharlemann cycle of length $3$} and the Scharlemann cycle it contains is referred to as its {\em core} Scharlemann cycle.
\end{defn}

\begin{remark}
Trigons of $\Lambda_x$ which are not extended Scharlemann cycles will be
classified later as Type I, I\!I, or I\!I\!I (see 
Definition~\ref{def:trigontypes}).
\end{remark}

\begin{defn}
The arc on the boundary of a vertex between two consecutive edges of a face of a subgraph of $G_Q$ or $G_F$ is a {\em corner} of that face.
\end{defn}

\begin{defn}
 If a set of edges $\sigma$ is the boundary of a face $f$ of $\Lambda_x$ then let $L(\sigma)$ be the set of labels of $\Lambda$ that appear on the corners of $f$.
\end{defn}

\begin{defn}
Let $\sigma$ be an extended Scharlemann cycle of $G_Q$.  Let $\Gamma_\sigma$ be the subgraph of $G_F$ consisting of the outermost (in $G_Q$) edges of $\sigma$ and the vertices to which these edges are incident.  We say {\em $\sigma$ lies in a disk} if $\Gamma_\sigma$ is contained in a disk in $\widehat{F}$.  We say {\em $\sigma$ lies in an essential annulus} if $\Gamma_\sigma$ is contained in an annulus in $\widehat{F}$ but does not lie in a disk. 
\end{defn}

\begin{lemma}\label{lem:doesnotlieindisk}
No extended Scharlemann cycle of $G_Q$ lies in a disk.
\end{lemma}
\begin{proof}
Assume $\sigma$ is an extended Scharlemann cycle of length $p$ that lies in a disk.  Note that $p>1$.
Let $f$ be the face of $\sigma$.  Let $D$ be a small disk in $\widehat{F}$ in which $\Gamma_{\sigma}$ lies (above).  
Consider the family $\mathcal{F}(\sigma)$ of extended Scharlemann cycles 
contained in $f$ with the same core Scharlemann cycle as $\sigma$ (excluding 
$\sigma$ but including the core Scharlemann cycle). If any of these lie in 
$D$ then we choose an innermost such, say $\sigma'$, and replace $\sigma$ by 
$\sigma'$. So we may assume that no element of $\mathcal{F}(\sigma)$ lies in 
$D$.  By disk exchanges we may assume $\Int f \cap D = \emptyset$.  Let $H$ be the $1$--handle neighborhood of the arc of $K$ that forms the corners of $f$.  
Then $H$ meets $D$ only in the two vertices corresponding to the labels of 
the edges of $\sigma$, and $\nbhd(D \cup f \cup H)$ is a punctured lens space of order $p$.  This cannot happen since $M$ is neither reducible nor a lens space.
\end{proof}

\begin{lemma}~\label{lem:AEntscc}
No simple closed curve of $Q \cap F$ that is 
trivial in $Q$ is trivial in $\hatF$.
\end{lemma}
\begin{proof}
Otherwise let ${\widehat D} \subset \hatF$ be the disk bounded by such a 
simple closed curve. Let $G_D$ be $G_F$ restricted to $\widehat D$. 
By Lemma~\ref{lem:FandQ}(3), $G_D$ is non-empty. By Lemma~\ref{lem:FandQ}(3),
there are no $1$--sided faces in $G_D$ and no $1$--sided faces in the
subgraph of $G_Q$ corresponding to the edges of $G_D$. The argument of 
Proposition~2.5.6 of \cite{cgls:dsok}, 
along with the assumption that $q \geq 2$,
implies that one of $G_D$ or $G_Q$ contains a Scharlemann cycle. Such
a Scharlemann cycle along with the argument of Lemma~\ref{lem:doesnotlieindisk}
above, would imply the contradiction that either $S^3$ or
$M$ contains a lens space summand. 
\end{proof}

Lemma~\ref{lem:FandQ}(2) along with Lemma~\ref{lem:AEntscc} imply the 
following.

\begin{cor}~\label{cor:AEntscc}
Any simple closed curve of $F \cap Q$ that lies in a disk face of $Q$ 
is an essential curve in $\hatF$.
\end{cor}


\section{Collections of curves with Property $P(k)$ and  $F_k(g)$}\label{section:Fkg}

\begin{defn}\label{defn3.1}
Let $\C$ be a collection of simple loops on a surface $S$, and let $k$ 
be a non-negative integer. 
We say that $\C$ has {\em Property $P(k)$} if 
\begin{itemize}
\item[(1)] the elements of $\C$ are essential and pairwise non-isotopic on $S$;
\item[(2)] any pair of elements of $\C$ meet transversely; 
\item[(3)] for all $c,c' \in \C$, $|c\cap c'| \le1$; and 
\item[(4)] any $c\in \C$ meets at most $k$ elements of $\C - \{c\}$.
\end{itemize}
For $k\ge 0$ and $g\ge2$ define $F_k(g) = \max\{|\C| : \C$ is a collection of 
simple loops on a closed, connected, orientable surface of genus~$g$ with 
Property~$P(k)\}$. 
\end{defn}

\begin{lemma}\label{lem:Fkg}
$F_k(g) \le A_k (g-2) + B_k$, where $A_k = k+2 \lfloor k/4 \rfloor +3$ and $B_k$ is:
\begin{align*}
3\ ,\quad&\text{if }\ k=0\ ;\\
5\ ,\quad&\text{if }\ k=1\ ;\\
7\ ,\quad&\text{if }\  k=2\text{ or } 3\ ;\\
9\ ,\quad&\text{if }\  k=4\ ;\\
10\ ,\quad&\text{if }\  k=5\ ;\\
12\ ,\quad&\text{if }\  k\ge 6\ .\\
\end{align*}

Moreover, we have equality if $g=2$ or $k=0,1\text{ or }\ 2$, i.e. 
\begin{align*}
F_k(2) &= B_k\ \text{ for all }\ k, \text{and } \\
F_0 (g) &= 3g-3\ ,\\
F_1(g) &= 4g-3\ ,\\
F_2(g) &=5g-3\ .
\end{align*} 
\end{lemma}

\begin{remark}
In the present paper we will only make use of the formula for $F_2(g)$ 
(and $F_0(g)$). However, we include the general case as possibly being of 
independent interest.
\end{remark}

\begin{proof}
We will prove the lemma by induction on $g$.
Since the case $k=0$ is easy and well-known, we assume $k\ge1$.
Let $\C$ be a collection of simple loops with Property~$P(k)$ on a closed, 
connected, orientable surface $S$ of genus~$g \ge 2$.

First we consider the special case where there is a loop $c_0\in \C$ 
that bounds a once-punctured torus $T_0\subset S$ containing a curve in $\C$
that intersects $k$ other elements of $\C$.
Let $S_0$ be the other component of $S$ cut along $c_0$, and let $S'$ 
be the closed surface of genus~$g-1$ obtained by capping off the boundary of 
$S_0$ with a disk $D_0$. 
Note that $c_0$ is disjoint from all the other members of $\C$.
It follows that either $k=1$ and there are exactly two elements of $\C$ 
in $\text{Int }T_0$, or $k=2$ and there are exactly three elements  of $\C$ 
in $\text{Int }T_0$.
Discarding from $\C$ these elements, together with $c_0$, we get a collection
$\C_0$ of loops on $S_0$.
Regarded as a collection of loops on $S'$, $\C_0$ satisfies all the 
conditions for Property~$P(k)$ except possibly (1): 
some pairs of loops in $\C_0$ may become isotopic on $S'$.
(Note that no loop in $\C_0$ is inessential in $S'$, as such a loop 
would be parallel to $c_0$ in $S$.)

Let $a,a'\in \C_0$ be curves that are isotopic on $S'$.
Then $a$ and $a'$ are disjoint, and cobound an annulus $A$ in $S'$. 
Since $a$ and $a'$ are not isotopic on $S$, we have $D_0\subset \Int A$.
In particular, it follows that no triple of curves in $\C_0$ is isotopic 
in $S'$.
Now suppose $b,b'\in\C_0$ is another pair of curves that are isotopic in $S'$.
Then $b$ and $b'$ cobound an annulus $B$ in $S'$ with $D_0\subset\Int B$.
Hence $A\cap B\ne \emptyset$. 
If $\partial A\cap\partial B = \emptyset$, then either $a$ or $a'$ lies 
in $B$, or $b$ or $b'$ lies in $A$.
But then $a$ or $a'$ would be isotopic to $b$ or $b'$ in $S$, a contradiction.
Hence, without loss of generality, $a\cap b\ne\emptyset$.
Therefore $b\cap A$ is a single transverse arc in $A$.
Since $b$ and $b'$ are isotopic in $S'$, we must also have $a\cap b'\ne
\emptyset$, and so $b'\cap A$ is also a single transverse arc.
Thus the pair $b,b'$ contributes two points of intersection to each 
of $a$ and $a'$.

It follows that if $k=1$ then at most one pair of curves in $\C_0$ are 
isotopic in $S'$.
We thus get a collection $\C'$ of loops on $S'$ with Property~$P(1)$ such that 
$|\C'| \ge |\C_0| -1$.
Therefore $|\C| \le |\C_0| +3 \le |\C'| +4$.
By induction $|\C'| \le 4(g-1) -3$, and hence $|\C| \le 4g-3$ as claimed.

If $k=2$ and at most one pair of curves in $\C_0$ are isotopic in $S'$ 
then the result follows easily by induction.
So suppose that two pairs of curves $a,a'$ and $b,b'$ in $\C_0$ become 
isotopic in $S'$.
Let $A,B$ be the annuli cobounded by $a,a'$ and $b,b'$ respectively, 
extended slightly so that $a\subset \Int A$ and $b\subset \Int B$.
Then $A\cap B$ is a disk containing $D_0$, and hence $A\cup B$ is a 
once-punctured torus. 
There is a simple loop $c$ in $A\cup B$ intersecting each of $a$ and $b$ 
transversely in a single point. 
Then $\C' = (\C - \{a',b'\})\cup \{c\}$ is a collection of loops on $S'$ 
with Property~$P(2)$.
Hence, using the inductive hypothesis, 
\begin{align*}
|\C| & \le |\C_0| +4\\
& \le \big(|\C'| +2-1\big) +4 = |\C'| +5\\
& \le \big(5(g-1)-3\big) +5 = 5g-3\ .
\end{align*}

We now consider the general case, and assume that $\C$ contains no loop 
$c_0$ as above. 
Since $A_k(g-2) +B_k$ is non-decreasing in $k$, 
we may assume that $\C$ does not have Property~$P(k-1)$, i.e.\ 
there is a loop $c_0\in \C$ that intersects $k$ other elements 
$c_1,c_2,\ldots,c_k\in \C$. 
In particular, $c_0$ is non-separating on $S$, and so cutting $S$ along 
$c_0$ gives a connected surface $S_0$.
Let $S' = S_0 \cup D_1 \cup D_2$ be the closed surface of genus~$(g-1)$ 
obtained by capping off the two boundary components of $S_0$ with disks 
$D_1$ and $D_2$. 
Let $S_1 = S_0 \cup D_1$.

Let $\C_1= \C- \{c_0,c_1,\ldots,c_k\}$. 
As a collection of loops on $S_1$, $\C_1$ satisfies Property~$P(k)$ except 
that some pairs of loops in $\C_1$ may become isotopic in $S_1$.
(Again, no loop in $\C_1$ is trivial in $S_1$ as such a loop would be 
parallel to $c_0$ in $S$.)
Let $a,a'$ be such a pair. 
The previous discussion applies verbatim, with $\C_0$ replaced by $\C_1$ 
and $S'$ replaced by $S_1$, so any other such pair $b,b'$ contributes two 
points of intersection to each of $a$ and $a'$.

At least $\left\lfloor \frac{k+1}2\right\rfloor$ of the curves $c_1,c_2,\ldots,c_k$ 
must intersect one of $a,a'$.
Hence the number of pairs of elements of $\C_1$ that are isotopic in 
$S_1$ is at most 
$1+\left\lfloor\frac{k-\lfloor\frac{k+1}2\rfloor}2\right\rfloor = 1+\left\lfloor\frac{k}4 \right\rfloor$.
Discarding one element of $\C_1$ from each such pair, we get a collection 
of loops $\C_2$ with  Property~$P(k)$ in $S_1$.
Now apply the same argument to $\C_2$, regarded as a collection of loops 
in $S' = S_1\cup D_2$. 
We may assume that
no element of $\C_2$ is inessential in $S'$, as such a curve $c$ 
would bound a disk $D$ in $S'$ with $D_1\cup D_2\subset \Int D$, 
and would therefore bound a once-punctured torus in $S$ containing $c_0$, 
putting us in the special case treated earlier. 
We then get a collection $\C'$ of loops in $S'$ satisfying 
Property~$P(k)$, with $|\C'| \ge |\C_1| - 2(1+ \lfloor k/4 \rfloor)$. 
Hence
\begin{align*}
|\C| & = |\C_1| + (k+1)\\
&\le |\C'| + (k+1) + 2(1+\lfloor k/4 \rfloor)\\
& = |\C'| + A_k\ .
\end{align*}
Now assume $g\ge 3$ and that the theorem holds for $(g-1)$.
Then 
\[|\C| \le (A_k(g-3) + B_k) + A_k = A_k (g-2) + B_k\ ,\]
and the result follows by induction on $g$.

To start the induction we will show that $F_k(2) = B_k$. 
Let $S$ be a closed, orientable surface of genus~2 and let $\C$ be  a 
collection of simple loops on $S$ with Property~$P(k)$. 
Let $\tau \colon S\to S$ be the hyperelliptic involution. 
The quotient $(S,\text{Fix}(\tau))/\tau \cong (S^2,V)$ where $V$ 
consists of six points.  
Let $\pi \colon S\to S^2$ be the quotient map. 
The elements of $\C$ can be isotoped so that $\tau(c)=c$ for all $c\in \C$.
If $c$ meets $\text{Fix}(\tau)$ then $\pi (c)$ is an arc in $S^2$ whose 
endpoints are distinct points in $V$; if $c\,\cap\,\text{Fix}(\tau)=\emptyset$
then $\pi (c)$ is a simple loop in $S^2 -V$.
By part~(3) of Definition~\ref{defn3.1}, if $c,c'$ are distinct elements 
of $\C$ then $\pi (c) \cap \pi (c')$ is either empty or a single point in $V$.
The arcs $\pi (c)$ form the edges of a graph $\Gamma$ in $S^2$ with 
vertex set $V$, and $\pi (\cup\,\C)$ is the disjoint union 
$\Gamma\smallcoprod\Lambda$ 
where the components of $\Lambda$ are simple loops.
Hence $|\C| = E(\Gamma) + |\Lambda|$, where $E(\Gamma)$ is the number of 
edges of $\Gamma$. 
Note that no two edges of $\Gamma$  share the same pair of endpoints, no 
component of $\Lambda$ bounds a disk in $S^2-V$, and no two components 
of $\Lambda$ are parallel in $S^2-V$. 
Finally, by part~(4) of Definition~\ref{defn3.1}, the sum of the valencies
of the vertices at the endpoints of each edge of $\Gamma$ is at most $k+2$.

It is straightforward to check that with these constraints, for $0\le k\le 5$
the maximum value of $E(\Gamma)+|\Lambda|$ is the $B_k$ given in the 
statement of the lemma. 
These values are realized by the configurations shown in Figure~\ref{fig:Curves2}.
Note that by passing to the 2-fold branched covering of $(S^2,V)$ 
we obtain a collection $\C$ of simple loops on $S$ with Property~$P(k)$.
For $k=6$ we have the 1-skeleton of the octahedron shown in Figure~\ref{fig:Curves2},
with 12 edges. 
An easy Euler characteristic argument shows that a graph $\Gamma$ in $S^2$ 
with six vertices, no loop edges and no parallel edges, has at most 12 edges.
It follows that $F_k(2) =12$ for $k\ge6$.

\begin{figure}
\centering
\input{Curves2.pstex_t}
\caption{}
\label{fig:Curves2}
\end{figure}

Finally we show that the inequality in the lemma is an equality when 
$k=0,1$ or 2. 
For $k=0$, this is easy and well known. 
For $k=1$ or 2, let $S$ be the closed surface of genus $g$ obtained by 
attaching $g$ once-punctured tori $T_i$, $1\le i\le g$, to a 
$g$--punctured sphere $P$ along their boundaries. 
In each $T_i$ take a collection $\C_i$ of 2 (resp.\ 3) simple loops with 
pairwise intersection numbers equal to 1, and in $P$ take a collection 
$\C_0$ of simple loops that cut $P$ into pairs of pants. 
Define $\C = (\bigcup_{i=1}^g \C_i) \cup\C_0 \cup\{\text{components of }
\partial P\}$, a collection of loops on $S$.
Then $\C$ satisfies Property~$P(1)$ (resp.\ $P(2)$) and 
$|\C| = 2g + (g-3) +g = 4g-3$ (resp.\ $3g +(g-3)+g=5g-3$).
\end{proof}

\begin{remark}
Consider the number $F(g) = \max_k\{F_k(g)\}$. \cite{MRT} shows $g^2+g \leq F(g) \leq (g-1)2^{2g}$ for $g\geq 2$.  Moreover, by essentially the same method as the base case of our induction, \cite{MRT} proves $F(2) = 12$.
\end{remark}

\section{Extended Scharlemann cycles of length $2$ and $3$.}\label{sec:extSC}

\begin{defn}
Let $\{a,b\}$ be labels of $G_Q$, then $\arc{ab}$ is one of the two 
intervals of labels running between $a$ and $b$ on an abstract vertex of
$G_Q$ that contains no labels $a,b$ on its interior. Typically, one such
interval is specified by the context. 
\end{defn}

If $\sigma$ is an extended Scharlemann cycle whose edges have
labels $\{a,b\}$ then we may think of $L(\sigma)$ as an interval
of labels $\arc{ab}$.

\begin{lemma}\label{lem:commonlabelsets}
Let $\sigma,\tau$ be extended Scharlemann cycles of length $2$ or $3$. Then either
\begin{enumerate}
\item $L(\sigma) \cup L(\tau)$ contains all labels;
\item $L(\sigma) \subset L(\tau)$ or $L(\tau) \subset L(\sigma)$;
\item $L(\sigma) \cap L(\tau)$ is a single interval of labels $\arc{xy}$
where $x$ is an extremal label of $\sigma$ and $y$ is an extremal label of
$\tau$; or
\item $L(\sigma) \cap L(\tau) = \emptyset$.
\end{enumerate}
\end{lemma}
\begin{proof}
$L(\sigma),L(\tau)$ are each a single interval of labels. By convexity, either
$L(\sigma) \cup L(\tau)$ contains all labels or $L(\sigma) \cap L(\tau)$
is convex. If nonempty, it is a single interval of intersection, $\arc{xy}$.
Clearly, $x$ and $y$ must be extremal in either $\sigma$ or $\tau$. If both are
extremal in, say, $\sigma$, then $L(\sigma)$ is contained in $L(\tau)$.
\end{proof}

\begin{defn}\label{def:longmobiusband}
Let $\sigma$ be an extended Scharlemann cycle of length $2$ with $L(\sigma) = \arc{xy}$.  Let $f$ be the union of the bigons bounded by $\sigma$ in $G_Q$.  
Construct the {\em long \mobius band} $A(\sigma)$ in $M$ by extending the corners of $f$ radially to $K$ in a neighborhood of $K$ in $M$.  Let $a(\sigma)$ be the set of simple closed curves $A(\sigma) \cap \widehat{F}$ formed by pairs of edges of $\sigma$. The {\em core labels} of $\sigma$ are the two labels of the Scharlemann cycle face of $\sigma$. The {\em core curve} of $\sigma$ is the element of $a(\sigma)$ on its core labels.  
\end{defn}

\begin{defn}
An {\em embedded $\theta$--curve} in a surface $\widehat{F}$ is a graph $\theta$ consisting of two vertices and three edges, between these two vertices, embedded in $\widehat{F}$ so that $\nbhd(\theta)$ is a thrice punctured sphere.  A {\em $\theta$--band} is homeomorphic to the product of an embedded $\theta$--curve with an interval.
\end{defn}

\begin{lemma}
Let $\sigma$ be a Scharlemann cycle of length $3$ in $G_Q$.  Then, regarding as points the two vertices of $G_F$ to which the edges of $\sigma$ are incident, the edges of $\sigma$ form an embedded $\theta$--curve in $\widehat{F}$.
\end{lemma}

\begin{proof}
 Let $f$ be the trigon of $G_Q$ bounded by $\sigma$.    Let $\{x, x+1\}$ be the label pair on the edges of $\sigma$.  With $\alpha$, $\beta$, and $\gamma$ label the corners of $f$ and hence label the endpoints of the edges of $\sigma$ too.  These labels then appear around vertices $x$ and $x+1$ of $G_F$ in opposite directions.  Since each edge of $\sigma$ connects two distinct corners of $f$, the endpoints of each edge have distinct labels.  Shrinking vertices $x$ and $x+1$ to points, this forces $\sigma$ to form an embedded $\theta$--curve in $\widehat{F}$.
\end{proof}

\begin{defn}\label{def:twistedthetaband}
Let $\sigma$ be an extended Scharlemann cycle of length $3$ with $L(\sigma) = \arc{xy}$.  Let $f$ be the union of the bigons and the trigon bounded by $\sigma$ in $G_Q$. Construct the {\em long twisted $\theta$--band} $\Theta(\sigma)$ in $M$ by extending the corners of $f$ radially to the core of $H_{xy}$.  Let $\theta(\sigma)$ be the set of embedded $\theta$--curves $\Theta(\sigma) \cap \widehat{F}$.  The {\em core labels} of $\sigma$ are the two labels of the Scharlemann cycle face of $\sigma$.  The {\em core $\theta$--curve} of $\sigma$ is the element of $\theta(\sigma)$ on its core labels.
\end{defn}

The following shorthand will be convenient.

\begin{defn}
If $\sigma$ is an extended Scharlemann cycle of length $3$, then take $A(\sigma)=\emptyset$ and $a(\sigma)=\emptyset$.
If $\sigma$ is an extended Scharlemann cycle of length $2$, then take $\Theta(\sigma)=\emptyset$ and $\theta(\sigma)=\emptyset$. If $\sigma$ is an extended
Scharlemann cycle of length $2$ or $3$, let $\ao(\sigma)=a(\sigma) \cup 
\theta(\sigma)$ and $\AO(\sigma)=A(\sigma) \cup \Theta(\sigma)$.
\end{defn}

\begin{defn}
Let ${\mathcal E} = \{ L(\sigma) | \sigma $ is an extended Scharlemann
cycle of $\Lambda$ of length $2$ or $3\}$ be the set of all labels that lie 
on the corners of extended 
Scharlemann cycles of length $2$ or $3$ in $\Lambda$.  
\end{defn}

\begin{defn}\label{def:Sigma}
Let $\Sigma$ be a collection of extended Scharlemann cycles of length
$2$ or $3$ in $\Lambda$ such that $L(\Sigma)=\cup \{ L(\sigma) | \sigma \in \Sigma \}$
contains $\mathcal E$.  
Let $\ao(\Sigma) = \bigcup_{\sigma \in \Sigma} \ao(\sigma)$.  
Assume $\Sigma$ is chosen so that the complexity 
$(|\Sigma|,|\ao(\Sigma)|\}$ is minimal in the lexicographic ordering among all such collections.
\end{defn}

\begin{lemma}\label{lem:commonlabelsetsrefined}
If $\sigma, \tau \in \Sigma$ then either 
\begin{enumerate}
\item $L(\sigma) \cup L(\tau)$ contains all labels;
\item $L(\sigma) \cap L(\tau)$ is a single interval of labels $\arc{xy}$
where $x$ is an extremal label of $\sigma$ and $y$ is an extremal label of
$\tau$; or 
\item $L(\sigma) \cap L(\tau) = \emptyset$.
\end{enumerate}
\end{lemma}
\begin{proof}
If none of these were to occur, then by Lemma~\ref{lem:commonlabelsets} the only other possibility is that either $L(\sigma) \subset L(\tau)$ or $L(\tau) \subset L(\sigma)$.  Therefore $\sigma$ or $\tau$ respectively may be dropped from $\Sigma$ to produce $\Sigma'$ so that $L(\Sigma')$ contains $\mathcal{E}$ and yet has lesser complexity than $\Sigma$.  This contradicts the minimality assumption on $\Sigma$.
\end{proof}

\begin{defn}
The extended Scharlemann cycle $\sigma'$ is obtained by {\em paring down} the extended Scharlemann cycle $\sigma$ if $\sigma'$ is contained in the disk face that $\sigma$ bounds.
\end{defn}

\begin{lemma}\label{lem:atmostoneintersection}
Two elements of $\ao(\Sigma)$ intersect at most once. Furthermore, each
vertex of $G_F$ belongs to at most two different elements
of $\ao(\Sigma)$.
\end{lemma}
\begin{proof}
If there exists $\sigma, \tau \in \Sigma$ with $c_\sigma \in \ao(\sigma)$ and $c_\tau \in \ao(\tau)$ intersecting twice, then $c_\sigma$ and $c_\tau$ have the same label pair.  Therefore we may pare down either $\sigma$ or $\tau$ (or completely eliminate one) to produce $\Sigma'$ so that $L(\Sigma')$ contains $\mathcal{E}$ and yet has lesser complexity than $\Sigma$.  This contradicts the minimality assumption on $\Sigma$.

If label $x$ of $G_Q$ were in $L(\sigma_1),L(\sigma_2)$ and 
$L(\sigma_3)$ with $\sigma_1,\sigma_2, \sigma_3 \in \Sigma$, then one 
of these label sets must be contained in the union of the other two.
This contradicts the minimality of $\Sigma$ thereby proving the final
statement of the lemma.
\end{proof}

\begin{lemma}\label{lem:gtsolidtorus}
Let $\sigma$ be an extended Scharlemann cycle of length $3$ that lies in an essential annulus $A$ of $\hatF$.  Let $f$ be the face bounded by $\sigma$ and assume $\Int f \cap A = \emptyset$.  Then $N=\nbhd(A \cup \Theta(\sigma))$ is a solid torus and the core of $A$ runs $3$ times in the longitudinal direction of $N$.
\end{lemma}
\begin{proof}
Let $H$ be the $1$--handle neighborhood of the arc of $K$ that forms the corners of $f$.  Then $N=\nbhd(A \cup \Theta(\sigma))$ may be obtained from attaching the $2$--handle $\nbhd(f)$ to the genus $2$ handlebody $\nbhd(A) \cup H$.  Because a meridional disk of $\nbhd(A)$ intersects $\bdry f$ once, $N$ is a solid torus.  One may then observe that the core of $A$ runs $3$ times in the longitudinal direction of $N$.  (Cf.\ Lemma~2.1 \cite{gt:dsokwylsagok}; Lemma~3.7 \cite{gl:dsokcetI}).
\end{proof}

\begin{lemma}\label{lem:essentialcurves}
Let $\sigma$ be an extended Scharlemann cycle of length $2$ or $3$ in $\Lambda$.
Assume an element of $\ao(\sigma)$ lies in an annulus $A$ of $\hatF$. Let
$\eta$ be the core of $A$. Then $\eta$ cannot bound a disk in $M$ whose 
framing on $\nbhd(\eta)$ is the same as that of $\hatF$. In particular, $\eta$
is neither trivial on $\hatF$ nor a meridian curve on either side of $\hatF$.
\end{lemma}

\begin{proof}
Assume some element, $c$, of $\ao(\sigma)$ lies in an 
annulus $A$ in $\hatF$
whose core is $\eta$. Assume for contradiction that there is a disk $D$
in $Y=M - \nbhd(\eta)$ whose boundary has the same slope on 
$\partial \nbhd(\eta)$
as $\hatF$. Pare down $\sigma$ so that $c$ is composed of the outermost
edges of $\sigma$. The corners of $\sigma$ belong to the same sub-arc, $k$, of
$K$. Let $f$ be the bigon or trigon face bounded by $\sigma$. 
Then $\AO(\sigma)$ (whose boundary is $c$) is $f$ radially contracted 
to $k$. We view both $\AO(\sigma)$ and $D$ as properly embedded in 
$Y$. By the framing assumption on $D$, we may take the boundaries of these
two surfaces to be disjoint. Isotop $k$ rel its endpoints on $\partial Y$
so that $k$ intersects $D$ minimally. Then any arc components of intersection
between $f$ and $D$ must be parallel on $f$ to the outermost edges of $\sigma$.
We may surger away any simple closed curves of intersection. If $k$ is disjoint
from $D$, then a neighborhood of $D \cup \nbhd(\eta) \cup \AO(\sigma)$ 
will be a lens space summand of $M$ (using Lemma~\ref{lem:gtsolidtorus}
when the length of $\sigma$ is $3$), a contradiction. If $k$ intersects
$D$, then the intersection between $D$ and $f$ is a collection of
arcs parallel to $\partial f$ with the same number of endpoints on each
corner of $f$. An innermost face of this graph in $f$ will give a length $2$ or
$3$ Scharlemann cycle with respect to $D$ which we can use, as in 
Lemma~\ref{lem:doesnotlieindisk}, to construct a lens space summand in
$M$.
\end{proof}

\subsection{Counting with extended Scharlemann cycles.}

\begin{lemma}\label{lem:parallelthetas}
Let $\calT$ be a set of embedded $\theta$--curves in $\hatF$ such that
\begin{enumerate}
\item no $\theta \in \calT$ lies in a disk in $\hatF$;
\item if $\theta,\theta' \in \calT$ then $\theta \cap \theta'$ is either
empty or a single vertex;
\item any vertex belongs to at most two elements of $\calT$.
\end{enumerate}
If $|\calT|>3F_2(g)$ then some $\theta \in \calT$ lies in an essential
annulus in $\hatF$.
\end{lemma}

\begin{proof}
By (1), for each $\theta \in \calT$ we may choose an essential circle $c_\theta$
in $\theta$ by deleting the interior of one of the edges of $\theta$. Let
$\calC=\{c_\theta|\theta \in \calT\}$. If $\calC' \subset \calC$ is a 
subcollection such that no two elements of $\calC'$ are isotopic in $\hatF$
then (after isotoping the elements of $\calC'$ into general position)
$\calC'$ has Property $P(2)$ of Definition~\ref{defn3.1}. Hence, by 
Lemma~\ref{lem:Fkg}, if $|\calC|>3F_2(g)$ then four elements of the 
$\calC$ must be isotopic in $\hatF$. Any
two of these curves are either disjoint or
intersect in a single vertex non-transversely. In either case we think of
two such curves as cobounding an (essential) annulus in $\hatF$. One sees 
that one of
these four curves must contain a vertex that lies strictly in the interior
of the annulus cobounded by two others. The corresponding $\theta$--curve
then lies within this annulus.
\end{proof}

\section{Thinning with extended Scharlemann cycles.}\label{sec:thinning}

\begin{defn}
Let $\sigma$ be an extended Scharlemann cycle of length $2$ with $|a(\sigma)|=n$.  Then we may order the elements of $a(\sigma)$ as $a_1, \dots, a_n$ such that $a_1$ is the core curve and $a_i \cup a_{i-1}$ bounds an annulus $A_i \subset A(\sigma)$ that is disjoint from the other elements of $a(\sigma)$ for $i=2, \dots, n$.  The core Scharlemann cycle in $\sigma$ forms the core curve $a_1$ which bounds the \mobius band $A_1$.
\end{defn}

\begin{defn}\label{def:thetaj}
Let $\sigma$ be an extended Scharlemann cycle of length $3$ with $|\theta(\sigma)|=n$.  Then we may order the elements of $\theta(\sigma)$ as $\theta_1, \dots, \theta_n$ such that $\theta_1$ is the core $\theta$--curve and $\theta_i \cup \theta_{i-1}$ bounds a product $\theta$--curve $\Theta_i \subset \Theta(\sigma)$ that is disjoint from the other elements of $\theta(\sigma)$ for $i=2, \dots, n$.  The core Scharlemann cycle in $\sigma$ forms the core $\theta$--curve which bounds the twisted $\theta$--band $\Theta_1$.
\end{defn}

\begin{lemma}\label{lem:annulusintheta}
Let $\sigma$ be an extended Scharlemann cycle of length $3$.  Assume $\theta_i, \theta_j \in \theta(\sigma)$, $i < j$,  lie in an essential annulus $B$ in $\widehat{F}$.  Exactly two of the edges, $e^i_1, e^i_2$ of $\theta_i$ cobound
a disk in $\hatF$. Let $e^j_1, e^j_2$ be the edges of $\theta_j$ 
that lie on the corresponding arms (Definition~\ref{def:trigonarms2})
of $\sigma$. Then $e^j_1, e^j_2$ cobound a disk on $\hatF$. 
\end{lemma} 
\begin{proof}
Since $\theta_i$ lies in the annulus $B$, two of its edges, $e^i_1, e^i_2$,
cobound a disk, $D$, in $B$, and (by Lemma~\ref{lem:essentialcurves}) the 
third creates a core curve of $B$ when joined
with either $e^i_1$ or $e^i_2$. For $k$ between (and including) $i$ and $j$,
let $\theta_k \in \theta(\sigma)$ be as in Definition~\ref{def:thetaj}. Let 
$e^k_1, e^k_2$ be the edges of $\theta_k$ such that $e^k_1$ lies on the same
arm of $\sigma$ as $e^i_1$, and $e^k_2$ on the same arm as $e^i_2$. Let 
$\gamma_k$ be the curve $e^k_1 \cup e^k_2$. Let $F_1,F_2$ be the union of the
faces on each arm between $e^i_1,e^j_1$ and $e^i_2,e^j_2$. Then $F_1 \cup F_2$
forms an annulus $A$ in $M$ whose boundary components are $\gamma_i$ and
$\gamma_j$.

By passing to a concentric subannulus of $B$, we may assume that $\partial B
= \gamma_i' \cup \gamma_j'$, where $\gamma_i',\gamma_j'$ are curves in 
$\theta_i,\theta_j$. Any component of $A \cap B$ is one of the curves
$\gamma_k$ for some $k$ between $i$ and $j$. The corresponding $\theta$--curve
$\theta_k$ is then contained in $B$, and does not lie in a disk by 
Lemma~\ref{lem:essentialcurves}. Also, $\gamma_k$ is either a core curve
of $B$ or bounds a disk in $B$. 

Assume for contradiction that $\gamma_j$ does not bound a disk in $B$.
Since $\gamma_i$ does bound a disk in $B$, there exist $k,k'$ between 
$i$ and $j$ such that $\gamma_k$ and $\gamma_{k'}$ are adjacent on $B$ among
the components of $A \cap B$, $\gamma_k$ bounds a disk in $B$ and $\gamma_{k'}$
is a core curve of $B$. By renumbering we may therefore assume that $k=i$
and $k'=j$; thus the annulus $A$ meets $B$ only in $\partial A$. Let $E$
be the disk $A \cup D$. By pushing $D$ slightly off $B$ in the appropriate
direction we may assume that $E \cap B = \partial E =\gamma_j$.

Let $l=min\{i,j\}$ and let $\sigma'$ be the extended Scharlemann cycle
within $\sigma$ that terminates at the edges of $\theta_l$. Then the interior
of the face bounded by $\sigma'$ is disjoint from $B$ and $E$. Let $N$
be a regular neighborhood in $M$ of $\Theta(\sigma') \cup B \cup E$. By 
Lemma~\ref{lem:gtsolidtorus}, $N$ is a punctured lens space of order $3$, a 
contradiction.
\end{proof}

The above lemma allows certain arguments for extended Scharlemann cycles of length $2$ to apply to extended Scharlemann cycles of length $3$.  To facilitate this we make the following definitions.

\begin{defn}
Assume $\sigma$ is an extended Scharlemann cycle of length $3$ with $\theta_i, \theta_j \in \theta(\sigma)$, $i < j$, lying in an essential annulus $B$ in $\widehat{F}$. Let $e^i_1,e^i_2$ be the edges
of $\theta_i$ that cobound a disk, $D_i$, in $B$ and $e^i_3$ be the third edge
of $\theta_i$. Let $e^j_1,e^j_2,e^j_3$ be the edges of $\theta_j$ on
the corresponding arms of $\sigma$. By Lemma~\ref{lem:annulusintheta},
$e^j_1,e^j_2$ cobound a disk $D_j$ in $B$. $D_i$ ($D_j$) is called the
{\em disk of parallelism} for $\theta_i$ ($\theta_j$, resp.). For $k=1,2,3$,
let $F_k$ be the faces along an arm of $\sigma$ 
between $e^i_k,e^j_k$. By Lemma~\ref{lem:annulusintheta}, $A_1=F_1 \cup F_3$
and $A_2=F_2 \cup F_3$ form isotopic annuli in $M$ whose boundaries lie
in $\theta_i, \theta_j$ and are isotopic to the core of $B$. $A_1,A_2$ are
called the {\em constituent annuli} of  $\Theta_i \cup \dots \cup \Theta_j$
of $\AO(\sigma)$ between $\theta_i,\theta_j$.
\end{defn}

\begin{lemma}\label{lem:SFS}
Let $\sigma$ be an extended Scharlemann cycle of length $m=2$ or $m=3$. 
Assume $a_i,a_j \in \ao(\sigma)$ , $i<j$,
lie on an annulus, $B$, in $\hatF$ such that all of the intersections of
$K$ with $B$ belong to $a_i,a_j$ (i.e. occur at the four vertices of 
$a_i,a_j$). 
Then either
\begin{enumerate}
\item $j=i+1$ and the annulus $A_{i+1}
\in \AO(\sigma)$ between $a_i,a_j$ (or constituent annulus when $\sigma$
has length $3$) is parallel into $\hatF$; or
\item $M$ is a Seifert fibered space over $S^2$ with three exceptional 
fibers, one of which has order $m$, and $K$ has bridge number $0$ with
respect to some genus $2$ Heegaard splitting of $M$. 
\end{enumerate}
\end{lemma}
\begin{addendum}
The conclusion above holds when $\sigma$ has length $3$ and $K$ intersects
$B$ only in $a_i,a_j$ or in the disk of parallelism of at most one of 
$a_i$ or $a_j$.
\end{addendum}
\begin{proof}

We assume $K$ intersects $B$ only at the four vertices of $a_i,a_j$.
Let $A = A_{i+1} \cup A_{i+2} \dots \cup A_j$ where 
$A_k \in \AO(\sigma)$ 
when $\sigma$ is of length $2$, and where $A$ is a   
constituent annulus of $\AO(\sigma)$ between $a_i,a_j$ when $\sigma$ is of 
length $3$. Let $T = A \cup B$.
Note that $A$ and $B$ only intersect in their boundaries (by Lemma~\ref{lem:essentialcurves} and the fact that any component of $\Int A \cap B$ must bound a 
disk in $A$), thus $T$ is an embedded torus in $M$ ($M$
contains no Klein bottles).
Since $M$ is assumed to be atoroidal, $T$ must be compressible in $M$. 
Let $D$ be a compressing disk. 

\medskip

{\em Addendum:} In the context of the Addendum, assume that $K$ 
intersects $B$ only in the four vertices of $a_i,a_j$ and in the disk of
parallelism of, say, $a_i$. Then we may choose the constituent annulus
of $\AO(\sigma)$ between $a_i,a_j$ and shrink $B$ so that $A$ and $B$
meet along their boundary and $B$ is disjoint from the interior of the
disk of parallelism of $a_j$. Then $B$ intersects $K$ only along $a_i,a_j$,
and we apply the arguments to this $A,B$.

\medskip

The proof of this lemma splits into three cases depending upon the relationship of $A_i$, $A_j$, and $D$ with respect to $\widehat{F}$.

{\bf Case (I):}  $A_{i+1}$ and $A_j$ lie on opposite sides of $\widehat{F}$.  

The following claim shows that Case I does not occur.

\begin{claim}
There is an isotopy of $M$ that reduces the width of $K$ with respect to
$\hatF$.
\end{claim}
\begin{proof}
The unfurling move from Section~4.3 of \cite{baker:sgkilshsbn} gives a 
homeomorphism of $M$ that reduces the width 
of $K$ and is isotopic to the identity. We describe this below. 

Let $V$ be a regular closed neighborhood of $T$ in $M$.  Let $\kappa, \kappa'$ be the spanning arcs $K \cap A = K \cap T$ on the annulus $A$.  
Then $V$ is an interval of tori, $T \times [0,1]$, where we 
coordinatize $T=S^1 \times
S^1$ so that $\partial A = S^1 \times p_1 \cup S^1 \times p_2$,  
$\kappa \subset q_1 \times S^1$, and $\kappa' \subset q_2 \times S^1$
for points $p_1,p_2,q_1,q_2$.  Now, up to isotopy rel-$\bdry$ in $V$, we may take $K \cap V$ to be the pair of arcs $\widehat{\kappa}, \widehat{\kappa'}$ which are $\kappa, \kappa'$ extended by product arcs in $V=T \times [0,1]$.  
%
%
%

  Let $h \colon V \to V$ be the  
homeomorphism that rotates $T \times i$ by $2\pi i$ in the factor ${*} \times
S^1$ and is the identity in the factor $S^1 \times {*}$. Under $h$, 
$\widehat{\kappa},\widehat{\kappa'}$ become isotopic to spanning 
arcs of $B$ extended by 
product arcs across $V$. As $h$ is the identity on $\partial V$, it extends
to $M$ as the identity outside $V$, and the image of $K$ under $h$
is strictly thinner with respect to $\hatF$ (the arcs above and below
$\hatF$ after the
homeomorphism are level-preserving isotopic to those of 
$K-(\kappa \cup \kappa')$). 
To finish the claim 
we must show that $h$ is isotopic to the identity on $M$. 

As above, let $D$ be a compressing disk for $T$ and 
$\calN = \nbhd(V \cup D)$. As $M$
is irreducible, the 2-sphere component, $S$, of $\partial \calN$ bounds 
a 3-ball in 
$M$ which either {\bf (a)} contains $\calN$ or {\bf (b)} meets it only 
along $S$. 

In case {\bf (a)}, $h$ is the identity on $M$ outside the 3-ball. So there
is an isotopy on the 3-ball, keeping $S$ fixed, taking $h$ restricted to the 
3-ball to the identity. This extends to an isotopy on $M$ taking $h$ to the
identity.

In case {\bf (b)}, note that $h$
as a homeomorphism of $V$ is isotopic to the identity by an isotopy that
keeps one boundary component of $V$ fixed and rotates the other component
of $\partial V$ by $2\pi$. Take such an isotopy that rotates the component
$T_1$ of $\partial V$ on the side containing $D$ and fixes the component, 
$T_2$, on the other side. $T_1$ bounds a solid torus, $N$, in $M$ 
that intersects $V$ in $T_1$. So we can  extend the isotopy of $h$ on 
$V$ across $N$. As this isotopy fixes $T_2$, we may extend it across all
of $M$, giving the desired isotopy of $h$ to the identity on $M$.
\end{proof}

{\bf Case (II):}  $A_{i+1}$ and $A_j$ lie on the same side of $\widehat{F}$ and $D$ near $B$ lies on the opposite side of $\widehat{F}$.

Let $N_2$ be the component of $M-T$ containing $D$.  By a slight isotopy supported in $\nbhd(A)$, push $K$ off of $T$ (this comes from looking at the labelling
around $a_i,a_j$ and the fact that $T$ is not a Klein bottle).   We may assume that $D$ is disjoint from $K$ since the exterior of $K$ is atoroidal.  We will need to handle the cases $m=2$ and $m=3$ separately.

{\bf Case (IIa):}  Assume $m=2$.
Then $A'=A_1 \cup \dots \cup A_i$ is an embedded \mobius band that intersects $T$ in $a_i$.  Now $K$ can be perturbed (avoiding $D$) so that it is disjoint from $T$ and intersects $A'$ in a single point. ($K$ is made disjoint from $T$ as above. To see it then intersects the \mobius band $A_1$ once, note that when we push $K$ off $A$ it starts off on one side of $A_i$ then ends up on the other.)  We may assume that $D$ intersects $A'$ in a nonempty collection of arcs which are essential in $A'$ (in fact parallel to $K \cap A'$ in $A'$ by working on the
intersections before perturbing $K$ off $A'$)  --- else we get a projective plane or can reduce
the number of intersections. Taking an outermost disk of $D - A'$ and compressing $A'$ along it constructs a compressing disk for $T$ in $N_2$ which intersects $K$ exactly once (from the intersection of $K$ with $A'$). But this contradicts that there is an essential compressing disk for $T$ which misses $K$ in $N_2$.

{\bf Case (IIb):} Assume $m=3$.

Then $\Theta' = \Theta_1 \cup \dots \cup \Theta_i \subset \AO(\sigma)$ is 
an embedded twisted $\theta$--band that intersects $T$ in 
$a_i \subset \theta_i$.
\medskip

{\em Addendum:} In the context of the Addendum in this case, 
we expand $B$ to include the disks of parallelism. If $K$ intersects the
interior of the disk of parallelism of, say, $a_i$ 
we surger this disk 
along the faces
of $\sigma$ connecting the disk of parallelism of $a_i$ with that of $a_j$
as in the proof of Lemma~\ref{lem:annulusintheta}. We take the resulting
$B$ for the argument below.

\medskip

Let $\kappa = K \cap \Theta'$. Then $D$ intersects $\Theta'$ away from 
$\kappa$. We consider $\Theta' - \nbhd(\kappa)$ as a large trigon, $f$,
with corners along $\kappa$ (indeed it is a trigon face in $\Lambda_x$ 
for some label $x$). Then we may assume $D$ intersects $f$ only in arcs
parallel to the corners of $f$. Furthermore, this collection of arcs is
non-empty -- else $D$ in union with $\Theta'$ along an annulus in $T$
gives rise to a punctured $L(3,1)$ in $M$ by Lemma~\ref{lem:gtsolidtorus}. 
The compression of $\Theta'$ along an outermost disk of $D-\Theta'$ contains a \mobius band $A'$ with $\bdry A' \subset T$ and $\kappa \subset A'$.  
Then we proceed as in Case (IIa) where $m=2$ to produce a compressing disk for $T$ in $N_2$ that intersects $K$ exactly once, contradicting that there is
an essential disk disjoint from $K$.

{\bf Case (III):}  $A_{i+1}$ and $A_j$ lie on the same side of $\widehat{F}$ and $D$ near $B$ lies on this side of $\widehat{F}$.  

Since $A_{i+1}$ and $A_j$ lie on the same side of $\widehat{F}$, $i$ and $j$ have different parity.

Let $N_1$ be the component of $M-T$ containing $D$.  As in Case (II), 
$K$ may be perturbed to miss $N_1$ completely. Note that in the context
of the Addendum, the constituent annuli can be taken disjoint from  the
disks of parallelism of both $a_i$ and $a_j$.

Compressing $\partial N_1$ along $D$ gives a 2-sphere disjoint from $K$. As
the exterior of $K$ is irreducible, we see that $N_1$ is a solid torus. We
show that either $A$ and $B$ are longitudinal annuli in $N_1$, that is that 
$\Delta(\bdry D, a_i) = 1$ on $\partial N_1$, thereby showing that
$A$ is parallel into $B$, or conclusion (2) holds.

If $\Delta(\bdry D, a_i) = 0$ then $M$ contains a punctured lens space
(coming from the elements of $\AO(\sigma)$ up to $a_i$ in union with
$D$ along $a_i$). 

Assume $\Delta(\bdry D, a_i) = n \geq 2$. Set $V = \nbhd(N_1 \cup A_1 \cup \dots \cup A_i)$ if $m=2$ or $V=\nbhd(N_1 \cup \Theta_1 \cup \dots \Theta_i)$ 
if $m=3$. In the context of the Addendum we expand $B$ to include disks 
of parallelism of $a_i,a_j$ and surger away intersects of $K$ with either 
of these disks of parallelism as in Case(IIb). Then $V$ is a Seifert fibered 
space over the disk with two exceptional fibers of orders $m$ and $n$ (note $A_1 \cup \dots \cup A_i$ or $\Theta_1 \cup \dots \Theta_i$ 
intersect $T$ only along their boundary, else the core curve of $B$ bounds 
a disk outside of
$N_1$ -- which along with $N_1$ would create a lens space summand in $M$).  

$K$ can be perturbed so that it is disjoint from $N_1$ and intersects $A_1 \cup \dots \cup A_i$ or $\Theta_1 \cup \dots \cup \Theta_i$ in a single point if $m=2$ or $m=3$ respectively (this is most easily seen by checking the labellings
around the vertices of $a_i$, see Figure~\ref{fig:Figure1}).
Thus $K \cap V$ is isotopic rel boundary  to the co-core of the unique essential
annulus in the Seifert fibered space, $V$. It follows from this that the 
twice-punctured torus
$S = \partial V - \nbhd(K)$ is incompressible in 
$V - \nbhd(K)$.

\begin{claim}\label{clm:ontoV}
 $K \cap (M-V)$ is isotopic rel
boundary in $M-V$ onto $\partial V$.
\end{claim}

\begin{proof}
Let $R=\partial (M-V-\nbhd(K))$. Then $R$ must be compressible in
the exterior of $K$, $X$, by Theorem 2.4.3 of \cite{cgls:dsok} (where $R$ is 
the $S$ there and the $r_0$ there is $p/q$). As $S$ is incompressible in
$V - \nbhd{K}$, $R$ must compress in $M-V-\nbhd{K}$. 
If $S$ were incompressible
in $M-V-\nbhd{K}$, then the Handle Addition Lemma, Lemma 2.1.1 of
\cite{cgls:dsok},  would imply that $\partial V$ is incompressible in $M-V$.
But this contradicts that $M$ is atoroidal. Thus $S$ must compress in 
$M-V-\nbhd{K}$. As $X$ is anannular, $S$ must compress in $M-V$ to a boundary 
parallel annulus. This implies that $K \cap (M-V)$ is isotopic in $M-V$
to $\partial V$, as claimed.
\end{proof}

Since $M$ is atoroidal, and $\partial V$ is incompressible in $V$, 
$\bdry V$ must compress outside of $V$.  
Compressing $\bdry V$ then yields a $2$--sphere that bounds a $3$--ball $B^3$.
If this 3-ball contained $V$, then $K$ could be isotoped into this 3-ball,
contradicting the irreducibility of the exterior of $K$. So we may assume
that this 3-ball lies outside of $V$, and consequently that $M-V$ is a 
solid torus. Then $M$ is a Seifert-fibered space over the 2-sphere with
three exceptional fibers one of which has order $m$.
Let $h = \nbhd(K \cap V)$. One checks that $H_1 = V - h$ is a genus 
two handlebody. Certainly, $H_2 = (M-V) \cup h$ is a genus two handlebody. 
That is, $H_1 \cup H_2$ is a Heegaard splitting of $M$.  
Recall that $K \cap (M-V)$ can be isotoped rel its boundary 
onto $\partial V$. Thus $K$ can be isotoped onto 
$\partial H_1$. This is conclusion $(2)$ of the Lemma.

Thus we may assume $\Delta(\bdry D, a_i) = 1$ and $N_1$ gives an isotopy of 
$A$ to $B$ as
desired. If $j>i+1$, then this isotopy from $A$ to $B$ gives a thinner 
presentation of $K$ (the two sub-arcs $K \cap A$ of $K$ are flattened onto
$\hatF$ by this isotopy, perturb them back slightly off of $\hatF$ to 
reinstate a Morse presentation of $K$ with fewer extrema.)
As we started with a thin presentation of $K$, it must be that $j=i+1$.

These 3 cases exhaust the possibilities, thereby completing the proof of 
Lemma~\ref{lem:SFS}.  
\end{proof}

\begin{lemma}\label{lem:triplets}
Let $\sigma$ be an extended Scharlemann cycle of length $m=2$ or $m=3$. 
Assume that conclusion $(2)$ of Lemma~\ref{lem:SFS} does not hold.
If $a_i,a_j,a_k \in \ao(\sigma)$
together lie on an annulus in $\hatF$, then there must be intersections of $K$ with
this annulus that do not belong to $a_i,a_j,a_k$.
\end{lemma}

\begin{proof} Without loss of generality assume $a_i, a_j$ cobound annulus $B_1$ on $\hatF$ and 
$a_j,a_k$ cobound annulus $B_2$ on $\hatF$, such that $K$ intersects $B_1 \cup B_2$
only along $a_i,a_j,a_k$ and such that $B_1,B_2$ have disjoint interiors. 
Then by Lemma~\ref{lem:SFS} we may assume $j=i+1, k=i+2$
and the annuli or constituent annuli, $A_i,A_{i+1}$, of $\AO(\sigma)$ 
are parallel into $B_1,B_2 \subset \hatF$.
But then we can use these parallelisms of the annuli to guide a thinning of $K$ with
respect to $\hatF$.
\end{proof} 

\begin{lemma}\label{lem:thetatriplets}
Let $\sigma$ be an extended Scharlemann cycle of length $m=3$. 
Assume that conclusion $(2)$ of Lemma~\ref{lem:SFS} does not hold.
Assume that $\theta$--curves $a_i,a_j,a_k \in \ao(\sigma)$
together lie on an annulus in $\hatF$, with $a_j$ between $a_i,a_k$ on 
$\hatF$.  Then there must be intersections of $K$ with
this annulus that do not belong to $a_i,a_j,a_k$ or to the disks of 
parallelism of $a_i,a_k$.
\end{lemma}

\begin{proof}
Assume the only other intersections of $K$ occur within the disks
of parallelism of $a_i,a_k$. We apply the proof of Lemma~\ref{lem:triplets} 
using the Addendum to Lemma~\ref{lem:SFS}.
\end{proof} 

\subsection{Parallelisms of $a(\sigma)$ and $\theta(\tau)$.}

\begin{defn} An extended Scharlemann cycle is called {\em proper} if 
in its corner no label appears twice.
\end{defn}

\begin{lemma}\label{lem:parallelesc}
Assume that $\hatF$ is a strongly irreducible Heegaard surface for
$M$.
Let $\sigma, \tau$ be proper extended Scharlemann cycles of $\Lambda$ of
length $2$ or 
$3$ with different core labels. If there is an annulus, $A$, in $\hatF$ 
containing an element of $\ao(\sigma)$ and an element of 
$\ao(\tau)$ 
then either
\begin{enumerate}
\item $L(\sigma)$ and $L(\tau)$ intersect in two non-empty label intervals,
$L(\sigma) \cup L(\tau)$ includes all labels, and one of $\sigma$ or $\tau$
has length $3$; or
\item $M$ is a Seifert fibered space over the $2$--sphere 
with three exceptional fibers, and the orders of two of these exceptional
fibers are the lengths of $\sigma,\tau$. Furthermore, there is a 
genus $2$ (hence minimal genus) Heegaard splitting of $M$
 with respect to which $K$ has bridge number at most 
$k-1$, where $k=|A \cap K|$.
\end{enumerate}
\end{lemma}

\begin{proof}
After possibly paring down $\sigma,\tau$ we may assume that 
$\ao_{\sigma}=\partial(\AO(\sigma)), 
\ao_{\tau}=\partial(\AO(\tau))$ are the only
elements of $\ao(\sigma),\ao(\tau)$ (resp.) lying entirely
in $A$. If there are 
intersections of $\AO(\sigma), \AO(\tau)$ with
$A$ coming from the interior of faces of
$\sigma, \tau$, then we may assume they are essential simple closed 
curves in $A$ (Lemma~\ref{lem:AEntscc}). 
Take such a curve that is innermost
on $\sigma$, say, bounding a disk $D$. Surgery of $A$ along $D$ gives a 
disk $A'$
which we may take (after possibly further surgeries) to intersect
$\AO(\sigma)$ only along its boundary. But then $\nbhd(\AO(\sigma)
\cup A')$ is a lens space summand in $M$ -- a contradiction. Thus 
we may assume that $\sigma, \tau$ intersect $A$ only along their edges.
The argument now divides up according to $L(\sigma) \cap L(\tau)$. See
Lemma~\ref{lem:commonlabelsets} for the possibilities. We work through
these, ending with Case IV that $L(\sigma) \cup L(\tau)$ contains all labels 
and 
$L(\sigma) \cap L(\tau) \neq \emptyset$. In this final case we show that
at least one of $\sigma,\tau$ must be of length $3$, giving conclusion (1)
above. The generic argument is Case I. 

\medskip

{\bf Case I:} $L(\sigma) \cap L(\tau) = \emptyset$

\medskip

\begin{proof}[Proof of Case I]
Then $\AO(\sigma)
\cap \AO(\tau) = \emptyset$ and we may assume $\AO(\sigma), 
\AO(\tau)$ intersect $A$ only along their boundaries.  
Consequently $N = 
\nbhd(\AO(\sigma) \cup A \cup \AO(\tau))$ is a Seifert fibered space
over the disk with two exceptional fibers whose orders are given
by the lengths of $\sigma, \tau$ (e.g.
Lemma~\ref{lem:gtsolidtorus}). Note that as $M$ contains no Klein bottles,
the length of either $\sigma$ or $\tau$ must be $3$.

\begin{claim}~\label{clm:stronglyirred}
$M$ is a Seifert fibered space over the 2-sphere with three exceptional fibers, where
the orders of the exceptional fibers include the lengths of $\sigma$ and $\tau$.
Furthermore, 
$p/q$ is not a boundary slope for $X$.
\end{claim}

\begin{proof}
As $M$ is atoroidal, $\partial N$ compresses in $M-N$. As $M$ is irreducible, 
the resulting $2$--sphere bounds a 3-ball. 

First assume this 3-ball contains $N$. Then the core of $A$ is an essential
curve in $\hatF$ that does not bound a meridian on either side of $\hatF$
(otherwise, as argued above, it in union with $\AO(\sigma)$ can be 
surgered to produce a lens space summand in $M$). Then Lemma~3.3 of 
\cite{bgl:og2hsfs} (which shows if a Heegaard surface of a $3$--manifold other than $S^3$ contains a non-meridional simple closed curve that lies in a $3$--ball, then the splitting is weakly reducible) contradicts our assumption that $\hatF$ is strongly 
irreducible.

Thus the 3-ball is disjoint from $\partial N$. That
is $M-N$ is a solid torus. Then $M$ is a Seifert fibered space over
the two sphere with three exceptional fibers with finite first homology.
Furthermore, $M$ contains no incompressible surface. Thus Theorem 2.0.3 of 
\cite{cgls:dsok} implies that
if $p/q$ were a boundary slope then $X(r)$ would contain an incompressible
surface whenever $\Delta(r,p/q) > 1$. Since $q \geq 2$ and $X(1/0)$ is
$S^3$, this is impossible.
\end{proof}

\begin{claim} $K \cap N$ consists of at most $k-2$ arcs which can be 
isotoped (rel $\partial N$) to a collection of co-cores of the essential 
annulus in 
the Seifert fibered space $N$.
\end{claim}
\begin{proof}
First note that each of $K \cap \AO(\sigma), K \cap \AO(\tau)$ can be isotoped
off of $\AO(\sigma),\AO(\tau)$ (resp.) to intersect 
$A$ in a single point.
See Figure~\ref{fig:Figure1}. 
\begin{figure}
\centering
\input{Figure1smaller.pstex_t}
\caption{}
\label{fig:Figure1}
\end{figure}
Thus we may take $K \cap N$ to be a collection of $k-2$ arcs of the form
$point \times I \subset A \times I$. Each of these arcs can be isotoped 
to a co-core of the essential annulus $c \times I \subset N$, 
where $c$ is the core of 
$A$: In particular, say $\sigma$ is a length $3$ extended Scharlemann cycle.
Let $f$ be the disk of parallelism in $G_F$ between two edges of 
$\partial \Theta(\sigma)$ (as shown in Figure~\ref{fig:Figure1}). Then the arc $p \times I \subset
f \times I$, where $p$ is a point of $f$, can be isotoped, rel $\partial N$, 
across $\nbhd(\Theta(\sigma))$ 
to a cocore of $c \times I$ in $N$. Figure~\ref{fig:Figure2} shows an isotopy of $K$ moving an intersection with $f$ out to an edge of $\bdry f$ and then through the trigon face of $\sigma$ to produce an intersection with $A - f$; the dashed circles show the resulting new intersections $K \cap \hatF$.
\end{proof}

\begin{figure}
\centering
\input{Figure2.pstex_t}
\caption{}
\label{fig:Figure2}
\end{figure}

$\partial N \subset M$ is a 2-torus. Let $T=\partial N - \nbhd(K)$. Maximally
compress $T$ in $X$. 
Claim~\ref{clm:stronglyirred} says that $p/q$ is not a $\partial$--slope, 
thus $T$
compresses to $\partial$--parallel annuli in $X$. Note that these annuli may be
nested in $X$. Each $\partial$--parallel annulus defines an isotopy in $M$ of 
a subarc of $K$, with endpoints in $K \cap \partial N$, onto $\partial N$ -- keeping
the endpoints of the arc fixed. Furthermore, every point of $K \cap \partial N$ 
belongs to such an arc of $K$. Let $k_1', \dots , k_r'$ be the subarcs of $K$
lying within the outermost of these $\partial$--parallel annuli and let $k_1, \dots,
k_r$ be the complementary subarcs of $K$. Note that the $k_i$ lie on the same
side of $\partial N$. Then the isotopies described above 
corresponding to the outermost annuli deform $K$ to the union of the arcs $k_1, \dots,
k_r$ along with arcs that lie in $\partial N$.  As the complement of $K$ in $M$ 
is hyperbolic,
it must be that the arcs $k_1, \dots, k_r$ lie in $N$. By the Claim above, 
$k_1, \dots, k_r$ can be isotoped to co-cores of the essential annulus in $N$. 
Let $\alpha$ be a disjoint co-core of this annulus. Then 
$N - \nbhd(\alpha), (M-N) \cup \nbhd(\alpha)$ is a genus $2$ splitting of $M$. 
Furthermore, each of $k_1, \dots, k_r$  is $\partial$--parallel in $N -\nbhd(\alpha)$. 
Perturbing each arc of $K - (\cup k_i)$ into $M-N$ puts $K$ in bridge position 
with respect to this genus $2$ splitting. Noting that $r \le k-2$, we have 
conclusion $(2)$ of Lemma~\ref{lem:parallelesc}.
\end{proof}

In the remaining cases, 
the argument of Case I gives the desired conclusion once we show how
to modify $\AO(\sigma)$ and 
$\AO(\tau)$ so that they are mutually disjoint and intersect $A$ only along their boundaries, $\ao_{\sigma},
\ao_{\tau}$.
We write $\AO(\sigma) = E_{\sigma} \cup F_{\sigma}, \AO(\tau) = 
E_{\tau} \cup F_{\tau}$ where $F_{\sigma},F_{\tau}$ is the union of
faces of $\sigma,\tau$ (resp.) (thought of as disks in $X$),
and where $E_{\sigma},E_{\tau}$ are rectangles in $\nbhd(K)$ describing
an extension of $F_{\sigma},F_{\tau}$ across $\nbhd(K)$ to form the
long \mobius band or twisted $\theta$--band. In what follows (at least up to
Case IV') we 
show how to choose 
$E_{\sigma},E_{\tau}$ to be
disjoint, making $\AO(\sigma),\AO(\tau)$ disjoint and 
intersecting $A$ along
$\partial A$.

Let $\{x,z\}$ be the extremal labels of $\sigma$, and $\{y,w\}$ the extremal 
labels of $\tau$. Let $\{\alpha, \beta\}$ or  $\{\alpha, \beta, \gamma\}$ be
the corners of $\sigma$ depending on the length of $\sigma$. Similarly let
$\{\alpha', \beta'\}$ or $\{\alpha', \beta', \gamma'\}$ be the corners of
$\tau$ according to its length.

\medskip

{\bf Case II:} $L(\sigma) \cap L(\tau)$ is a single interval of labels
$\arc{xy}$ (including a point interval).

\medskip

\begin{proof}

First we consider the case of a point interval, that is, when $x=y$
(and $z \neq w$). 
Then $\ao_{\sigma},\ao_{\tau}$ intersect in a single 
point (at vertex $x=y$).
As $\ao_{\sigma},\ao_{\tau}$ both lie in $A$, they must be 
non-transverse
around vertex $x$ on $G_F$. This means that as one reads around vertex
$x$ on $G_F$, labels $\{\alpha, \beta\}$ ($\{\alpha, \beta, \gamma\}$) do not
separate the labels $\{\alpha', \beta'\}$ ($\{\alpha', \beta', \gamma'\}$). 
See Figure~\ref{endof3fig5-alt}
for $\sigma,\tau$ of length $2$.
 We choose disjoint $E_{\sigma},
E_{\tau}$ as pictured in Figure~\ref{endof3fig7} (with $x=y$), making 
$\AO(\sigma), \AO(\tau)$ disjoint. The argument for Case I now gives
a genus $2$ splitting of $M$ with bridge number at most $k-1$  
($K$ intersects
$A$ along $\ao_{\sigma}, \ao_{\tau}$ only $3$ times here, rather than
$4$).

Thus we may assume that $\{x,z\} \cap \{y,w\} = \emptyset$. Let $b_{\sigma}$
be the component of $\ao(\sigma)$ through vertex $y$.  
Then $b_{\sigma}$ intersects
$\ao_{\tau}$ in a single point (at $y$). 
See 
Figures~\ref{endof3fig6}, ~\ref{endof3fig6-theta}.
Again $b_{\sigma}$ is disjoint from
$\ao_{\sigma}$ which lies in $A$ with $\ao_{\tau}$, 
so $b_{\sigma}$ must intersect
$\ao_{\tau}$ tangentially. Thus, as one transverses vertex $y$ in $G_F$ the 
$\{\alpha,\beta\}$ ($\{\alpha,\beta, \gamma \}$) labels do not separate 
$\{\alpha',\beta'\}$ ($\{\alpha',\beta', \gamma' \}$). We then
may choose disjoint $E_{\sigma},E_{\tau}$ as pictured in 
Figures~\ref{endof3fig7}, ~\ref{endof3fig7-theta}. 
 The argument of Case I
now gives a genus $2$ splitting of $M$ with bridge number at most $k-2$.
\end{proof}

\begin{figure}
\centering
\input{endof3fig5-altV2.pstex_t}
\caption{}
\label{endof3fig5-alt}
\end{figure}


\begin{figure}
\centering
\input{endof3fig7.pstex_t}
\caption{}
\label{endof3fig7}
\end{figure}

\begin{figure}
\centering
\input{endof3fig6.pstex_t}
\caption{}
\label{endof3fig6}
\end{figure}

\begin{figure}
\centering
\input{endof3fig6-theta.pstex_t}
\caption{}
\label{endof3fig6-theta}
\end{figure}

\begin{figure}
\centering
\input{endof3fig7-theta.pstex_t}
\caption{}
\label{endof3fig7-theta}
\end{figure}

\medskip

{\bf Case III:} $L(\tau) \subset L(\sigma)$ or $L(\sigma) \subset L(\tau)$

\medskip

\begin{proof}
The argument is the same in either case. We assume $L(\tau) \subset L(\sigma)$.

We may assume that, say, $y \neq x,z $. 
Let $b_{\sigma}$ be the component of 
$\ao(\sigma)$ through vertex $y$ -- connecting $y$ to another vertex $r$.
If $r=w$, then $\sigma,\tau$ have the same core labels, contrary to our
assumption.

Thus we assume $r \neq y,w,x,z$ ($r \neq y$  by the Parity Rule). 
Then $b_{\sigma}$ intersects $\ao_{\tau}$
in a single point (at the vertex $y$). Since $b_{\sigma}$ is disjoint from
$\ao_{\sigma}$, and $\ao_{\sigma}, \ao_{\tau}$ are contained in $A$,
$b_{\sigma}$ must intersect $\ao_{\tau}$ tangentially. That is, as 
one transverses
around the (fat) vertex $y$ of $G_F$ the labels $\{\alpha, \beta\}$ 
($\{\alpha, \beta, \gamma\}$) are not
separated by the labels $\{ \alpha', \beta' \}$ 
($\{ \alpha', \beta', \gamma' \}$). Thus in $\nbhd(K)$, we
may choose disjoint $E_{\sigma},E_{\tau}$ as pictured in Figure~\ref{fig:plmb2} 
(which illustrates the case of length $2$ with $r$ between $y$ and $w$) thereby making $\AO(\sigma),\AO(\tau)$ disjoint. 
We now apply the argument from Case I giving a $k-2$ bridge presentation of $K$.
\end{proof}

\begin{figure}[h]
\centering
\input{plmb2.pstex_t}
\caption{}
\label{fig:plmb2}
\end{figure}

\medskip

{\bf Case IV:} $L(\sigma) \cup L(\tau)$ contains all labels of $G_Q$, and
$L(\sigma)$ overlaps $L(\tau)$ in two intervals of labels: $\arc{xy}$ and
$\arc{wz}$. 

\begin{figure}[h]
\centering
\input{plmb5.pstex_t}
\caption{}
\label{fig:plmb5}
\end{figure}

\medskip

This case is conclusion (1) of the Lemma once we have shown that
one of $\sigma,\tau$ is of length $3$. So we assume that $\sigma,\tau$
are both of length $2$ and we are as in Figure~\ref{fig:plmb5}. 

If $\{x,z\} = \{y,w\}$ then $a_{\sigma}=\bdry A(\sigma), 
a_{\tau}=\bdry A(\tau)$ are isotopic
on $\hatF$ and both go through vertices $x,z$ of $\hatF$. Thus $A(\sigma),
A(\tau)$ can be amalgamated along their boundary to create an embedded
Klein bottle.

Next assume that $x=y$ but $z \neq w$. Let $b_{\sigma}$ be the component of 
$a(\sigma)$ through vertex $w$. As $b_{\sigma}$ is disjoint from $a_{\sigma}$
and intersects $a_{\tau}$ once at $w$, $b_{\sigma}$ and $a_{\tau}$ intersect
non-transversely. Thus around vertex $w$, the labels $\{ \alpha, \beta \}$ do 
not separate $\{ \alpha', \beta' \}$. Similarly, as $a_{\sigma}, a_{\tau}$
intersect in a single point at vertex $x$ and yet are isotopic, their
intersection is non-transverse. That is, around vertex $x$, the labels 
$\{ \alpha, \beta \}$ do 
not separate $\{ \alpha', \beta' \}$. Figure~\ref{fig:plmb7} (with $x=y$) shows that we
can choose disjoint $E_{\sigma},E_{\tau}$.

Thus we may assume $\{x,z\} \cap \{y,w\} = \emptyset$. Let $b_{\sigma}$ be
the component of $a(\sigma)$ through vertex $y$, and let $r$ be the other 
vertex of $G_F$ to which $b_{\sigma}$ is incident. Then $r \neq x,z$.

Assume $r \neq w$. Then as $b_{\sigma}$ intersects $a_{\tau}$ once and is 
disjoint from $a_{\sigma}$, it must intersect $a_{\tau}$ non-transversely.
That is, around vertex $y$ the labels $\{ \alpha, \beta \}$ do not
separate $\{ \alpha', \beta' \}$. Let $c_{\sigma}$ be the component of
$a(\sigma)$ through vertex $w$. Again, $c_{\sigma}$ must intersect $a_{\tau}$
non-transversely at $w$. Hence around $w$ in $G_F$, the labels 
$\{ \alpha, \beta \}$ do not separate $\{ \alpha', \beta' \}$. Thus we
may choose disjoint disks $E_{\sigma},E_{\tau}$ in $\nbhd(K)$ as pictured
in Figure~\ref{fig:plmb7}.

\begin{figure}[h]
\centering
\input{plmb7.pstex_t}
\caption{}
\label{fig:plmb7}
\end{figure}

This leaves us with the case that $r = w$, whose argument is slightly different
from the preceding ones.

\medskip
{\bf Case IV':} In Case IV above $r=w$.

\medskip

\begin{proof}
This is the case when the core labels of $\sigma,\tau$ are ``antipodal''
labels. Let $b_{\sigma}$ be the component of $a(\sigma)$ through vertices
$y$ and $w$ of $G_F$. Then $b_{\sigma}$ and $a_{\tau}$ intersect twice.
Since $b_{\sigma}$ is disjoint from $a_{\sigma}$ which is isotopic to $a_{\tau}$,
the algebraic intersection number of $b_{\sigma}$ and $a_{\tau}$ is $0$.
Thus we can choose $E_{\sigma}, E_{\tau}$ in $\nbhd(K)$ so that they are
either (1) disjoint as in Figure~\ref{fig:plmb7} or (2) intersect in exactly two arcs as in Figure~\ref{fig:plmb8}. This follows since the labels $\{\alpha,\beta\}$
must separate $\{\alpha',\beta'\}$ either (1) around neither vertices $y,w$
or (2) around both vertices $y,w$.

\begin{figure}[h]
\centering
\input{plmb8.pstex_t}
\caption{}
\label{fig:plmb8}
\end{figure}

$A$ is the annulus on $\hatF$ between $a_{\sigma}$ and $a_{\tau}$.  
We may assume $\Int A$ is disjoint from the vertices of $G_F$ (else there
is another component of $a(\sigma)$ or $a(\tau)$ in $A$, which we initially
assumed is not the case). 
Consider $A(\sigma)=E_{\sigma} \cup F_{\sigma}, A(\tau)=E_{\tau} \cup F_{\tau}$
where $F_{\sigma},F_{\tau}$ is the union of faces of $\sigma,\tau$. Then
$A(\sigma),A(\tau)$ are either (1) disjoint or (2) intersect in two double
arcs (from $x$ to $y$ and $w$ to $z$ along $K$). If (1), $M$ contains 
an embedded Klein bottle. If (2), 
$S=A(\sigma) \cup A \cup A(\tau)$ is a Klein bottle that 
self-intersects in a single double-curve (again, 
we may assume $A(\sigma),A(\tau)$ are 
disjoint from $A$ except along $a(\sigma),a(\tau)$). The two preimage curves
are disjoint from the cores of each of $A(\sigma), A(\tau)$, and $A$,
and consequently bound disjoint disks, \mobius bands in the pre-image.
We may surger along the double curve to
obtain an embedded projective plane or Klein bottle in $M$.
\end{proof}

This completes the proof of Lemma~\ref{lem:parallelesc}.
\end{proof}

\begin{cor}\label{cor:noparallelesc2}
Let $\sigma, \tau$ be proper extended Scharlemann cycles of $\Lambda$ of
length $2$. If there is an annulus, $A$, in $\hatF$ 
containing a component of $a(\sigma)$ and a component of 
$a(\tau)$ 
then $\sigma$ and $\tau$ have the same core labels.
\end{cor}

\begin{proof}
In this context, the proof of Lemma~\ref{lem:parallelesc} constructs 
within $M$ a Seifert fibered space over the disk with two exceptional
fibers each of order $2$. But then $M$ would contain
a Klein bottle, a contradiction. 
\end{proof}

\begin{lemma}\label{notwo}
Assume $\hatF$ is a strongly irreducible Heegaard surface for
$M$.
If $\sigma, \tau$ are extended Scharlemann cycles of $\Lambda$ of length
$2$ or $3$ with the property that $L(\sigma) \cup L(\tau)$ includes
all labels, then either 
\begin{enumerate}
\item $t \leq 32(g-1)$; or 
\item $M$ is a Seifert fibered
space over $S^2$ with three exceptional fibers, where at least one 
of the exceptional fibers has order $2$ or $3$. Furthermore, there is a
genus $2$ Heegaard splitting of $M$ with respect to which $K$ is at
most $4$--bridge.
\end{enumerate}
\end{lemma}

\begin{proof}
We may assume, by paring down as needed, that $\sigma$ and $\tau$ are proper.
We may also assume $|L(\sigma)| \geq |L(\tau)|$; hence $|L(\sigma)| \geq t/2$.
If $(1)$ does not hold then we have
$|a(\theta(\sigma))|=|L(\sigma)|/2 \geq t/4 > 8(g-1)$. Note that the elements
of $a\theta(\sigma)$ are disjoint.

Since a $\theta$--curve has Euler characteristic $-1$, the maximal number of 
disjoint $\theta$--curves on $\hatF$, none of which lies in an annulus, is 
$-\chi(\hatF)=2(g-1)$. Suppose $\sigma$ has length $3$. Since 
$|a\theta(\sigma)|>8(g-1)$, there must be more than $6(g-1)$ elements of 
$a\theta(\sigma)$ that lie in annuli, and hence at least three such elements 
that are isotopic on $\hatF$ and contain no further elements of 
$a\theta(\sigma)$ between them. The same holds {\em a fortiori} if $\sigma$
has length $2$.
 
By Lemma~\ref{lem:triplets}, either conclusion $(2)$ of Lemma~\ref{notwo}
holds, or there must be intersections of $K$ with an 
annulus in $\hatF$ between 
a pair of these elements, $c_1,c_2$.  We assume the latter. 
Such an intersection gives rise to a 
vertex of $G_F$ that must correspond to an element of $L(\tau)$ that is not in $L(\sigma)$.
Thus there is an element, $b$, of
$\ao(\tau)$ lying in the annulus of $\hatF$ cobounded by $c_1,c_2$.  Then $b$ must be
parallel to and disjoint from, say, $c_1$. Taking $b$ to be the nearest such to $c_1$, we may assume 
$c_1,b$ cobound an annulus $A \subset \hatF$ such that $|K \cap A| \leq 5$ (four from the vertices of $c_1,b$ and at most one more in the interior of $A$ belonging to an element of $\ao(\tau)$ whose other vertex belongs to $c_1$).
 Now pare down $\sigma, \tau$ so that 
$c_1 = \partial \AO(\sigma)$ and $b = \partial \AO(\tau)$.
Then, as $b$ contains a vertex which is not in $L(\sigma)$,
$L(\sigma)$ and $L(\tau)$ can intersect in at most one label interval.
Now Lemma~\ref{lem:parallelesc} gives 
conclusion (2) above. 
\end{proof}

\section{Trigons of Type I or I\!I}\label{sec:typeIorII}

\begin{lemma}\label{lem:existsS22}
Let $\sigma$ be a trigon face of $\Lambda_x$ which is not an extended
Scharlemann cycle. Then $\sigma$ must contain a length $2$ Scharlemann
cycle in its interior.

In particular, assume $\sigma$ is not an $x$--edge cycle. 
That is, one corner, $\beta$, of 
$\sigma$ has
no $x$ labels, and another corner, $\gamma$, has two. 
Then there is a length $2$ extended 
Scharlemann cycle in $\sigma$ on the arm 
opposite $\beta$. Furthermore, if $\sigma$ has a 
second Scharlemann cycle in its interior, it must lie on the
arm of $\sigma$ between $\beta$ and $\gamma$. 
\end{lemma}

\begin{figure}
\centering
\input{xcycletrigon.pstex_t}
\caption{}
\label{fig:xcycletrigon}
\end{figure}

\begin{figure}
\centering
\input{noncycletrigon.pstex_t}
\caption{}
\label{fig:noncycletrigon}
\end{figure}


\begin{proof}
Assume that
$\sigma$ is an $x$--edge cycle. Then we may assume it
is as in Figure~\ref{fig:xcycletrigon}, where $j,k$ and $l$ are not all
equal.
If there is no order $2$ Scharlemann cycle in $\sigma$ 
then we must have $k < j+1$, $\ell < k+1$, and $j < \ell+1$.  
All three inequalities cannot simultaneously be true.

Now assume $\sigma$ is not an $x$--edge cycle.
We may take it as in Figure~\ref{fig:noncycletrigon}.
Assume there is no Scharlemann cycle in the arm of $\sigma$ opposite to
$\beta$.  
Then $j$ cannot be in $\arc{\ell+1,x}$ on $\gamma$. Thus $\ell+1$ must be
in $\arc{j+1,x}$ on $\alpha$. Hence the length of the $\beta$--corner
must be at least as large as that of $\gamma$ and $\beta$ would contain a 
label $x$.

\begin{figure}
\centering
\input{SC2intrigon.pstex_t}
\caption{}
\label{fig:SC2intrigon}
\end{figure}
Now assume that $\sigma$ contains a second Scharlemann cycle in its $\alpha
\beta$--arm. Refer to Figure~\ref{fig:SC2intrigon}. 
Looking at the $\alpha \gamma$--arm,
we see that the interval $\arc{x\ell}$ on the extension of $\alpha$ 
must be at least as
large as $\arc{jx}$ on $\gamma$, which in turn is strictly larger than
$\arc{kx}$ on $\alpha$. But $\arc{kx}$ on $\alpha$ must
be larger than $\arc{ks}$ on $\beta$.  Looking at the $\gamma \beta$--arm,
though, we see that $\arc{ks}$ on $\beta$ is larger than $\arc{x\ell}$
on $\gamma$ -- a contradiction.
\end{proof}

\begin{defn}\label{def:trigontypes}
Let $\sigma$ be a trigon face of $\Lambda_x$, considered as a subgraph of
$\Lambda$, that does not correspond to an extended Scharlemann cycle of 
length $3$.
By Lemma~\ref{lem:existsS22}, $\sigma$ must 
contain a Scharlemann cycle of length $2$.
As $\sigma$ is
a face of $\Lambda_x$  any such Scharlemann cycle must be the core of an
extended Scharlemann cycle in $\sigma$ that abuts the central trigon of 
$\sigma$ and is disjoint from any $x$--edge of $\sigma$ (i.e.
the extended Scharlemann
cycle lies within $\sigma$ abutting the central trigon). There
can be at most two such Scharlemann cycles within $\sigma$, which must
lie in different arms of $\sigma$.  
We say that $\sigma$ is a {\em Type I trigon}
if there is only one and the corners of its central
trigon are on three different pairs of labels.  It is of {\em Type I\!I } if it contains two such
Scharlemann cycles.  It is a {\em Type I\!I\!I  trigon} 
otherwise -- that is,
it contains a single length $2$ Scharlemann cycle and the three corners in 
its central trigon represent two different label pairs.
\end{defn}

\begin{defn}\label{def:T}
Let $T=T(j,k,N)$ be one of the two subgraphs of $\Lambda$ pictured in 
Figure~\ref{fig:trigonwithlongarms}, 
where $j >0, N \geq 0$ and either
\begin{enumerate}
\item
$0 \le k < j$ and $T$ is the subgraph on the top.  In this case we say
that $T$ is of {\em Type I}. Furthermore, we require that $2j+2N+2 \leq t$, 
which is equivalent to no label appearing twice on any corner of $T$.
Define $\epsilon$ to be the extended Scharlemann cycle in $T$ with corner $\arc{j,-j+1}$ in the $\alpha \gamma$--arm.
Note that $T$ has $N+1$ bigons in the $\alpha \gamma$--arm beyond the extended
Scharlemann cycle $\epsilon$, and $N$ bigons in the $\alpha \beta$-- and $\beta \gamma$--arms.
\item
$j < k \leq 3j$ and $T$ is the subgraph on the bottom. 
In this case we say that $T$ is of {\em Type I\!I}. 
Furthermore, we require that $j+k+2N+2 \leq t$, which is equivalent to no label appearing twice on any corner of $T$. 
Define  $\epsilon$ and $\epsilon'$ to be the extended Scharlemann cycles in $T$ with corners $\arc{j, -j+1}$ and $\arc{j+1,k}$ in the $\alpha \gamma$-- and $\alpha \beta$--arms respectively.
Note that $T$ has $N+1$ bigons  in the $\alpha \gamma$-- and $\alpha \beta$--arms beyond the extended
Scharlemann cycles $\epsilon$ and $\epsilon'$, and
$N$ bigons in the $\beta \gamma$--arm.
\end{enumerate}
In this figure, $\alpha, \beta, \gamma$ 
are corners of vertices of $\Lambda$, and we refer to the $\alpha$--corner
of $T$ as its {\em ideal corner}, denoted $c(T)$. Note that every corner
of $T$ (as a labelled interval) is contained in $c(T)$. Note also that there are $2N+2$ labels appearing in $T$ which are not labels in an extended
Scharlemann cycle in $T$.
\end{defn}

\begin{figure}
\centering
\input{trigonwithlongarmsV4-switch.pstex_t} 
\caption{}
\label{fig:trigonwithlongarms}
\end{figure}

\begin{lemma}\label{lem:findT}
Let $\sigma$ be a trigon face of $\Lambda_x$ that corresponds to a
trigon of Type I or I\!I. Let $n$ be the shortest distance (along a corner of
$\sigma$) from $x$ to an
extended Scharlemann cycle within $\sigma$. After possibly relabelling
$\sigma$, there is   
a subgraph $T(j,k,N)$ of $\sigma$ such that $T(j,k,N)$ is either
Type I or Type I\!I according to whether $\sigma$ is Type I or Type I\!I,
and such that $N=n-1$. In particular, $x$ corresponds to a label in 
$c(T)$ (after relabelling). 
\end{lemma}

\begin{proof}
Assume $\sigma$ is of Type I. Then we may assume after possibly relabelling
that $\sigma$ is as in the left or right of Figure~\ref{fig:findT1} according to whether $\sigma$ is
an $x$--cycle or not (by Lemma~\ref{lem:existsS22}). 
\begin{figure}
\centering
\input{findT1.pstex_t}
\caption{}
\label{fig:findT1}
\end{figure}
 Note that $-j+1 < k < j$ as $\sigma$ is of Type I.
(If $k=\pm j$, then $\sigma$ would be of Type I\!I\!I.  By parity, $k \neq j+1$ or $-j+1$.
If $k \in \arc{(j+1)x}$ then there would be a second Scharlemann cycle on the $\alpha\beta$--arm.  
If $k \in \arc{x (-j-1)}$ then,
regardless of whether $\sigma$ is an $x$--cycle or not, there is a second Scharlemann cycle along the $\beta \gamma$--arm.) 
Letting $N=x-(j+1)$, 
$\sigma$ contains at least $N$ bigons along the
$\alpha \beta$-- and $\beta \gamma$--arms, and $N+1$ out from the extended
Scharlemann cycle on the $\alpha\gamma$--arm  ($x-(k+1) > x-(j+1)
=N$; then from the $\alpha \gamma$--arm deduce that the length of the 
interval from $x$ to $-j$ is
at least $N+1$, the length of $\arc{jx}$ ). Furthermore, 
$N = n-1$. If $k<0$, then relabel
using $i \mapsto 1-i$ and take the involution of the trigon interchanging the
$\alpha$ and $\gamma$ corners. We then have the desired $T(j,k,N)$.

So assume $\sigma$ is of Type I\!I. Then after possibly relabelling $\sigma$,
we can write $\sigma$ as in Figure~\ref{fig:findT2} where the width of the upper
extended Scharlemann cycle is at most that of the lower and where
either $r=r'=x$ or $s=s'=x$ (using Lemma~\ref{lem:existsS22} if $\sigma$
is not an $x$--edge cycle). 
\begin{figure}
\centering
\input{findT2.pstex_t}
\caption{}
\label{fig:findT2}
\end{figure}
The first condition gives that $k-(j+1) \leq
j-(-j+1)$, implying $k \leq 3j$ as desired. 
Let $N=\mbox{min}\{-j-s, r-(k+1)\}$. Then there are $N+1$ bigons beyond the
extended Scharlemann cycle on each
of the $\alpha \beta, \alpha \gamma$--arms of $\sigma$. As either
$r=r'$ or $s=s'$, there are $N$ bigons on the $\beta \gamma$--arm as well.
Thus $\sigma$ contains $T(j,k,N)$ as desired. Finally, note that $N=n-1$.

Note that, whether $\sigma$ is
Type I or I\!I, the only label that appears twice on $\sigma$
is the $x$--label. Thus the only label that can appear 
twice in the constructed $T(j,k,N)$ must be the extremal
ones on the ideal corner of $T$. As $j,k$ have the same 
parity, this is not possible.  
\end{proof}

\begin{remark}
In $T(j,k,N)$, $k,j$
have the same parity by the Parity Rule of Section~\ref{sec:graphcombo}. Thus $|j-k| \ge 2$.
\end{remark}

\begin{defn}\label{def:L2}
Let $\calL_\II = \{x| \Lambda_x$ has a trigon face which is of Type I or
Type I\!I, but $x \not\in L(\Sigma)\}$. Note that $x \not\in L(\Sigma)$ means that
$x$ is not a label in an extended Scharlemann cycle of length $2$ or $3$.
\end{defn} 

\begin{defn}\label{def:cornersI&II}
A subgraph $\sigma$ of $\Lambda$ is said to be of {\em Type $T(j,k,N)$} for 
some $j,k \in \N, N \in \N \cup \{0\}$ if after relabelling it becomes $T(j,k,N)$. For such
a $\sigma$, let $c(\sigma)$ be the corner of $\sigma$ (before relabelling)
corresponding to the ideal corner of $T(j,k,N)$. Let ${\mathfrak T} = \{ \sigma |
\sigma \, \hbox{is of Type}\, T(j,k,N), j,k,N \in \N \}$. Let $\Sigma_\II$ be a
smallest subcollection of $\mathfrak T$ such that for any $x \in \calL_\II$ there
is a $\sigma_x \in \Sigma_\II$ with $x \in c(\sigma_x)$. Lemma~\ref{lem:findT}
guarantees the existence of $\Sigma_\II$. For $\sigma \in
\Sigma_\II$, we denote by $j_\sigma, k_\sigma, N_\sigma$,
its corresponding parameters. Thus $|\calL_\II| \le \sum_{\sigma \in 
\Sigma_\II} (2N_\sigma+2)$.
\end{defn}

\begin{remark}\label{atmost2sigma2}
Note that if three label intervals share a common label, then one of these intervals must be contained in the union of the other two.
Thus the minimality of $\Sigma_\II$ guarantees that if
$x \in \calL_\II$ then $x \in c(\sigma)$ for at most $2$ elements $\sigma$
of $\Sigma_\II$. Remember that $c(\sigma)$ contains
all labels appearing on the corners of $\sigma$.
\end{remark}

\begin{defn}\label{def:partitionSII} Partition $\SII$ as follows:

$\SII(\I)=\{\sigma \in \Sigma_\II \, | \, \sigma \hbox{ is of Type I} \}$ and $\SII(\II)=\{\sigma \in \Sigma_\II \, | \, \sigma \hbox{ is of Type I\!I} \}$.

$\SII(\I,1)=\{\sigma \in \SII(\I) \, | \,
 N_\sigma 
< j_\sigma - k_\sigma\}$

$\SII(\I,2)=\{\sigma \in \SII(\I) \, | \,
 2k_\sigma-j_\sigma \leq 0
 \, ; \,   
   j_\sigma - k_\sigma \leq N_\sigma \leq   4j_\sigma - 2k_\sigma - 1 
  \}$

$\SII(\I,3)=\{\sigma \in \SII(\I) \, | \,
 2k_\sigma-j_\sigma > 0
 \, ; \, 
  j_\sigma - k_\sigma \leq N_\sigma <  k_\sigma 
 \}$

$\SII(\I,4)=\{\sigma \in \SII(\I) \, | \,
 2k_\sigma-j_\sigma > 0
 \, ; \, 
 k_\sigma \leq N_\sigma \leq   4j_\sigma-2k_\sigma-1 
  \}$

$\SII(\I,5)=\{\sigma \in \SII(\I) \, | \,
 2k_\sigma-j_\sigma > 0
  \, ; \, 
  4j_\sigma-2k_\sigma-1   < N_\sigma 
  \}$

$\SII(\I,6)=\{\sigma \in \SII(\I) \, | \,
 2k_\sigma-j_\sigma \leq 0
  \, ; \,  
4j_\sigma-2k_\sigma-1  < N_\sigma \}$


$\SII(\II,1)=\{\sigma \in \SII(\II) \, | \, 
N_\sigma < k_\sigma-j_\sigma \}$

$\SII(\II,2)=\{\sigma \in \SII(\II)  \, | \,
 2j_\sigma-k_\sigma \leq 0
 \, ; \, 
k_\sigma-j_\sigma \leq  N_\sigma  \leq  3k_\sigma-j_\sigma-1 
 \}$

$\SII(\II,3)=\{\sigma \in \SII(\II) \, | \,
 2j_\sigma-k_\sigma > 0
 \, ; \, 
  k_\sigma-j_\sigma \leq  N_\sigma  <  k_\sigma  \}$

$\SII(\II,4)=\{\sigma \in \SII(\II) \, | \,
 2j_\sigma-k_\sigma > 0 
 \, ; \, 
   k_\sigma  \leq  N_\sigma \leq   3k_\sigma-j_\sigma -1
 \}$

$\SII(\II,5)=\{\sigma \in \SII(\II) \, | \,
  2j_\sigma-k_\sigma > 0
  \, ; \,    3k_\sigma-j_\sigma-1 <  N_\sigma  \}$

$\SII(\II,6)=\{\sigma \in \SII(\II) \, | \,
 2j_\sigma-k_\sigma \leq 0
 \, ; \,  3k_\sigma-j_\sigma-1 <  N_\sigma  \}$

We will abbreviate $\SII(i,j)$ as $(i,j)$.
\end{defn}

One checks that Definition~\ref{def:partitionSII} does indeed partition $\SII$. That is,

\begin{lemma} 
$\SII(i,j) \cap \SII(r,s) \neq \emptyset$ iff $i=r, j=s$.
Furthermore, $\SII = \cup \SII(i,j)$. 
\end{lemma}

To each of the partition elements of Definition~\ref{def:partitionSII}
we associate collections of essential closed curves that will
eventually allow us to bound $\sum_{\SII} N_\sigma$ in Theorem~\ref{thm:Type12boundN} at the end of this section.

\begin{claim}\label{clm:j-k}
Let $T(j,k,N)$ corresponding to 
$\sigma \in \Sigma_\II$ be of Type I. 
Let $\epsilon$ be the maximal extended Scharlemann cycle for $\sigma$. 
For each 
$0 \leq s \leq \min\{N,j-k-1\}$  there is a configuration
$H(s)$ in $\hatF$ consisting of embedded curves $C(s),C'(s)$ in $a(\epsilon)$
connected by a path of three edges of $\sigma$. When $k \ne 0$, $C(s)$ and $C'(s)$ are 
disjoint, when $k=0$, they are identical. Furthermore:
\begin{enumerate}
\item If $H(s_1),H(s_2)$, $s_2>s_1$, share any vertices or edges then either
\begin{itemize}
\item $s_2=s_1+2k$, and $C(s_1)=C'(s_2)$ with no other shared vertices 
or edges; or
\item $s_1+s_2=2k-1$, and $C'(s_1)=C'(s_2)$ with no other shared vertices
or edges.
\end{itemize}
In particular,  
\[\calC_\sigma = \{C(s)| 0 \leq s \leq \min\{N,j-k-1\}\}\]
 is a collection of $\min\{N+1,j-k\}$ disjoint curves. Thus
\[
|\calC_\sigma| = 
\begin{cases}
N+1 & \mbox{ if } \sigma \in (\I,1) \\
j - k & \mbox{ if } \sigma \in (\I) - (\I,1).
\end{cases}
\]

\item If $H(s)$ shares a vertex with an element of $a(\epsilon)$ then
that element is either $C(s)$ or $C'(s)$.
\item No $H(s)$ lies in an annulus on $\hatF$.
\end{enumerate}
\end{claim}

\begin{proof}
It is convenient to work with $\sigma$ after it is relabelled as $T(j,k,N)$; 
however $H(s)$ is taken to be under the original labelling. 

For $0 \leq s \leq \min\{N,j-k-1\}$, let $C(s)$ be the essential curve  of $a(\epsilon)$ formed by the two $\edge{k+1+s,-k-s}$--edges in $\epsilon$ and let $C'(s)$ be the essential curve  of $a(\epsilon)$ formed by the two $\edge{k-s,-k+1+s}$--edges. 
$C(s)$ and $C'(s)$ share no edges or vertices as long as $k \ne 0$. Otherwise,
$C(s)$ and $C'(s)$ are the same. 
The distinct edges $e_1,e_2, e_3$ of Figure~\ref{fig:JohnTwo} now join $C(s)$ 
and $C'(s)$ on $\widehat{F}$. 
Define $H(s)$ to be the configuration on $\hatF$ given by $C(s), C'(s), e_1, e_2, e_3$ along with the vertices to which they are incident as shown in Figure~\ref{fig:JohnThree}. When $k=0$, $C(s)$ and $C'(s)$ are the same and the edges
$e_1$ and $e_3$ are incident to different vertices of $C(s)$ and on opposite
sides of $C(s)$. Thus when $k=0$, $H(s)$ cannot lie in an annulus on
$\hatF$. 
 
One checks that $H(s_1),H(s_2)$ share vertices or edges only as described in 
item (1) above. In particular, $C(s_1),C(s_2)$ are disjoint.  To check (2), note that for $H(s)$ the labels of $e_i$ are either those of $C(s)$ or $C'(s)$ or lie outside the labels of $\epsilon$.

\begin{figure}                                           
\centering 
\input{JohnTwo.pstex_t} 
\caption{}
\label{fig:JohnTwo}   
\end{figure} 

\begin{figure} 
\centering 
\input{JohnThree.pstex_t} 
\caption{}  
\label{fig:JohnThree} 
\end{figure}                   

Now fix $s$ and assume that $H(s)$ lies in an annulus on $\hatF$. 
As noted above, $k \neq 0$. Then
$C'(s)$ and $C(s)$ are different and cobound an annulus $B$ on 
$\hatF$. See Figure~\ref{fig:JohnThree}.         
                                                                                
As the edges $e_1,e_2, e_3$ lie in $B$, no vertex from $-j+1$ to $j$     
can lie in the interior of $B$. Let $A$ be the annulus contained in $A(\epsilon)$ with boundary                                                                
$C'(s) \cup C(s)$ obtained from the union of the appropriate bigons        
in $\epsilon$. Note the annuli $A$ and $B$ will have disjoint interiors (by (2)).       
                                                                                
The orientations of $C'(s)$ and $C(s)$ given by moving                     
along an edge from labels $\alpha$ to $\gamma$ must be in the same direction
along $B$ since  otherwise $A \cup B$ would form a Klein bottle.  
This forces the labeling of the edges in Figure~\ref{fig:JohnThree} to be as pictured. 
(This labeling is also forced because $k-s$ and $k+1+s$ have opposite parity.) 

Let $\calT$ be the torus $A \cup B$ which is necessarily separating.  
We will arrive at a contradiction by finding an embedded curve that intersects 
$\calT$ once.                       
                                                                                     
Let $\eta_1$ be the arc which is the corner of Figure~\ref{fig:JohnTwo} along 
$\alpha$ that runs between the labels $k+1+s$ and $j$ within the extended 
Scharlemann cycle $\epsilon$.  Noting that the curves $C(s+1), \dots, C(j-k-1)$ 
must lie outside of $B$ (since they are essential on $\widehat{F}$ and 
disjoint from $e_1,e_2,e_3$), we see that $\eta_1$
lies on one side of $\calT$.  In particular $\eta_1$ intersects $\calT$ only at 
the vertex $k+1+s$.                                                                       

Let $\eta_2$ be the arc that is the corner in Figure~\ref{fig:JohnTwo}  
along $\alpha$ running between labels $j+1+s$ and $j$.  When $s<k$ 
the vertices $k+1-s, \dots, k$ must lie outside of $B$, hence so do the 
edges of $T=T(j,k,N)$ which are incident to $\eta_2$ in Figure~\ref{fig:JohnTwo}. 
Therefore $\eta_2$ lies entirely on one side of $\calT$.  In particular,              
$\eta_2$ intersects $\calT$ only at the vertex $j+1+s$. When $s \geq k$, then 
the vertices $k+1-s, \dots, k$ may share vertex $-k+1+s$ with $B$ (this happens 
when $2k-1 \geq s \geq k$). But the edge in $T$ connecting $-k+1+s$ to 
$2k+j-s$ in $\eta_2$ must lie outside of $B$ on $G_F$ as it is incident to 
vertex $-k+1+s$ at label $\beta$ -- which from the labelling in 
Figure~\ref{fig:JohnThree} directs this edge outside of $B$. Again, we 
conclude that $\eta_2$ intersects $\calT$ only at vertex $j+1+s$.                          
  
The arcs $\eta_1$ and $\eta_2$ approach their endpoints at the vertices $k+1+s$ 
and $j+1+s$, respectively, in opposite directions along $K$.  Since the vertices $k+1+s$ 
and $j+1+s$ have the same parity and lie on 
$B \subset \widehat{F}$, $\eta_1$ and $\eta_2$ lie on opposite sides of $\calT$.   Hence $\eta_1 \cup \eta_2 \cup e_2 \cup e_3$ forms a loop that may be perturbed to transversely intersect $\calT$ just once.  
See Figure~\ref{fig:JohnThreeFive}.
\begin{figure}    
\centering        
\input{JohnThreeFive-test.pstex_t} 
\caption{} 
\label{fig:JohnThreeFive} 
\end{figure}                                                    
\end{proof}

\begin{claim}\label{clm:k-j}
Let $T(j,k,N)$ corresponding to 
$\sigma \in \Sigma_\II$ be of Type I\!I. Let $\epsilon,\epsilon'$ be the extended Scharlemann cycles for $\sigma$ with corners $\arc{-j+1,j},\arc{j+1,k}$ (resp.) in $T(j,k,N)$. 
For each $0 \leq s \leq \min\{N, \frac{1}{2}(k-j)-1\}$  there is a configuration $H(s)$ in $\hatF$ consisting of disjoint curves $C(s),C'(s)$ in $a(\epsilon),a(\epsilon')$ (resp.) connected by a path of three edges of $\sigma$. Furthermore:
\begin{enumerate}
\item $H(s_1),H(s_2)$ share no vertices or edges when $s_1 \neq s_2$.
In particular, 
\[\calC_\sigma = \{C(s), C'(s)| 0 \leq s \leq \min\{N, \frac{1}{2}(k-j)-1\}\}\]
 is a collection of $\min\{2N+2, k-j\} \geq \min\{N+1,k-j\}$ disjoint curves on $\hatF$. Thus 
\[
|\calC_\sigma| \geq
\begin{cases}
N + 1 & \mbox{ if } \sigma \in (\II,1) \\
k-j & \mbox{ if } \sigma \in (\II) - (\II,1).
\end{cases}
\]

\item If $H(s)$ shares a vertex with an element of $a(\epsilon)$ or $a(\epsilon')$ then that element is either $C(s)$ or $C'(s)$.
\item No $H(s)$ lies in an annulus on $\hatF$.
\end{enumerate}
\end{claim}

\begin{proof}
We work with $\sigma$ after it is relabelled as $T(j,k,N)$; 
however $H(s)$ is taken to be under the original labelling. 

\begin{figure} 
\centering  
\input{JohnFour.pstex_t} 
\caption{} 
\label{fig:JohnFour}                                                                                        
\end{figure} 

\begin{figure}  
\centering 
\input{JohnFive.pstex_t} 
\caption{}
\label{fig:JohnFive}
\end{figure}                                                                                                 

For $0 \leq s \leq  \min\{N, \frac{1}{2}(k-j)-1\}$, let $C(s)$ be the essential curve  of $a(\epsilon)$ formed by the two $\edge{j-s,-j+1+s}$--edges in $\epsilon$, and let $C'(s)$ be the essential curve  of $a(\epsilon')$ formed by the two $\edge{j+1+s, k-s}$--edges in $\epsilon'$.

Define $H(s)$ to be the connected subgraph of $\hatF$ which is the
union of $C(s)$,$C'(s)$, the edges $e_1,e_2,e_3$ of Figure~\ref{fig:JohnFour} and the vertices they connect as shown in Figure~\ref{fig:JohnFive}.
Note that each of these edges is in $T(j,k,N)$ since $s \leq  N$ and that they are all distinct.  
One now checks that $H(s_1),H(s_2)$ share no vertices or edges if $s_1 \neq s_2$.   
Also the two vertices of $H(s)$ not in $C(s) \cup C'(s)$ lie outside $a(\epsilon)$ and $a(\epsilon')$ as they are labeled $-j-s$ and $k+1+s$, which implies conclusion (2).

Assume some $H(s)$ lies in an annulus on $\hatF$. Then $C(s),C'(s)$ are parallel on $\hatF$, and Corollary~\ref{cor:noparallelesc2} implies that $\epsilon$ and $\epsilon'$ have the same core labels, 
a contradiction.
\end{proof}

\begin{lemma}\label{lem:differencekj}
\[
\sum_{\sigma \in \Sigma_\II} \min\{N_\sigma+1, |j_\sigma-k_\sigma|\} 
\leq 
\sum_{\sigma \in \Sigma_\II} |\calC_\sigma| 
\leq  
36(g-1)
\]
\end{lemma}

\begin{proof}
Let $\calC = \cup_{\sigma \in \Sigma_\II} \calC_\sigma$, with $\calC_\sigma$ as in 
Claims~\ref{clm:j-k}, ~\ref{clm:k-j}. 
By Claim~\ref{clm:j-k}, Claim~\ref{clm:k-j}, and Remark~\ref{atmost2sigma2}, 
any vertex belongs to at most two elements of $\calC$.

Partition $\calC$ as $\calC_0 \cup \calC_1 \cup \calC_2$ where 
$\calC_0= \{ C \in \calC : C$ is disjoint from all other elements of $\calC \}$,
$\calC_1 = \{C \in \calC :$ there exists $C' \in \calC$ such that $C$ 
and $C'$ meet at a single vertex \}, and
$\calC_2 = \{ C \in \calC :$ there exists $C' \in \calC, C' \neq C,$ 
such that $C$ and $C'$ meet at two vertices \}. Note that for $i \neq j$,
the elements of $\calC_i$ are disjoint from the elements of $\calC_j$.

Let $\calD_2$ be a subcollection of disjoint curves in $\calC_2$ with
$|\calD_2|=|\calC_2|/2$.

Let $\calD_1$ be the collection of simple closed curves on $\hatF$ 
obtained from the elements of $\calC_1$ by removing all non-transverse
points of intersection by small perturbations. Thus any two elements of 
$\calD_1$ either intersect transversely at a single point or are disjoint.
We may abstractly represent $\calD_1$  by a graph $\Gamma$ whose nodes 
correspond to the elements of $\calD_1$ and where two nodes are joined by an
edge if and only if the corresponding elements of $\calD_1$ intersect.  
Partition $\calD_1$ as $\calD_1^{(0)} \cup \calD_1^{(1)}$ where $\calD_1^{(0)}$
consists of the elements of $\calD_1$ that correspond to isolated nodes of
$\Gamma$. The elements of $\calD_1^{(1)}$ then correspond to components of
$\Gamma$ that are either paths with $k$ nodes, $k \geq 2$, or cycles with
$k$ nodes, $k \geq 3$. In each path (resp. cycle) component of $\Gamma$ with 
$k$ nodes, we can choose $\lfloor \frac{k+1}{2} \rfloor$ (resp. $\lfloor \frac{k}{2} \rfloor$) nodes
such that the corresponding elements of $\calD_1^{(1)}$ are disjoint. Let 
the resulting subcollection of $\calD_1^{(1)}$ be $\calE_1$. 

Let $\calF = \calC_0 \cup \calD_1^{(0)} \cup \calD_2$. Note that since
any element of $\calE_1$ has algebraic intersection number $1$ with
some element of $\calD_1^{(1)}$, no element of $\calE_1$ is parallel
on $\hatF$ to any element of $\calF$. Furthermore, if two elements of
$\calE_1$ are parallel on $\hatF$ then the corresponding nodes of $\Gamma$
belong to either a path component of $\Gamma$ with three nodes or a cycle
component with four nodes. Removing one element from each parallel pair
gives a subcollection $\calE$ of $\calE_1$ with $|\calE| \geq |\calD_1^{(1)}|/4$,
such that no element of $\calE$ is parallel to either another element of 
$\calE$ or an element of $\calF$. 

Let $C \in \calD_1^{(0)}$. Then there exists $C' \in \calC_1$ such that
$C$ and $C'$ intersect, non-transversely, at a single vertex. If 
$C \cup C'$ were contained in an annulus in $\hatF$ then by 
Corollary~\ref{cor:noparallelesc2} the corresponding core Scharlemann cycles would
have to have the same label pair and hence $C$ and $C'$ would share two
vertices, a contradiction. It follows that in any parallelism  class of 
elements of $\calF$ the corresponding elements of $\calC$ (i.e. before
perturbing to remove non-transverse intersections) are disjoint. 

We claim that no more than six elements of $\calF$ can be parallel on
$\hatF$. For suppose we have seven elements of $\calF$ that are parallel.
Let $C_r,C_l$ be the outermost elements of this group. Let $C_1, \dots, C_5$ be elements
of $\calC$ between these. Let $C_i \in \calC_{\sigma_i} \subset \calC$ and $C_i \subset H(s_i)$. In the notation of Claims~\ref{clm:j-k},~\ref{clm:k-j}: 
if $\sigma_i$ is of Type I, then $C_i=C(s_i)$; if $\sigma_i$
is of Type I\!I, then $C_i$ is $C(s_i)$ or $C'(s_i)$.
Then for each $i$, $H(s_i)$ must share a vertex with either $C_r$ or $C_l$, else 
$H(s_i)$ lies in an annulus, contradicting Claims~\ref{clm:j-k},~\ref{clm:k-j}. 
As $C_i$ is disjoint
from $C_r, C_l$, we can assign to each $i=1, \dots, 5$ a vertex $v_i \in H(s_i)$ which
lies in $C_r$ or $C_l$ and is the closest such to $C_i$ in the path metric of $H(s_i)$.
As $C_r$ and $C_l$ each contain only two vertices, it must be that, without loss of generality, $v_1=v_2 \in 
C_r$. Let $C_r \in \calC_{\sigma_r}$ with $C_r \subset H(s_r)$. 
By Claims~\ref{clm:j-k},~\ref{clm:k-j}, $\sigma_r \neq \sigma_1,\sigma_2$.
(Say $\sigma_1=\sigma_r$. Then $v_1 \in H(s_1) \cap H(s_r)$, implying that
 $\{C_1,C_r\}=\{C(s_1),C'(s_1)\}$ are parallel. This and the definition of 
$v_1$ imply that 
$H(s_1)$ lies in an annulus.) 
Then Remark~\ref{atmost2sigma2} implies that $\sigma_1
=\sigma_2$. By Claims~\ref{clm:j-k},~\ref{clm:k-j}, it must be that either 
$\sigma_1$
is of Type I and $v_1=v_2 \in C'(s_1)=C'(s_2)$; or that $\sigma_1$ is of Type I\!I
and $H(s_1)=H(s_2)$. In either case, it must be that one of $H(s_1)$
or $H(s_2)$, say $H(s_1)$, intersects $C_l$ in a vertex closer to $C_1$ in $H(s_1)$ 
than $v_1$, a contradiction.

Hence there exists a subcollection $\calF'$ of $\calF$ with 
$|\calF'| \geq |\calF|/6$ such that no two elements of $\calF'$ are
parallel. Then $\calE \cup \calF'$ is a collection of disjoint, essential,
non-parallel simple closed curves on $\hatF$. Hence $|\calE| + |\calF'|
\leq 3(g-1)$. Therefore
\begin{align*}
36(g-1) &\ge 12|\calE| + 12|\calF'| \\  
        &\ge 3|\calD_1^{(1)}| + 2 |\calF| \\  
        &= 3|\calD_1^{(1)}| + 2 (|\calC_0| + |\calD_1^{(0)}| + |\calD_2|) \\  
        &\ge 2 |\calC_0| + 2 |\calC_1| + |\calC_2| \\
        &\ge |\calC_0| + |\calC_1| + |\calC_2| = |\calC|
\end{align*}

as claimed.
\end{proof}

\begin{claim}\label{clm:2k-j}
Let $T(j,k,N)$ corresponding to 
$\sigma \in \Sigma_\II$ be of Type I such that $N \geq j-k$ and $2k-j>0$. 
Let $\epsilon$ be the maximal
extended Scharlemann cycle for $\sigma$. For each 
$0 \leq s \leq \min\{N-(j-k), 2k-j-1\}$  there is a configuration
$H(s)$ in $\hatF$ consisting of disjoint curves $C(s),C'(s)$ in $a(\epsilon)$
connected by a path of four edges of $\sigma$. 
\begin{enumerate}
\item If $H(s_1),H(s_2)$, $s_2>s_1$, share any vertices or edges then $s_2=s_1+j-k$. 
In this case $C(s_1)=C'(s_2)$ and the edge $e_1$ of $H(s_1)$ is the edge $e_4$ of $H(s_2)$. 
Besides these edges and the vertices to which they are incident, there are no other shared vertices or edges between $H(s_1),H(s_2)$. 
In particular,  
\[\calB_\sigma = \{C(s)| 0 \leq s \leq \min\{N-(j-k),2k-j-1\}\}\]
 is a collection of $\min\{N-(j-k)+1,2k-j \}$ disjoint curves. Thus 
\[
|\calB_\sigma|=
\begin{cases}
N - (j-k)+1 & \mbox{ if } \sigma \in (\I,3) \\
2k-j     & \mbox{ if } \sigma \in (\I,4) \cup (\I,5).
\end{cases}
\]
\item Each vertex of $G_F$ lies in at most one element of $\calB_\sigma$.
\item If $H(s)$ shares a vertex with an element of $a(\epsilon)$ then
that element is either $C(s)$ or $C(s')$.
\item No $H(s)$ lies in an annulus on $\hatF$.
\end{enumerate}
\end{claim}

\begin{proof}
We work with $\sigma$ after it is relabelled as $T(j,k,N)$; 
however $H(s)$ is taken to be under the original labelling. 
Note that as $\sigma$ is Type I, $j>k$.

For $0 \leq s \leq \min\{N-(j-k), 2k-j-1\}$ let $C(s)$  be the 
essential curve of $a(\epsilon)$ on $\widehat{F}$ formed by the two
$\edge{2k-j-s,-2k+j+1+s}$--edges in $\epsilon$ (note that 
$\max\{1,k-N\} \leq 2k-j-s \leq 2k-j=k-(j-k)<k$).
Let $C'(s)$ be the essential curve of $a(\epsilon)$ formed by the two            
$\edge{k-s,-k+s+1}$--edges in $\epsilon$ 
(note that $3 \leq \max\{j-k+1,j-N\} \leq k-s \leq k\}$). 
The four edges $e_1, e_2, e_3, e_4$ shown in 
Figure~\ref{fig:trigonwithlongarmsdetail}, along with $C(s),C'(s)$, give 
the configuration, $H(s)$, in the 
graph $G_F$ depicted in Figure~\ref{fig:nicelyconnected}. The vertices these
edges connect are included in $H(s)$. 

Assume $H(s_1),H(s_2), s_1<s_2$ 
share vertices. This can only happen if $s_2=s_1+j-k$ and the shared vertices
and edges are as claimed. 
In particular $C(s_2)$ shares no vertices (hence no edges) with $C(s_1)$.
This verifies conclusions $(1)$ and $(2)$. As the edges 
$e_1, e_2, e_3, e_4$ lie outside of $\epsilon$, conclusion $(3)$ holds.

\begin{figure}
\centering
\input{trigonwithlongarmsdetail.pstex_t}
\caption{}
\label{fig:trigonwithlongarmsdetail}
\end{figure}

\begin{figure}
\centering
\input{nicelyconnectedV2.pstex_t}
\caption{}
\label{fig:nicelyconnected}
\end{figure}

Assume for contradiction that $H(s)$ lies in an annulus in $\hatF$. Let
$B$ be the annulus in $\hatF$ cobounded by $C(s),C'(s)$.
Let $A$ be the annulus contained in $A(\epsilon)$ with boundary 
$C(s) \cup C'(s)$ obtained from the union of the bigons in $\epsilon$ with 
labels $-j+2k-s, -j+2k+1-s, \dots, k-s$ and $j-2k+1+s, j-2k+s, \dots, 1-k+s$.  
The labeling of Figure~\ref{fig:makeKlein} shows that the induced orientation on $\partial A$
does not cancel an orientation on $\partial B$.
Note that in Figure~\ref{fig:makeKlein} we use that $j-k$ is even, and hence that vertices $-j+2k-s$ and $k-s$ are parallel.
Thus $A \cup B$ is a Klein bottle.  This is a contradiction.
\begin{figure}
\centering
\input{makeKlein.pstex_t}
\caption{}
\label{fig:makeKlein}
\end{figure}
\end{proof}

\begin{claim}\label{clm:Type1j}
Let $T=T(j,k,N)$ corresponding to $\sigma \in \Sigma_\II$ be of Type I\!I such that $N > k-j$ and $2j-k>0$. 
Let $\epsilon, \epsilon'$ be the extended Scharlemann cycles for $\sigma$ with corners $\arc{-j+1,j},\arc{j+1,k}$ (resp.) in $T(j,k,N)$.
 For each  $0 \leq s \leq \min\{N-(k-j), j-1\}$  there is a configuration
$H(s)$ in $\hatF$ consisting of disjoint curves $C(s),C'(s)$ in $a(\epsilon)$
connected by a path of four edges of $\sigma$. Furthermore:
\begin{enumerate}
\item If $C(s_1)$ shares vertices with $H(s_2)$, $s_1 \neq s_2$, then 
$C(s_1)=C'(s_2)$ and $C(s_1),H(s_2)$ share no other vertices or edges.
In particular,  
\[\calB_\sigma = \{C(s)| 0 \leq s \leq  \min\{N-(k-j),j-1\}\}\] 
is a collection of $\min\{N-(k-j)+1,j \}$ disjoint curves.
Thus 
\[
|\calB_\sigma|=
\begin{cases}
N - (k-j)+1 & \mbox{ if } \sigma \in (\II,3) \\
j                & \mbox{ if } \sigma \in (\II,4) \cup (\II,5).
\end{cases}
\]
\item Each vertex of $G_F$ lies in at most one element of $\calB_\sigma$.
\item If $H(s)$ shares a vertex with an element of $a(\epsilon)$ or $a(\epsilon')$ then that element is either $C(s)$ or $C(s')$. 
\item No $H(s)$ lies in an annulus on $\hatF$.
\end{enumerate}
\end{claim}

\begin{proof}
We work with $\sigma$ after it is relabelled as $T(j,k,N)$; 
however $H(s)$ is taken to be under the original labelling. 
Note that as $\sigma$ is Type I\!I, $j<k \leq 3j$.  See Figure~\ref{fig:trigonwithlongarmsdetailII}.

\begin{figure}
\centering
\input{trigonwithlongarmsdetailII.pstex_t}
\caption{}
\label{fig:trigonwithlongarmsdetailII}
\end{figure}

For $0 \leq s \leq \min\{N-(k-j), j-1\}$  let $C(s)$  be the essential curve of $a(\epsilon)$ on $\widehat{F}$ formed by the two $\edge{j-s,-j+1+s}$--edges in $\epsilon$ (note that 
$1 \leq \max\{k-N,1\} \leq j-s \leq j$).
Let $C'(s)$ be the essential curve of $a(\epsilon)$ formed by the two            
$\edge{2j-k-s,-2j+k+s+1}$--edges in $\epsilon$ 
(note that $-j+1 < \max\{1-k+j,j-N\} \leq 2j-k-s \leq 2j-k < j \}$). 
The four edges $e_1, e_2, e_3, e_4$ shown in 
Figure~\ref{fig:trigonwithlongarmsdetailII}, along with $C(s),C'(s)$, give 
a configuration, $H(s)$, in the graph $G_F$. The vertices these
edges connect are included in $H(s)$. Note that these seven vertices of 
$H(s)$ are distinct ($k-j$ is even). 

Assume $C(s_1)$ shares vertices with $H(s_1)$, $s_1 \neq s_2$. As the vertices
of $C(s_1)$ are labelled between $-j+1$ and $j$ and the vertices
$j+k+1-(j-s_2), 2k+1-(j-s_2),-j-k+(j-s_2)$ are outside this range, $C(s_1)$
can only share vertices with $C(s_2)$ or $C'(s_2)$. As these are curves
of $a(\epsilon)$, either $C(s_1)=C(s_2)$ or $C(s_1)=C'(s_2)$. As
$1 \leq j-s \leq j$ for each $s$, it must be that $C(s_1)=C'(s_2)$
(corresponding to $s_2-s_1=j-k$ or $s_2+s_1=3j-k-1$). This proves parts $(1)$ and $(2)$ of the claim.

For parts $(3)$ and $(4)$, we argue as for $(3)$ and $(4)$ of Claim~\ref{clm:2k-j}.
\end{proof}

\begin{lemma}\label{clm:boundCB}
For $\sigma \not\in \cup_{i = \I,\II; j=3,4,5} \, \Sigma_\II(i,j)$ define $\calB_\sigma = \emptyset$.  Then:
\[ \sum_{\sigma \in \Sigma_\II} |\calB_\sigma| \leq 24(g-1)\]
\end{lemma}

\begin{proof}
Consider the collection $\calB$ of all elements of $\calB_\sigma$ taken  (without equating elements) over all $\sigma \in \Sigma_\II$ (and hence over all $\sigma \in \cup_{i = \I,\II; j=3,4,5} \, \Sigma_\II(i,j)$). 
By Remark~\ref{atmost2sigma2}, each vertex of $G_F$ is a label in at most two elements of $\Sigma_\II$. On the other hand,
for a fixed $\sigma$ a vertex of $G_F$ lies in at most one
element of $\calB_\sigma$. 
Thus any vertex belongs to at most two elements of $\calB$. 

Partition $\calB$ as $\calB_0 \cup \calB_1 \cup \calB_2$ as in the proof of
Lemma~\ref{lem:differencekj}. Again as in the proof of that lemma, we get 
disjoint collections of curves $\calE$ and $\calF$, where no element
of $\calE$ is parallel on $\hatF$ to either another element of $\calE$
or an element of $\calF$, and any pair of parallel elements of $\calF$ 
are (before perturbation) disjoint.

We claim that at most four elements of $\calF$ can be parallel on $\hatF$. 
For suppose we had five such parallel curves.
They all come from extended 
Scharlemann cycles of length $2$ within elements of $\Sigma_\II$. 
Consequently, these five parallel, disjoint curves must arise from 
at most two different elements of $\cup_{i=\I,\II;j=3,4,5} \SII(i,j)$. Otherwise, by 
Corollary~\ref{cor:noparallelesc2}, the maximal extended Scharlemann cycles
in each of three different elements $\sigma_1,\sigma_2,\sigma_3$ of $\cup_{i=\I,\II;j=3,4,5} \SII(i,j)$ must
have the same core Scharlemann cycles. This would imply that the ideal
corners of these elements would overlap, contradicting the minimality of 
$\Sigma_\II$ (see Remark~\ref{atmost2sigma2}).

Thus three of these five parallel curves, $c_1,c_2,c_3$, all belong
to $\calB_\sigma$ for a single $\sigma \in \cup_{i=\I,\II;j=3,4,5} \SII(i,j)$. 
They all belong to the
same extended Scharlemann cycle in $\sigma$. Take $c_2$ to be between $c_1,c_3$. Then $c_2 \subset H(s)$ coming from
$\sigma$ as described in Claims~\ref{clm:2k-j} and ~\ref{clm:Type1j}.
If $H(s)$ is disjoint from both $c_1$ and $c_3$, then $H(s)$ lies
in the annulus between $c_1,c_3$, contradicting these Claims. So we
may assume that $H(s)$ intersects, say, $c_1$. But then, again by
these Claims, $c_1,c_2$ must form the curves $C(s),C'(s)$ of
$H(s)$. (In these Claims, if $H(s)$ intersects an element of an 
extended Scharlemann cycle in $\sigma$, then that element is
$C(s)$ or $C'(s)$.) Then $H(s)$ lies in the annulus between $c_1,c_2$,
contradicting these Claims. (If $H(s)$ intersected $c_3$ the same
argument applied to $c_3$ would equate it with $c_1$.)

We thus get a subcollection $\calF'$ of $\calF$ such that no two elements
of $\calF'$ are parallel and $|\calF'| \geq |\calF|/4$. Therefore
$|\calE| + |\calF'| \leq 3(g-1)$. Hence
\begin{align*}
24(g-1) &\ge 8|\calE| + 8|\calF'| \\
        &\ge 2|\calD_1^{(1)}| + 2|\calF| \\
        &= 2|\calB_0| + 2|\calB_1| + |\calB_2| \\
        &\ge |\calB|
\end{align*}

as claimed.
\end{proof}

\begin{remark}
Lemmas~\ref{lem:differencekj} and ~\ref{clm:boundCB} linearly bound
$\sum N_\sigma$ where this sum is taken over 
$\Sigma_\II - ((\I,5) \cup (\I,6) \cup (\II,5) \cup (\II,6))$. These bounds
come from the collections $\calC_\sigma$ and $\calB_{\sigma}$.
In a similar way, the following collections $\calA_{\sigma}$ will
be used to linearly bound $\sum N_{\sigma}$ where the sum is taken over 
the remaining $(\I,5) \cup (\I,6) \cup (\II,5) \cup (\II,6)$.
\end{remark}

\begin{claim}\label{clm:TypeIIesscurves}
Let $T(j,k,N)$ corresponding to 
$\sigma \in \Sigma_\II$ be of Type I such that $N  > 4j-2k-1$ (so that $\sigma \in (\I,5) \cup (\I,6)$). 
For each $0 \leq s \leq N-2j$  there is an essential curve $\cala(s)$ in $\hatF$ consisting of 6 distinct edges and vertices. 
Furthermore, $\cala(s_1)$ intersects $\cala(s_2)$, $s_2 > s_1$, if and only if  $s_2 - s_1$ is an element of $\{j-k, 2j, k+j\}$. 

Indeed, assume $s_2 - s_1$ is an element of $\{j-k, 2j, k+j\}$. Either
\begin{enumerate}
\item $k \neq 0$ and $\cala(s_1),\cala(s_2)$ share exactly two vertices and a single edge spanning them. 
After perturbing to be transverse along their common edge,  $\cala(s_1)$ and $\cala(s_2)$ intersect transversely once. Or
\item $k=0$ and $\cala(s_1), \cala(s_2)$ intersect in exactly two edges and the four vertices they connect. In this case $\cala(s_1),\cala(s_2)$ can be perturbed along these edges so that they are transverse and intersect algebraically, geometrically twice.
\end{enumerate}
 
As $N-2j+1 \geq 2(j-k)$, for each $s$ there is an $s'$ such that $|s-s'|=j-k$. In particular, each $\cala(s)$ is essential in $\hatF$.

Define $\calA_\sigma=\{\cala(s)|1 \leq s \leq N-2j\}$. Then the elements 
of $\calB_{\sigma}$ are disjoint from the elements of $\calA_{\sigma}$.
Furthemore, each vertex of $G_F$ belongs to at most two elements of 
$\calA_{\sigma}$.
\end{claim}

\begin{proof}

\begin{figure}
\centering
\input{smalltypeItrigon.pstex_t}
\caption{}
\label{fig:smallTypeItrigon}
\end{figure}

We work with $\sigma$ after it is relabelled as $T=T(j,k,N)$; 
however $\cala(s)$ is taken to be under the original labelling. 

$T$ is shown in Figure~\ref{fig:smallTypeItrigon}.  Let $b_0$, $c_0$, $d_0$, 
$e_0$, $f_0$, and $g_0$ be the edges pictured in 
Figure~\ref{fig:smallTypeItrigon}.
Let $\cala(0)$ be the curve obtained by taking their union. 
For $0 \leq s \leq N-2j$ 
we may take the corresponding edges $b_s$, $c_s$, $d_s$, $e_s$, $f_s$, and $g_s$ 
as we move out the arms of $T$ and form the curve 
$\cala(s)$ by taking their union.  Thus we obtain a sequence of curves 
$\cala(0), \dots, \cala(N-2j)$ on $\widehat{F}$. Note that
$N \geq 2(j+(j-k))-1>2j$.

If $\cala(s_1)$ and $\cala(s_2)$ intersect on the Heegaard surface 
then they intersect in vertices and possibly coincide along edges.

The labels of the vertices of $\cala(s)$ are $3j+s$, $k+2j+s$, $j+s$, $-j+1-s$, $k-2j+1-s$, and $-3j+1-s$ in decreasing order.  
Each of these vertices is different, so $\cala(s)$ is embedded. 
Recall $N-2j \geq s \geq 0$.  As no label appears twice along a corner of 
$T$, the vertices of $\cala(s_1)$ and $\cala(s_2)$ overlap if and only if:
\begin{enumerate}
\item $s_2 - s_1 = j-k$; $k+2j + s_2 = 3j+s_1$ and $-j-s_2+1 = k-2j-s_1+1$
\item $s_2 - s_1 = 2j$; $j+s_2 = 3j + s_1$ and $-j-s_2+1 = -3j-s_1+1$
\item $s_2 - s_1 = j+k$; $j+s_2 = k+ 2j+s_1$ and $k-2j-s_2+1 = -3j-s_1+1$.
\end{enumerate}

Figure~\ref{fig:smallTypeIGF} shows $\cala(s)$ as it lies on $\widehat{F}$. 
If $\cala(s_1), \cala(s_2)$ intersect then they
agree exactly along a single edge (note in Figure~\ref{fig:smallTypeItrigon} that the sum of the labels at an edge determines the
arm in which the edge lies, and no label appears twice along a corner of $T$). 
If $k \ne 0$, then $\cala(s_1)$ and $\cala(s_2)$ can be perturbed so that 
they intersect exactly once (the $\alpha, \beta, \gamma$ labels 
that are not labels of edges of $\cala(s_1)$, say, 
alternately appear on one side or the other of $\cala(s_1)$).   
If $k=0$, then $(1)$ and $(3)$ both hold ($e_{s_1}=b_{s_2}$ and $f_{s_2}=c_{s_1}$)
and one similarly checks that $\cala(s_1),\cala(s_2)$ intersect algebraically
and geometrically twice after perturbation. In any case, each vertex of $G_F$
belongs to at most two elements of $\calA_{\sigma}$.

The edges constituting elements of $\calB_{\sigma}$ come from the extended
Scharlemann cycle of $\sigma$, while the edges of an element of $\calA_{\sigma}$ lie outside the extended Scharlemann cycle. (Note $\cala(0) \not \in \calA_{\sigma}$.) This verifies 
the last sentence of the Claim.
\end{proof}

\begin{figure}
\centering
\input{smalltypeIGF.pstex_t}
\caption{}
\label{fig:smallTypeIGF}
\end{figure}

\begin{claim}\label{clm:TypeIesscurves}
Let $T(j,k,N)$ corresponding to 
$\sigma \in \Sigma_\II$ be of Type I\!I such that $N  > 3k-j-1$ (so that $\sigma \in (\II,5) \cup (\II,6)$). 
For each $0 \leq s \leq N-k-j$  there is an essential curve $\cala(s)$ in $\hatF$ consisting of 6 distinct edges and vertices. 
Furthermore, $\cala(s_1)$ intersects $\cala(s_2)$, $s_2 > s_1$, if and only if  $s_2 - s_1$ is an element of $\{k-j, 2j, k+j\}$. 

Indeed, assume $s_2 - s_1$ is an element of $\{k-j, 2j, k+j\}$.
Either
\begin{enumerate}
\item $k \neq 3j$ and $\cala(s_1),\cala(s_2)$ share exactly two vertices and a single edge spanning them. 
Furthermore, after perturbing to be transverse along their common edge,  $\cala(s_1)$ and $\cala(s_2)$ intersect transversely once. Or
\item $k=3j$ and $\cala(s_1), \cala(s_2)$ intersect in exactly two edges and the four vertices they connect. In this case $\cala(s_1),\cala(s_2)$ can be perturbed along these edges so that they are transverse and intersect algebraically, geometrically twice.
\end{enumerate}

As $N-j-k+1 \geq 2(k-j)$, for each $s$ there is an $s'$ such that $|s-s'|=k-j$. In particular, each $\cala(s)$ is essential in $\hatF$.

Define $\calA_\sigma=\{\cala(s)|1 \leq s \leq N-k-j\}$. Then  
the elements of $\calA_{\sigma}$ are disjoint from those of 
$\calB_{\sigma}$. Furthermore, a vertex of $G_F$ belongs to at most two
elements of $\calA_\sigma$.
\end{claim}

\begin{proof}

\begin{figure}
\centering
\input{smalltypeIItrigon.pstex_t}
\caption{}
\label{fig:smallTypeIItrigon}
\end{figure}

We work with $\sigma$ after it is relabelled as $T=T(j,k,N)$; 
however $\cala(s)$ is taken to be under the original labelling. 

$T$ is shown in Figure~\ref{fig:smallTypeIItrigon}.  Let $b_0$, $c_0$, $d_0$, 
$e_0$, $f_0$, and $g_0$ be the edges pictured in 
Figure~\ref{fig:smallTypeIItrigon}.
Let $\cala(0)$ be the curve obtained by taking their union. 
For $0 \leq s \leq N-k-j$ 
we may take the corresponding edges $b_s$, $c_s$, $d_s$, $e_s$, $f_s$, and $g_s$ 
as we move out the arms of $T$ and form the curve 
$\cala(s)$ in $\hatF$ by taking their union.  Thus we obtain a sequence of curves 
$\cala(0), \dots, \cala(N-k-j)$ on $\widehat{F}$. Note that
$N \geq k+j$.

If $\cala(s_1)$ and $\cala(s_2)$ intersect on the Heegaard surface 
then they intersect in vertices and possibly coincide along edges.

The vertices of $\cala(s)$ are $2k+j+1+s$, $k+2j+1+s$, $k+1+s$, $-j-s$, $-k-s$, and $-k-2j-s$. 
These vertices are distinct and $\cala(s)$ is embedded.   
Furthermore, the vertices of $\cala(s_1)$ and $\cala(s_2)$ overlap if and only if:
\begin{enumerate}
\item $s_2 - s_1 = k-j$; $k+2j+1+s_2 = 2k+ j + 1+ s_1$ and $-j-s_2 = -k-s_1$
\item $s_2 - s_1 = 2j$; $k+1+s_2=k+2j+1+s_1$ and $-k-s_2 = -k-2j-s_1$
\item $s_2 - s_1 = k+j$; $k+1+s_2 = 2k+j+1+s_1$ and $-j-s_2=-k-2j-s_1$.
\end{enumerate}

Figure~\ref{fig:smallTypeIIGF} shows $\cala(s)$ as it lies on $\widehat{F}$.  
As before one sees that if $\cala(s_1),\cala(s_2)$ intersect, then they must coincide along an edge. They can then be perturbed 
so that they intersect exactly once -- except when $k=3j$ when (1) and 
(2) both hold.  If $k=3j$ one 
checks that $\cala(s_1)$ and $\cala(s_2)$ intersect algebraically and geometrically twice after perturbation. In any case, a vertex of $G_F$ belongs to at most
two elements of $\calA_\sigma$.

The edges constituting elements of $\calB_{\sigma}$ come from the extended
Scharlemann cycles of $\sigma$, while the edges of an element of $\calA_{\sigma}$ lie outside the extended Scharlemann cycles.  This verifies the last sentence of the Claim.
\end{proof}

\begin{figure}
\centering
\input{smalltypeIIGF.pstex_t}
\caption{}
\label{fig:smallTypeIIGF}
\end{figure}

\begin{lemma}\label{lem:curvestogether}
For $\displaystyle \sigma \not\in \bigcup_{\overset{i = \I,\II}{j=3,4,5}} \Sigma_\II(i,j)$ define $\calB_\sigma = \emptyset$. For $\displaystyle \sigma \not\in \bigcup_{\overset{i = \I,\II}{ j=5,6}} \, \Sigma_\II(i,j)$ define $\calA_\sigma = \emptyset$.  Then:
\[ \sum_{\sigma \in \SII} (|\calA_\sigma|+|\calB_\sigma|) \leq 240(g-1) \]
\end{lemma}

\begin{proof}
Consider the collection $\mathcal D$ of all elements of $\calA_{\sigma}, \calB_{\sigma}$ taken (without equating elements) over all 
$\sigma \in \cup_{i=1,2;j=3,4,5} \SII(i,j)$. 
Let $\cala_\sigma$ be an element of $\calA_\sigma$. Then $\cala_\sigma$ is disjoint
from any element of $\calB_\sigma$, meets at most three other elements
of $\calA_\sigma$, and each vertex of $\cala_\sigma$ belongs to at most one
other element of $\calA_\sigma$. Also, any two elements of $\calB_\sigma$ are
disjoint. By Remark~\ref{atmost2sigma2}, each vertex of $G_F$ is a label
in at most two elements of $\Sigma_\II$. It follows that each element of $\calD$
meets at most $3+6 \cdot 2=15$ other elements of $\calD$, the maximum occurring
when an element $\cala_\sigma \in \calA_\sigma$ meets three other elements of 
$\calA_\sigma$, and for each vertex $v_i$ of $\cala_\sigma$ there exists 
$\sigma_i \in \Sigma_\II, \sigma_i \neq \sigma,$ such that $v_i$ belongs 
to two elements of $\calA_{\sigma_i}, 1 \leq i \leq 6$.
That is we find at least 
$(1/16)\sum_{\sigma \in \Sigma_\II}(|\calA_\sigma|+|\calB_\sigma|)$ disjoint
curves coming from $\mathcal D$.   We need to show that there are at most $5(3g-3)$ such curves.  

Assume for contradiction that $(1/16)\sum_{\sigma \in \Sigma_\II}(|\calA_\sigma|+|\calB_\sigma|)>5(3g-3)$.
Then there must be a collection $\mathcal E$ of six disjoint curves
of $\mathcal D$ that are isotopic on $\hatF$.

\begin{claim}\label{clm:noA}
No element of $\mathcal E$ belongs to $\calA_\sigma$ for $\sigma \in 
\Sigma_\II$.
\end{claim}

\begin{proof}
Assume there were an $\cala_\sigma(s) \in {\mathcal E}$ for some 
$\sigma$ and $s$ as in Claims~\ref{clm:TypeIIesscurves}, ~\ref{clm:TypeIesscurves}. By the same claims, we can find another element  $\cala_\sigma(s')$ 
coming from the
same $\calA_\sigma$ that 
shares at least one edge with $\cala_\sigma(s)$ and which can be perturbed to intersect 
$\cala_\sigma(s)$
algebraically once or twice. As $\cala_\sigma(s')$ must share at least one vertex with 
any element of
$\mathcal E$ that it intersects, one of the five curves of 
$\calE - \{\cala_\sigma(s)\}$ 
must be disjoint from $\cala_\sigma(s')$. But all of these isotopic curves must have 
algebraic intersection number at least $1$ with $\cala_\sigma(s')$.
\end{proof}

\begin{claim}\label{clm:aeps} 
For some $\sigma \in \Sigma_\II$ there are three elements
of $\mathcal E$ that belong to $\calB_\sigma$. 
Furthermore there is an extended Scharlemann cycle $\epsilon$ in
$\sigma$ such that each of these three belong $a(\epsilon)$.
\end{claim}

\begin{proof}
By Claim~\ref{clm:noA}, all the elements of $\mathcal E$ come from
$\cup \calB_{\sigma}$. Consequently they all come from extended 
Scharlemann cycles of length $2$ within elements of $\Sigma_\II$. 
By Corollary~\ref{cor:noparallelesc2}, these extended Scharlemann
cycles must have the same core labels. As each of these core labels
will lie in the ideal corners of the corresponding elements of $\Sigma_\II$, Remark~\ref{atmost2sigma2} implies that these six 
parallel, disjoint curves in $\mathcal E$ belong to
at most two different elements of $\Sigma_\II$. Hence three belong
to the same $\sigma \in \Sigma_\II$. Again by 
Corollary~\ref{cor:noparallelesc2} they must belong to the same extended
Scharlemann cycle within $\sigma$.
\end{proof}

Let $c_1,c_2,c_3$ be three parallel curves of Claim~\ref{clm:aeps}
with $c_2$ between $c_1,c_3$. Then $c_2 \subset H(s)$ coming from
$\sigma$ as described in Claims~\ref{clm:2k-j} and ~\ref{clm:Type1j}.
If $H(s)$ is disjoint from both $c_1$ and $c_3$, then $H(s)$ lies
in the annulus between $c_1,c_3$, contradicting these Claims. So we
may assume that $H(s)$ intersects, say, $c_1$. But then, again by
these Claims, $c_1,c_2$ must form the curves $C(s),C'(s)$ of
$H(s)$. (In these Claims, if $H(s)$ intersects an element of an 
extended Scharlemann cycle in $\sigma$, then that element is
$C(s)$ or $C'(s)$.) Then $H(s)$ lies in the annulus between $c_1,c_2$,
contradicting these Claims. (If $H(s)$ intersected $c_3$ the same
argument applied to $c_3$ would equate it with $c_1$.)
\end{proof}
 
\begin{remark}
After 
Lemmas~\ref{lem:differencekj} and ~\ref{clm:boundCB}, to get a linear
bound on $\sum_{\Sigma_\II} N_{\sigma}$ we need only bound 
$\sum_{(\II,5),(\II,6)} (N_\sigma - 2j_\sigma)
+\sum_{(\I,5),(\I,6)}(N_\sigma - (j_\sigma + k_\sigma))$. This can be done as in the proof of Lemma~\ref{lem:curvestogether} using only the collections $\calA_\sigma$. By including the $\calB_\sigma$ above, we get a 
slightly improved bound.
\end{remark}

\begin{thm}\label{thm:Type12boundN}

\[|\calL_\II|/2 \le \sum_{\sigma \in \Sigma_\II} (N_\sigma+1) \leq 480(g-1)\]

That is, there are at most 960(g-1) labels $x$ in $G_Q$ such that 
$\Lambda_x$ has a trigon face which is of Type I or
Type I\!I but $x \not\in L(\Sigma)$. 
\end{thm}

\begin{proof}
The first inequality comes from Definition~\ref{def:cornersI&II}, so we focus on the second which follows from:  
\[
\begin{array}{rlllllll}
\sum_{\sigma \in \SII} (N_\sigma +1) \leq &6\sum_{\sigma \in \SII} |\calC_\sigma| &+ &\sum_{\sigma \in \SII} |\calB_\sigma| &+& \sum_{\sigma \in \SII} (|\calB_\sigma| + |\calA_\sigma|) & (\ast)\\
\leq &6 \cdot 36(g-1) &+ &24(g-1) &+ &240(g-1) & (\ast\ast)\\
= &480(g-1)
\end{array}
\] 
Inequality $(\ast\ast)$ comes from Lemma~\ref{lem:curvestogether},
Lemma~\ref{lem:differencekj}, and Lemma~\ref{clm:boundCB}.
For inequality $(\ast)$ we take the definitions of $\calC_\sigma$, $\calB_\sigma$, and $\calA_\sigma$ coming from Claims~\ref{clm:j-k}, \ref{clm:k-j}, \ref{clm:2k-j}, \ref{clm:Type1j}, \ref{clm:TypeIIesscurves},  and \ref{clm:TypeIesscurves} along with the stipulations that if  $\sigma \not\in \cup_{i = \I,\II; j=3,4,5} \, \Sigma_\II(i,j)$ then $\calB_\sigma = \emptyset$ and if $\sigma \not\in \cup_{i = \I,\II; j=5,6} \, \Sigma_\II(i,j)$ then $\calA_\sigma = \emptyset$.  The estimates of Figure~\ref{fig:sumtable} verify inequality $(\ast)$.  These estimates arise from the partition of $\SII$ in Definition~\ref{def:partitionSII} along with the counts listed in the above Claims.  
\begin{figure}
\centering
\rotatebox{90}{
$
\begin{array}{c|ccccccc}
&  && 6|\calC_\sigma|    &+&  |\calB_\sigma| &+&  |\calB_\sigma| + |\calA_\sigma| \\
\hline
\hline
2k \leq j \\
\hline
\sigma \in(\I,1) & 6N_\sigma + 6  & \leq & 6(N_\sigma+1)	&+&   0 	&+& 	0   \\
\sigma \in(\I,2)& 4j_\sigma-2k_\sigma & \leq & 6(j_\sigma-k_\sigma) &+&  0	&+& 	0   \\
\sigma \in(\I,6) & N_\sigma+j_\sigma+3(j_\sigma-2k_\sigma) & =& 6(j_\sigma-k_\sigma) &+&   0	&+& 	N_\sigma-2j_\sigma   \\
\hline
\hline
2k > j \\
\hline
\sigma \in(\I,1)& 6N_\sigma + 6  & \leq & 6(N_\sigma+1)	&+&   0 	&+& 	0   \\
\sigma \in(\I,3)  & 2N_\sigma + 4(j_\sigma-k_\sigma)+2 & = & 6(j_\sigma-k_\sigma) &+&   N_\sigma-(j_\sigma-k_\sigma)+1	&+& 	N_\sigma-(j_\sigma-k_\sigma)+1   \\
\sigma \in(\I,4)  & 4j_\sigma-2k_\sigma & =& 6(j_\sigma-k_\sigma) &+&   2k_\sigma-j_\sigma 	&+& 	2k_\sigma-j_\sigma    \\
\sigma \in(\I,5) & N_\sigma + 2(j_\sigma-k_\sigma) & = & 6(j_\sigma-k_\sigma) &+&   2k_\sigma-j_\sigma 	&+& 	2k_\sigma-j_\sigma + N_\sigma-2j_\sigma   \\
\hline \hline
2j \leq k \\
\hline
\sigma \in(\II,1)&  6 N_\sigma + 6		&  <   & 6  \min\{2N_\sigma+2 ,k-j\}     &+&          0         &+&      0     \\
\sigma \in(\II,2) & 3k_\sigma-j_\sigma	&\leq & 6(k_\sigma-j_\sigma)  &+&      0   &+&  0 \\
\sigma \in(\II,6) & \overset{ N_\sigma + 3(k_\sigma-2j_\sigma)}{\underset{ + (k_\sigma-j_\sigma)+k_\sigma}{}}	&=& 6(k_\sigma-j_\sigma)  &+&    0 &+&    N_\sigma - (k_\sigma + j_\sigma) \\
\hline
\hline
2j > k \\
\hline
\sigma \in(\II,1) &  6 N_\sigma + 6		&  <   &  6 \min\{2N_\sigma+2 ,k-j\}      &+&          0                                       &+&      0     \\
\sigma \in(\II,3) & 2N_\sigma+4(k_\sigma-j_\sigma)+2 &=& 6(k_\sigma-j_\sigma)  &+&      N_\sigma-(k_\sigma-j_\sigma) +1  &+&   N_\sigma-(k_\sigma-j_\sigma) +1 \\
\sigma \in(\II,4)& 3k_\sigma-j_\sigma  &\leq & 6(k_\sigma-j_\sigma)  &+&     j_\sigma 				&+&   j_\sigma  \\
\sigma \in(\II,5) & N_\sigma + 5(k_\sigma-j_\sigma) &=& 6(k_\sigma-j_\sigma)  &+&     j_\sigma 				&+&   j_\sigma + N_\sigma - (k_\sigma + j_\sigma) \\
\hline
\end{array}
$
}
\caption{}
\label{fig:sumtable}
\end{figure}
In Figure~\ref{fig:sumtable}, each quantity in the left hand side of the second column is at least $N_\sigma+1$.  To see this, note that  $j_\sigma - k_\sigma \geq 2$ for $\sigma \in (\I)$ and $k_\sigma -j_\sigma \geq 2$ for $\sigma \in (\II)$, and then further note the upper bounds on $N_\sigma$ for $\sigma \in (\I,2), (\I,4), (\II,2), (\II,4)$ given in Definition~\ref{def:partitionSII}.
\end{proof}
%


\section{Type I\!I\!I trigons}\label{sec:typeIII}
This section is primarily dedicated to elucidating the structure of Type I\!I\!I trigons which are cycles.
We construct a collection $\Sigma_\III$ of such trigons that captures all the labels $x$ of $\Lambda$ that give rise to Type I\!I\!I trigons (whether $x$--cycles or not) that do not belong already to $L(\Sigma)$.

\begin{figure}
\centering
\input{FigureTypeIIIv2.pstex_t}
\caption{}
\label{fig:FigureTypeIII}
\end{figure}

Recall that a Type I\!I\!I trigon on the label $x$ is a trigon face of $\Lambda_x$ which, regarded as a subgraph of $\Lambda$, has a single Scharlemann cycle (of length $2$) and whose central trigon face (of $\Lambda$) has only two corner types. Figure~\ref{fig:FigureTypeIII} shows a generic (as we shall see below) Type I\!I\!I $x$--cycle trigon where the central trigon face is in the two corner types $\arc{j,j+1},\arc{l-1,l}$.

We now define the terms {\em ideal corner}, {\em trunk}, and {\em gnarl} associated to a type I\!I\!I trigon.
\begin{defn}
Let $\sigma$ be a Type I\!I\!I trigon on $x$ which is a cycle (i.e.\ the edges of $\sigma$ form a cycle in $\Lambda_x$ so that, when oriented, the tails have label $x$). Let $p,p+1$ be the labels of its core Scharlemann cycle.
Let $\epsilon$  be the extended Scharlemann cycle containing the core Scharlemann cycle of $\sigma$ which abuts its central trigon face. There is a label $y$ such that $\sigma$ has exactly two edges, $e_1,e_2$ with labels $x,y$. One can see that $x+y \equiv p + (p+1)$  $(\text{mod}\, t)$. That is, around a vertex of $\Lambda$, the labels $x,y$ are equidistant from labels $p,p+1$.  
Define the {\em ideal corner of $\sigma$}, denoted $c(\sigma)$, 
to be the label interval on a vertex of $\Lambda$ between labels $x$ and $y$ that contains labels $p,p+1$. These edges $e_1,e_2$ lie on different arms, $a_1,a_2$ of $\sigma$.  The ideal corner $c(\sigma)$ appears twice as an actual corner of $\sigma$ in $\Lambda$ -- once between $e_1,e_2$. 
For any pair of labels $i,j$ in $c(\sigma)$ such that $i+j \equiv 2p+1$  $(\text{mod}\, t)$, there are exactly two edges of $\Lambda$ in $\sigma$ with endpoints $i,j$. These edges lie in the arms $a_1, a_2$ of $\sigma$.  The subgraph of $G_F$ formed by the vertices $i,j$ and these two edges is a circle in $\widehat F$. If it is an element of $a(\epsilon)$ we call it a {\em trunk} of $\sigma$; otherwise we refer to it as a {\em gnarl} of $\sigma$ (see also Definition~\ref{def:simpleandwrapping}).   To sum up, the ideal corner $c(\sigma)$ has the following properties:
\begin{itemize}
\item Every label in $c(\sigma)$ belongs
to a unique gnarl or trunk of $\sigma$. 
\item $x$ is a label in $c(\sigma)$.
\item $c(\sigma)$ is symmetric about the consecutive labels $p,p+1$ of the core of $\sigma$.
\end{itemize}
\end{defn}

As a subgraph of $\Lambda$, two of the actual corners of $\sigma$ are the same as $c(\sigma)$.  The other either properly contains $c(\sigma)$ or is properly contained in it. After paring down $\sigma$, we may assume this corner is properly contained in $c(\sigma)$. 
Note that the Type I\!I\!I trigon resulting from this paring down has the same ideal corner but is now a cycle trigon in $y$ rather than $x$. 

\begin{defn}\label{def:SigmaIII}
Let $\calL_\III$ be the collection of labels of $G_Q$ that are in the ideal 
corners of cycle
trigons but are not the labels of extended Scharlemann cycles of length $2,3$. 
That is, 
$\calL_\III=\{ x \in c(\sigma) | \sigma \hbox{ is a cycle trigon of Type I\!I\!I} \}
- L(\Sigma)$.
Let $\Sigma_{\III}$ be a collection of smallest cardinality of 
Type I\!I\!I cycle trigons with
the property that $\calL_\III \subset \cup_{\sigma \in \Sigma_{\III}}c(\sigma)$.
We may assume that each element, $\sigma$, of $\Sigma_\III$ has
been pared down so that each actual corner of $\sigma$ is a subset of its ideal corner.
\end{defn}

\begin{remark}\label{sigma2remarks1} \hfill
\begin{enumerate}
\item For every $x \in \calL - L(\Sigma)$ such that $\Lambda_x$ has a trigon which is a Type I\!I\!I non-cycle trigon, there is a $\sigma \in \Sigma_\III$ with $x \in c(\sigma)$, by Lemma~\ref{lem:noncycleIII} below.
\item Item (1) and the definition of $\Sigma_{\III}$ imply: The ideal corners of elements of $\Sigma_\III$ include all labels, $x$, of $\calL - L(\Sigma)$ such that $\Lambda_x$
has a Type I\!I\!I trigon. 
\item The elements of $\Sigma_{\III}$ have different core Scharlemann cycles
(by the minimality of $\Sigma_{\III}$,
because the ideal corners are symmetric about the core labels). As different
core Scharlemann cycles cannot lie on a common annulus in $\hatF$ (else there
is a Klein bottle), we can construct a family of $|\Sigma_\III|$ different
curves on $\hatF$ that satisfy Property $P(2)$. Hence $|\Sigma_\III| \leq F_2(g)$, by Lemma~\ref{lem:Fkg}. 
\item Note that $\calL_{\II}$ and $\calL_{\III}$ are not necessarily 
disjoint. Also note that every element of $\calL_{\III}$ corresponds
to the vertex of a gnarl.

\end{enumerate}
\end{remark}

\begin{lemma}\label{lem:noncycleIII}
Let $x \in \calL - L(\Sigma)$. If there is a Type I\!I\!I $x$--trigon of $\Lambda$ 
that is not an $x$--cycle, then there is a trigon $\sigma \in \Sigma_\III$ such that $x \in c(\sigma)$.
\end{lemma}

\begin{figure}
\centering
\input{Figurenoncyclev2.pstex_t}
\caption{}
\label{Figurenoncyclev2}
\end{figure}

\begin{proof}
The $x$--trigon has one corner, $\alpha$, with both endpoints labelled $x$.
It has a corner, $\beta$, with no label $x$, and a corner, $\gamma$, with
exactly one endpoint labelled $x$. Let $j,j+1$ be the labels of the central
trigon face at $\alpha$. Let $\ell-1,\ell$ be the other labels of this central
trigon, such that $j,\ell$ have opposite parity. Two of the edges of the
$x$--trigon will have labels $x,y$ where $x+y \equiv j+\ell$  $(\text{mod}\, t)$. 
Using the fact that no
other labels $x$ appear on $\alpha, \beta, \gamma$, one sees that, 
up to whether the labels increase or decrease as you move clockwise along
these corners, the $x$--trigon appears as in 
Figure~\ref{Figurenoncyclev2}. 
Inside this $x$--trigon is the pictured $y$--cycle, forming a 
Type I\!I\!I cycle trigon. The ideal corner of this Type I\!I\!I $y$--cycle trigon contains $x$ (it is given by $\gamma$ in the figure). By definition, there is a $\sigma \in \Sigma_\III$ such that 
$c(\sigma)$ includes all elements in this corner.
\end{proof}

\begin{defn} 
Let $\sigma$ be a Type I\!I\!I cycle trigon. If $u,v$ are labels 
in $c(\sigma)$, then $\arc{uv}$ is the interval of labels between
them contained in $c(\sigma)$. 
\end{defn}

Let $\sigma \in \Sigma_\III$ be as above.
Then up to the ordering of the labels around a
corner (i.e.\ whether they increase as you move clockwise or anti-clockwise
around a corner) and rechoosing which label is called $x$, we may assume $\sigma$ appears as in Figure~\ref{fig:FigureTypeIII}.
 The core Scharlemann cycle of $\sigma$ 
  has labels $p,p+1$ where $p=(j+\ell-1)/2$.  
There  
$c(\sigma)$ occurs along the corners labelled $\alpha, \gamma$. The shorter corner is labelled $\beta$.  The lower arms in the figure are the arms $a_1,a_2$.
Pair the vertices in the disjoint intervals $\arc{j+\ell-x,\ell-1}$ and 
$\arc{j+1,x}$. Each such pair of vertices has two edges in $\sigma$ that forms a
gnarl on $\hatF$. See Figure~\ref{fig:Figuregnarl}. 

\begin{figure}
\centering
\input{FiguregnarlV2.pstex_t}
\caption{}
\label{fig:Figuregnarl}
\end{figure}

\begin{property}\label{property:1}
The $\alpha,\beta$ labels at the ends of edges of $\sigma \in \Sigma_\III$ 
(Figure~\ref{fig:FigureTypeIII})
incident to a gnarl of $\sigma$ (but not belonging to a gnarl of $\sigma$) 
lie on the same side of the gnarl in $\hatF$. See Figure~\ref{fig:Figuregnarl}.
Specifically, let $r,s$ be the labels in $\sigma$ of two edges $e_1,e_2$ forming
a gnarl, with $r$ the label at corner $\alpha$. There is another edge, $e_3$,
in $\sigma$ incident to the $\alpha$--corner with label $s$. Assume there is
another edge, $e_4$, in $\sigma$ with label $r$ at the $\beta$--corner. Consider
these edges in $\hatF$. Then $e_3,e_4$ lie on the same side in $\hatF$ of the
gnarl $e_1 \cup e_2$. (In Figure~\ref{fig:Figuregnarl}, $\{r,s\}$ is $\{j+1,
l-1\}$ and $\{j+2,l-2\}$.)
\end{property}

\begin{figure}
\centering
\input{FigureTypeIIIrefinedV3.pstex_t}
\caption{}
\label{FigureTypeIIIrefined}
\end{figure}

\begin{figure}
\centering
\input{FigurecomponentofsigmaV2.pstex_t}
\caption{}
\label{Figurecomponentofsigma}
\end{figure}

\begin{figure}
\centering
\input{FigurebranchorderV2.pstex_t}
\caption{}
\label{Figurebranchorder}
\end{figure}

Set $w=j-\ell+1$. We separate the edges in the upper arm of $\sigma$ 
into mod $w$ classes:
Let $n$ be such that  
\[ (1+j)+nw \leq x \leq j+(n+1)w=(\ell-1)+(n+2)w \]

\begin{defn}\label{def:simpleandwrapping}
Let $r$ be such that $(j+\ell+1)/2=p+1 \leq r \leq j $. Note that $p=(j+\ell-1)/2 \geq j+\ell-r \geq \ell$. We may now label $\sigma$ as in Figure~\ref{FigureTypeIIIrefined}.

For each 
$r$ we see in $\widehat{F}$ the {\em component of $\sigma$} pictured in 
Figure~\ref{Figurecomponentofsigma} . 
Note that a gnarl in this component may or may
not be essential on $\widehat{F}$ -- if essential we call it a {\em wrapping
gnarl}, if inessential we call it a {\em simple gnarl}.  The {\em interior} of a simple gnarl is the interior of the disk it bounds on $\widehat{F}$.  
Recall that $\epsilon$ is
the extended Scharlemann cycle in $\sigma$. The component of
$\sigma$ (a subset of $\widehat{F}$) corresponding to $r$ 
consists of a curve  of $a(\epsilon)$ between the vertices with labels $r,\ell+j-r$, called the {\em trunk} of the component, along with two {\em branches}: 
$r, (\ell+j-r)+w, r-w, (\ell+j-r)+2w, \dots$ and $\ell+j-r, r+w, (\ell+j-r)-w, r+2w, \dots$. We think of a branch of the component of $\sigma$ as a path
of gnarls in $\hatF$.  We orient the branch by labelling its vertices 
{0,1,2, \dots} as pictured on the left of Figure~\ref{Figurebranchorder}
(which shows this labeling for the two branches of a component). 
The right side of that figure pictures the corresponding labelling on the 
trigon $\sigma$.
If $\B$ is a 
branch of $\sigma$ and $x$ is a vertex of this branch, let 
$b_{\B}(x) \in \{0,1,2,\dots\}$ be the label coming from the oriented {\B}
(e.g. $b_{\B_1}(r)=0, b_{\B_1}((l+j-r)+w)=1$ for $\B_1$ the first 
branch described
above.)
\end{defn}

Observe:
\begin{itemize}
\item Each branch terminates in a gnarl ($\sigma$ has been pared down, preserving $c(\sigma)$, so
that its third corner is properly contained in $c(\sigma)$).
\item Every label in $c(\sigma)$ lies in a unique gnarl or unique trunk (curve of $a(\epsilon)$), and in a unique component of $\sigma$. Note that if
a label corresponds to the vertex of a trunk curve then it is in $L(\Sigma)$.
\end{itemize}

By examining the oscillation of the labels \{0,1,\dots\} along
$c(\sigma)$ in Figure~\ref{Figurebranchorder}, we get the following.

\begin{property}\label{property:2}
Let $x,y,z$ be vertices on a branch $\B$ of $\sigma$. 
Then $x,y$ are of the same sign 
(as vertices of $G_F$) if and only if $b_{\B}(x) \equiv b_{\B}(y)$  $(\text{mod}\, 2)$.  
Furthermore, assume $b_{\B}(x) < b_{\B}(y)$ and 
$b_{\B}(x) \not\equiv b_{\B}(y)$  $(\text{mod}\, 2)$. If either
\begin{itemize}
\item[(a)] $b_{\B}(z) < b_{\B}(x)$; or
\item[(b)] $b_\B(z) < b_\B(y)$ and $b_\B(z) \equiv b_B(y)$  $(\text{mod}\, 2)$
\end{itemize}
then $z \in \arc{xy} \subset c(\sigma)$. In particular, if $x,y$ form a 
gnarl of $\B$ and
$z$ immediately precedes this gnarl along $\B$, then $z \in \arc{xy} 
\subset c(\sigma)$.
\end{property}

Every vertex in $G_F$ corresponding to a label in $\sigma$ lies in a 
unique component
of $\sigma$ that includes an essential curve from $a(\epsilon)$, its trunk,
and is connected to that curve by edges in its branch.  Such a 
vertex in the interior of a simple gnarl of $\sigma$ must lie on a
component of $\sigma$ that, by Property~\ref{property:1}, would lie entirely
within the interior of the gnarl. In particular, the essential trunk 
curve of this component lies within the simple gnarl, a contradiction. 
Thus we have the following:

\begin{property}\label{property:3}
Let $\sigma \in \Sigma_\III$.
A simple gnarl of $\sigma$ contains no vertices in its interior which 
are labels of $\sigma$.
\end{property}


\begin{lemma}\label{lem:entersonce}
Thinking of a branch as a directed
path of gnarls in $\hatF$, if a branch of $\Sigma_\III$ enters the 
interior of a simple gnarl of $\Sigma_\III$, it never leaves. That is,
let $\sigma_1,\sigma_2 \in \Sigma_\III$. Let $\g_1$ be a simple gnarl
of $\sigma_1$ and $\B_2$ be a branch of $\sigma_2$. 
If a vertex (of $G_F$), $x$, of $\B_2$ lies in the
interior of $\g_1$, then at least one of the 
vertices in $\g_1$ is in $\B_2$. Furthermore, if $y$ is a vertex  in
both $\g_1$ and $\B_2$ then $b_{\B_2}(y) < b_{\B_2}(x)$.
\end{lemma}
\begin{proof}
The vertex $x$ of $\B_2$ in the interior of $\g_1$ must belong to a
simple gnarl $\g_2$ of $\B_2$ (otherwise the vertex belongs to an essential 
simple closed curve intersecting $\g_1$ at most once). 
Let $y_1,y_2$ be the vertices of
$\g_1$ and $x_1,x_2$ the vertices of $\g_2$. See Figure~\ref{Figureinandout}.
As $\B_2$ (as a directed path in $\widehat F$) must pass 
through $\g_1$ to reach $x$ from its trunk, we may take
$y_1$ so that it is in $\B_2$ such that 
$b_{\B_2}(y_1) \leq b_{\B_2}(x_i), i=1,2$.

We take $x_i$ to be the same sign as $y_i$. Possibly $x_1=y_1$ and $x_2=x$ or
$x_2=y_2$ and $x_1=x$.

\begin{figure}
\centering
\input{FigureinandoutV2.pstex_t}
\caption{}
\label{Figureinandout}
\end{figure}

We assume for contradiction that $y_2$ is also in $\B_2$ and that
$b_{\B_2}(y_2) \geq b_{\B_2}(x_i), \, 
i=1,2$ ($x_1,x_2$ are consecutive along $\B_2$). Note that together 
$c(\sigma_1) \cup c(\sigma_2)$
does not account for all labels. (Otherwise Lemma~\ref{lem:nonemptygnarl} and 
Property~\ref{property:3}
applied to $\g_2$ implies that $\g_2$ must be the only gnarl of $\B_2$.
Corollary~\ref{3ingnarl} says that
there must be a vertex in the interior of $\g_1$ that does not belong
to $\g_2$. By Property~\ref{property:3}, this vertex must belong then to
a trunk of $\B_2$. But the trunk gives
rise to an essential curve on $\widehat F$ intersecting $\g_1$ at most once.) 
Thus $c(\sigma_1), c(\sigma_2)$ overlap on a label interval. Let 
$\arc{y_1,y_2}$ be the interval between $y_1,y_2$ included in both $c(\sigma_1)$
and $c(\sigma_2)$.

\smallskip

\noindent {\bf Case (I):} $x_2 \neq y_2$

\smallskip

Applying Property~\ref{property:2}(b)
to the sequence $y_1, x_2, y_2$ on the branch $\B_2$ (with $z=x_2$), 
$x_2 \in \arc{y_1,y_2}_{\sigma_2}=\arc{y_1,y_2}$. 
But then 
$x_2 \in c(\sigma_1)$ and in the interior of $\g_1$, 
contradicting Property~\ref{property:3}.

\smallskip

\noindent {\bf Case (II):} $x_2 = y_2$ and $x_1=x$.

\smallskip
Then $x_1$ is in the interior of $\g_1$. As $b_{\B_2}(x_1) \equiv b_{\B_2}(y_1)$ 
  $(\text{mod}\, 2)$ and $b_{\B_2}(y_1) < b_{\B_2}(x_1)$ there must be another
vertex $x'$ in $\B_2$ in the interior of $\g_1$ such that
$b_{\B_2}(y_1) < b_{\B_2}(x') < b_{\B_2}(x_1) < b_{\B_2}(y_2)$ and such that
$b_{\B_2}(x') \equiv b_{\B_2}(y_2)$ $(\text{mod}\, 2)$.
Applying Property~\ref{property:2}(b) to the sequence $y_1,x',y_2$ we conclude 
that $x'
\in \arc{y_1,y_2}$. That is, $x'$ is a label in $\sigma_1$, contradicting
Property~\ref{property:3}.
\end{proof}

By taking a union of faces in the lower arms of 
Figure~\ref{FigureTypeIIIrefined}, we get:

\begin{property}\label{property:4}
Let $\sigma \in \Sigma_\III$. 
Two gnarls in $\sigma$ cobound an embedded annulus constructed from the
faces of $\sigma$.
\end{property}

\begin{lemma}\label{lem:nonemptygnarl}
Let $\sigma \in \Sigma_{\III}$. If any simple gnarl of $\sigma$ contains
no vertices of $G_F$ in its interior, then this gnarl is 
the only gnarl of $\sigma$.
(i.e.\ all other labels in $\sigma$ are parts of trunks of components 
of $\sigma$). There are at most $F_2(g)$ such gnarls.
\end{lemma}

\begin{proof}  
The proof of this lemma divides into two cases according to whether or not the extended Scharlemann cycle within $\sigma$ is actually just a Scharlemann cycle.

\noindent {\bf Case (I):} The extended Scharlemann cycle within 
$\sigma$ is a Scharlemann cycle.  Let $e$ be the bigon bounded by this Scharlemann cycle.  See Figure~\ref{Figureemptygnarl1} for the labeling we will use.  The {\em initial gnarl} is the gnarl on vertices $j-2, j+1$. We assume that $\sigma$
has more than the initial gnarl, in particular, that there is a face $f_3$
as pictured.

\begin{figure}
\centering
\input{Figureemptygnarl1.pstex_t}
\caption{}
\label{Figureemptygnarl1}
\end{figure}

\begin{figure}
\centering
\input{Figurerectangles-altV2.pstex_t}
\caption{}
\label{Figurerectangles-alt}
\end{figure}

\begin{figure}
\centering
\input{Figurethindisk1-alt.pstex_t}
\caption{}
\label{Figurethindisk1-alt}
\end{figure}

\begin{figure}
\centering
\input{Figuresurface1.pstex_t}
\caption{}
\label{Figuresurface1}
\end{figure}

Assume some simple gnarl of $\sigma$ bounds a disk $D'$ in $\widehat{F}$ whose interior is disjoint from $K$.  Let $B$ be the annulus between this gnarl and the initial gnarl given by Property~\ref{property:4}.  (If the initial gnarl bounds $D'$ then $B = \emptyset$.)  Set $D=B \cup D'$.  Then $D$ is a disk bounded by the initial gnarl and can be taken to have interior disjoint from both $K$ and the faces of $\sigma$.

Let $R_{\alpha \gamma}$ be the rectangle on $\bdry X$ bounded by the curves $j-2$, $j$, $\alpha \cap (e \cup f_1)$ and $\gamma \cap (f_2 \cup e)$ that is disjoint from $\beta$.
Let $R_{\beta \gamma}$ be the rectangle on $\bdry X$ bounded by the curves $j-1$, $j+1$, $\beta \cap (f_2 \cup f_3)$ and $\gamma \cap (e \cup f_1)$ that is disjoint from $\alpha$.  (By the curves $j-2, j-1, j, j+1$ on $\bdry X$, we mean the curves that are the boundaries of those vertices.)
In fact, let $e,e'$ be two push-offs of the face $e$
so that $e$ and $f_2$ agree along their $\edge{j-1,j}$--edge and so that $e'$ agrees 
with $f_1$ along its $\edge{j-1,j}$--edge. Modify $R_{\alpha \gamma}, R_{\beta \gamma}$
along $\partial X$ between $j-1,j$ by deforming in slightly. See  
Figure~\ref{Figurerectangles-alt}.

Form the disk $e \cup f_2 \cup f_3 \cup R_{\alpha \gamma} \cup R_{\beta \gamma} \cup D \cup f_1 \cup e'$ of Figure~\ref{Figurethindisk1-alt}.  This gives a thinning long disk $D_{\arc{j-1,j+2}}$ taking the subarc $\arc{j-1,j+2}$ of $K$ (view $\arc{j-1,j+2}$ as part of $\alpha \cap \bdry X$) to the arc on $\widehat{F}$ pictured in Figure~\ref{Figuresurface1}.  This contradicts the thinness of $K$ and finishes the proof of Case (I).

\noindent {\bf Case (II):} The extended Scharlemann cycle in $\sigma$ is not
a Scharlemann cycle (has length at least $2$). 

\begin{figure}
\centering
\input{Figureemptydisk2.pstex_t}
\caption{}
\label{Figureemptydisk2}
\end{figure}

\begin{figure}
\centering
\input{Figurerectangles2.pstex_t}
\caption{}
\label{Figurerectangles2}
\end{figure}

\begin{figure}
\centering
\input{Figurethindisk2-test-alt.pstex_t}
\caption{}
\label{Figurethindisk2}
\end{figure}

See Figure~\ref{Figureemptydisk2}. We assume that $\sigma$ has more 
than one gnarl, giving rise to face $f_4$ in the figure.
Let $R_{\beta \gamma}$ and $R_{\alpha \gamma}$ be the rectangles on $\bdry X$ bounded by $j-1 \cup j+1 \cup ((e \cup f_4)
\cap \beta) \cup ((f_1 \cup f_2) \cap \gamma)$ and $\ell-1 \cup \ell+1 \cup
((e \cup f_3) \cap \gamma) \cup ((f_1 \cup f_2) \cap \alpha)$, respectively. See 
Figure~\ref{Figurerectangles2}.  As in Case (I), the assumption of the empty simple gnarl allows us to construct a disk $D$ in $X$ which is disjoint from the faces of $\sigma$
and from $K$.  Construct the thinning disk $D_{\arc{j-1,j+2}} = e \cup f_1
\cup f_3 \cup f_4 \cup R_{\alpha \gamma} \cup R_{\beta \gamma}$ depicted in Figure~\ref{Figurethindisk2}. $D_{\arc{j-1,j+2}}$ gives
an isotopy of $\arc{j-1,j+2}$ onto the arc in $\widehat{F}$ indicated in Figure~\ref{Figurethindisk2} thereby thinning $K$.  This contradicts the thinness of $K$ and finishes the proof of Case (II).

For the last statement in the lemma, recall from Remark~\ref{sigma2remarks1}(3)
that $|\Sigma_\III| \le F_2(g)$. By the above, the gnarls containing no interior
vertices belong to different elements of $\Sigma_\III$.
\end{proof}

\begin{cor}\label{3ingnarl}
Let $\sigma \in \Sigma_{\III}$. If a gnarl in $\sigma$ has vertices of $G_F$ 
in its
interior, it must have at least three. In particular, if $\sigma$ has more 
than one gnarl, then
any simple gnarl of $\sigma$ must contain
at least three vertices in the interior of the disk it bounds on $\widehat{F}$.
\end{cor}

\begin{proof}  
If a simple gnarl only has one vertex in its interior then the 
graph within would have a $1$--sided face (a contradiction), 
because the number of labels on the interior of the gnarl is less than $q t$
(see Figure~\ref{Figurecomponentofsigma} for the labeling around a gnarl), 
the total number of labels (recall that $q$ is the denominator of the 
Dehn surgery coefficient). 

Now assume the simple gnarl on vertices $x,y$ has exactly two vertices $v_1$ and $v_2$ in its interior. There are fewer than $t$ edges connecting $v_1,v_2$.  Otherwise they are all parallel and describe a cable space in the exterior
of $K$ (section 5 of \cite{GLi}) contradicting that $K$ is hyperbolic.  Thus there are more than $2 q t - 2t$ labels on $x,y$ that are endpoints of edges connecting $v_1,v_2$ to $x,y$ (no 1-sided faces). But there are fewer than $q t$ such labels on $x,y$ in the interior of the gnarl.  This is a contradiction since $q$ is at least $2$.

By Lemma~\ref{lem:nonemptygnarl}, if $\sigma$ has more than one 
gnarl, 
any of its simple gnarls must contain vertices in its interior.
\end{proof}

\begin{lemma}\label{lem:parallelwrapping}
Let $\sigma \in \Sigma_{\III}$. Let $\B$ be a branch of (a component of) 
$\sigma$. Let $\trunk$ be the trunk of $\B$ and $\g$ be a gnarl of $\B$. If
$\trunk$ and $\g$ cobound an annulus $A$ on $\widehat F$, then either
$\B$ is disjoint from $\Int A$ or $M$ is a Seifert fibered space over the
2-sphere with three exceptional fibers, one having order $2$ and another
order $3$. In this latter case, if $\hatF$ is from a genus $2$
splitting of $M$ with respect to which $K$ has smallest bridge number,
then $|K \cap \Int A| \ge t/2 - 2$.
\end{lemma}

\begin{figure}
\centering
\input{FigureparallelwrappingV2.pstex_t}
\caption{}
\label{fig:Figureparallelwrapping}
\end{figure}

\begin{proof}
Let $A$ be the annulus in $\widehat{F}$ cobounded 
by $\g$ and $\trunk$.

If $\sigma$ contains a simple gnarl $\g_s$, then let $D_s$ be the disk in $\widehat{F}$ that $\g_s$ bounds.  Property~\ref{property:4} gives an annulus $B$ from $\g_s$ to $\g$ (and transverse to $\hatF$ there) arising from $\sigma$.  
Therefore, $\g$ bounds the embedded disk 
$D = B \cup D_s \subset M$ transverse to $\widehat F$ along $\g$ (by Property
3, $B$ is disjoint from the interior of $D_s$).  
On the other hand, $\trunk$ bounds a long \mobius band, $E$, coming 
from the extended
Scharlemann cycle of $\sigma$ which is transverse to $\widehat F$ at $\trunk$.
Then $A \cup D \cup E$ is an immersed projective plane that can be surgered
to produce an embedded projective plane in $M$ -- a contradiction.
Therefore $\sigma$ contains no simple gnarls.

 We assume $\B$ intersects $\Int A$. Then Property~\ref{property:1} along with 
the fact that $\B$ contains no simple gnarls guarantees that $\g$ is the only gnarl in $\B$. That is, $\B$ consists of $\trunk$, $\g$, and a single edge $e$ connecting them that lies in $A$.
Without loss of generality, we may assume we are as in Figure~\ref{fig:Figureparallelwrapping}.  Note that $\Int A$ contains no vertices that appear as labels in $\sigma$ as any such would be part of an essential curve in $\Int A$ disjoint from $e$. Similarly, $e$ is the only edge of $\sigma$ contained in $\Int A$. 

Let  $h_{wr}$, $h_{rs}$, and $h_{sz}$ be the consecutive handles in $\nbhd(K)$ running between vertices $w,r,s,z$.  Let $N=\nbhd(A \cup h_{sz} \cup h_{wr} \cup 
F_1 \cup F_2)$ where the $F_i$ are the disks in $\sigma$ from 
Figure~\ref{fig:Figureparallelwrapping}.  
Then $\pi_1(N)$ has presentation $<x,y,c\,|\,x^2y,xcyc>$  where $x,y$ correspond to $h_{sz},h_{wr}$, respectively, $e,e_1,e_2$ are
retracted to a base point, and $c$ represents the core of $A$. 
This can be rewritten $<x,m\,|\,m^2x^{-3}>$ (using
$m=xc$). Thus $N$ is a trefoil knot exterior and furthermore the core of $A$
is its meridian (attaching a disk along $c$ kills the group). Let $\hat{E}
=\nbhd(E \cup h_{rs})$ be a thickening of the long \mobius band coming from
the extended Scharlemann cycle, and consider the submanifold $N'=N \cup \hat{E}$. As $M$ is atoroidal, $\partial N'$
must compress in $M-N'$. As $M$ is irreducible, compressing $\partial N'$
gives a 2-sphere which bounds a 3-ball in $M$. If this 3-ball contained $N'$,
then $M$ would contain a projective plane constructed from a disk in the
$3$--ball bounded
by the meridian of $N$ and the long \mobius band. 
Thus the 3-ball is disjoint from $N'$, and $M-N'$ is a
solid torus. This implies that $M$ is a Seifert fibered space over the
2-sphere with three exceptional fibers, one of which has order $2$ and another
of order $3$. 

So we now assume that $\hatF$ is a genus $2$ splitting with respect to
which $K$ has smallest bridge number. Let $n=|\Int A \cap K|$. Let $D$
be the disk in $A$ gotten by removing from $A$ an open neighborhood of
$\partial A \cup e$. Then $K$
can be perturbed so that it intersects $N'$ in $n+2$ arcs, each of the
form $p \times I$ in $D \times I$ (e.g. perturb $K$ from label $w$ to
label $z$ between $\alpha$ and $\gamma$ in $\nbhd(K)$). 

Now $H'=\nbhd( h_{sz} \cup h_{wr} \cup 
F_1 \cup F_2)$ is a genus $2$ handlebody as neighborhoods of an arc in 
$F_1$ and an arc in $F_2$ give a pair of meridians. Furthermore, the
first of these meridians shows that the annulus along which $H'$ and
$\hat{E}$ meet is primitive in $H'$. Thus $H=N' - D \times I$ is a 
genus $2$ handlebody.

Note that the surgery curve $p/q$ is not a boundary slope for $X$,
the exterior of $K$. For as $M$ contains no incompressible surface, this
would contradict Theorem 2.0.3 of \cite{cgls:dsok}.
Consider the punctured torus $T=\partial N'- \nbhd{K}$. Maximally compress
$T$ in $X$. Then $T$ compresses to $\partial$--parallel annuli.
Each $\partial$--parallel annulus defines an isotopy in $M$ of 
a subarc of $K$, with endpoints in $K \cap \partial N'$, onto $\partial N'$ 
and keeping
the endpoints of the arc fixed. Furthermore, every point of 
$K \cap \partial N'$ 
belongs to such an arc of $K$. Let $k_1', \dots , k_r'$ be the subarcs of $K$
lying within the outermost of these $\partial$--parallel annuli and let $k_1, \dots,
k_r$ be the complementary subarcs of $K$. Note that the $k_i$ lie on the same
side of $\partial N'$. Then the isotopies described above 
corresponding to the outermost annuli deform $K$ to the union of the arcs $k_1, \dots,
k_r$ along with arcs that lie in $\partial N$.  As the complement of $K$ in $M$ 
is hyperbolic,
it must be that the arcs $k_1, \dots, k_r$ lie in $N'$. That is,  
$k_1, \dots, k_r$ are of the form $p \times I$ for $p \in D$. 
Let $\tau$ be an arc $q \times I$ for $q \in D$ that is disjoint from $K$. 
Then 
$[N' - \nbhd(\tau)] \cup [(M-N') \cup \nbhd(\tau)]$ is a genus $2$ 
Heegaard splitting of $M$. 
Furthermore, each of $k_1, \dots, k_r$  is $\partial$--parallel 
in the handlebody $N' -\nbhd(\tau)$. 
Perturbing each arc of $K - (\cup k_i)$ into $M-N'$ puts $K$ in bridge position 
with respect to this genus $2$ splitting. Noting that $r \le n+2$ and that
$t$ is at most twice the minimal bridge number that $K$ has with respect
to a genus $2$ splitting of $M$, we get that $t/2 \leq n+2$ as desired.
\end{proof}

\begin{lemma}\label{lem:samelabels} 
A gnarl or trunk of $\Sigma_{\III}$ shares
at most a single vertex (of $G_F$) with another gnarl or trunk of 
$\Sigma_{\III}$. 
Furthermore, a vertex of $G_F$ belongs to at most two elements among the
collection of gnarls and trunks of $\Sigma_{\III}$.
\end{lemma}

\begin{figure}
\centering
\input{Figurexyedge.pstex_t}
\caption{}
\label{Figurexyedge}
\end{figure}

\begin{proof}
Let $x,y$ be shared vertices of one gnarl or trunk with another gnarl or trunk.
The two must belong to different trigons $\sigma_1,\sigma_2$ in $\Sigma_{\III}$. 
Let $p,p+1$ be the labels of the core Scharlemann cycle of one of these trigons.
See Figure~\ref{Figurexyedge}. Then $x+y \equiv 2p+1$ $(\text{mod}\, t)$ implying that 
$p=(x+y-1)/2$ or $p=(x+y-1+t)/2$. If the cores of $\sigma_1$ and 
$\sigma_2$ are different, then all labels of $G_Q$ are contained in 
$c(\sigma_1) \cup c(\sigma_2)$ and the core Scharlemann cycles are `antipodal'.
By the minimality of $\Sigma_{\III}$, $\sigma_1,\sigma_2$ are the only elements
of $\Sigma_\III$. Because their cores are antipodal, we may pare one of them
down so that their corners are disjoint but their union contains all labels.
We may then have assumed that this was $\Sigma_\III$ and the resulting gnarls
and trunks share no labels at all.

If a vertex of $G_F$ belongs to three different gnarls or trunks of 
$\Sigma_{\III}$,
then it must belong to $c(\sigma_1) \cap c(\sigma_2) \cap c(\sigma_3)$ for
three different trigons in $\Sigma_{\III}$. But this contradicts the minimality
of $\Sigma_{\III}$.
\end{proof}

\begin{defn} Lemma~\ref{lem:samelabels} says that two simple gnarls of 
$\Sigma_{\III}$ whose interiors overlap must be nested. Thus we define 
a simple gnarl of $\Sigma_{\III}$ 
to have depth $0$ if it is 
innermost among simple gnarls of $\Sigma_{\III}$. 
A simple gnarl of $\Sigma_{\III}$ has depth $n+1$ if it contains 
a gnarl of depth
$n$ in its interior (and no higher). 
\end{defn}

\begin{lemma}\label{lem:depth<2}
Any two simple gnarls 
of $\Sigma_{\III}$ whose interiors overlap (hence are nested) must share a 
vertex. A simple gnarl of $\Sigma_{\III}$ has at most two gnarls of
$\Sigma_{\III}$ nested within it. Every simple gnarl of $\Sigma_{\III}$ has 
depth at most $1$. 

\end{lemma}
\begin{proof}
Let $\g_1,\g_0$ be simple gnarls 
on branches $\B_1,\B_0$ of $\sigma_1,\sigma_0$
in $\Sigma_\III$. Assume $\g_0$ is nested within $\g_1$. By Property $3$,
$\sigma_1$ and $\sigma_0$ must be different elements of $\Sigma_\III$. Let
$x_i,y_i$ be the vertices of the gnarl $\g_i$ for $i=0,1$. 
Assume that $\{x_0,y_0\} \cap \{x_1,y_1\} = \emptyset$. 
As the trunk curve of $\g_0$ is essential in $\hatF$, at least one of 
$\{x_1,y_1\}$ must precede $\g_0$ on $\B_0$, say
$x_1$. By Property 2(b), $x_1 \in \overline{x_0,y_0}_{\sigma_0}$.
As Property 3 implies that neither $x_0$ nor $y_0$ can be in 
$\overline{x_1,y_1}_{\sigma_1}$, it must be that 
$y_1 \in \overline{x_0,y_0}_{\sigma_0}$.
By assumption, and the minimality of $\Sigma_{\III}$, 
$\overline{x_1,y_1}_{\sigma_1}
\subset \overline{x_0,y_0}_{\sigma_2}$. But this contradicts the minimality of 
$\Sigma_{\III}$. This proves that two nested simple 
gnarls of $\Sigma_{\III}$ must
share a vertex. As a vertex can belong to at most two gnarls in $\Sigma_{\III}$,
this implies that a simple gnarl has at most two nested within it.

Let $\g_2,\g_1,\g_0$ be simple gnarls of $\Sigma_{\III}$ with $g_0$ 
nested within $\g_1$
which is nested with $\g_2$. Let $\{x_i,y_i\}$ be the vertices of $\g_i$,
$i=0,1,2$. By the paragraph above we may assume $x_1=x_0$. By 
Lemma~\ref{lem:samelabels}, $y_0$ must lie strictly in the interior of 
$\g_1$. Again by the paragraph above $\g_0$ and $\g_2$ must share a vertex.
It cannot be $y_0$ as it lies in the interior of $\g_0$. Thus $x_0$ lies
in all three gnarls, contradicting Lemma~\ref{lem:samelabels}. Thus
a simple gnarl of $\Sigma_{\III}$ has depth at most $1$.
\end{proof}

\begin{lemma}\label{lem:boundwrapping}
There are at most $6F_2(g)$ wrapping gnarls in $\Sigma_\III$ unless
$M$ is a Seifert fibered space over the 2-sphere with three exceptional
fibers, one having order $2$ and another order $3$. In the latter case, if $\hat F$
is a genus $2$ Heegaard splitting of $M$ with respect
to which $K$ has smallest bridge number (among genus $2$ splittings), then
there are at most $7F_2(g)$ wrapping gnarls in $\Sigma_\III$.
\end{lemma}

\begin{proof}
First assume that $M$ is not a Seifert fibered space as described.
Let $\calC$ be the collection of wrapping gnarls coming from the 
trigons in $\Sigma_\III$,
and assume for contradiction that $|\calC|>6F_2(g)$.
Lemma~\ref{lem:samelabels} shows that $\calC$ satisfies properties $(1),(3),(4)$
of Definition~\ref{defn3.1}. If two elements of $\calC$ intersect non-transversely, perturb them so they
are disjoint so that property $(2)$ of Definition~\ref{defn3.1} is also satisfied.  Then Lemma~\ref{lem:Fkg} guarantees that there are
at least seven elements of $\calC$ that are isotopic on $\hatF$. Any two
such elements much be disjoint. Let $\g_1,\g_2$ be the outermost gnarls of 
these seven cobounding an annulus $A$. Let $\g$ be a gnarl in $A$ between
$\g_1,\g_2$. Then $\g$ belongs to a branch of $\Sigma_\III$. Call the part
of the branch between $\g$ and its trunk, the {\em short branch} of $\g$.
If the short branch of $\g$ is disjoint from $\g_1,\g_2$, then the
trunk of $\g$ is parallel to $\g$ in $A$; furthermore, as a branch 
always locally lies on the same side of a constituent gnarl (Property 1),
the entire branch will lie in $A$ between $\g$ and its trunk, contradicting
Lemma~\ref{lem:parallelwrapping}. Thus the short branch of $\g$ 
must intersect either $\g_1$ or $\g_2$. As
each vertex of a branch either belongs to a gnarl or trunk curve,
we may associate to $\g$, $b(\g)$, the first gnarl or trunk curve on the
short branch that shares a vertex with either $\g_1$ or $\g_2$ (orienting
the short branch from $\g$ to its trunk). If $\g,\g'$ are different
gnarls between $\g_1,\g_2$, then $b(\g)$ and $b(\g')$ must be different.
Otherwise, $\g,\g'$ belong to the same branch ($b(\g)=b(\g')$ cannot be
the trunk curve of two different branches by Lemma~\ref{lem:samelabels}) 
and are therefore disjoint before
perturbation. Then $\g'$, say, must occur on the short branch of $\g$.
As above, Property 1 shows that the entire branch of $\g$ lies between
$\g,\g'$, contradicting Lemma~\ref{lem:parallelwrapping}. Similarly,
one sees that $b(\g)$ is neither $\g_1$ nor $\g_2$. Lemma~\ref{lem:samelabels}
says that there are at most four trunks or gnarls of $\Sigma_\III$ that
share a vertex with $\g_1$ or $\g_2$, other than $\g_1$ or $\g_2$ themselves.
This contradicts that there are five gnarls between $\g_1,\g_2$.

Now assume $M$ is a Seifert fibered space over the 2-sphere with three
exceptional fibers including orders $2,3$ and $\hatF$ comes from a
genus $2$ splitting for which $K$ has smallest bridge number.
Let $\calC$ be the collection of wrapping gnarls coming from the 
trigons in $\Sigma_\III$, and assume for contradiction that $|\calC|>7F_2(g)$.
Now at least eight elements of $\calC$ are isotopic on $\hatF$.
The above shows that there must be two gnarls $\g,\g'$ of $\calC$ 
whose branches 
lie entirely in an annulus $A$ on $\hatF$. In particular the trunk curves 
of these gnarls are parallel (or equal) in $A$. Then by 
Corollary~\ref{cor:noparallelesc2} and Remark~\ref{sigma2remarks1}(3), 
there is a
$\sigma \in \Sigma_\III$ containing both of these branches. These branches
then are either disjoint or share the same trunk curve. Let $A_1,A_2$ be the
sub-annuli of $A$ that lie between these gnarls and their trunk curves.
Then the interiors of these annuli are disjoint and each intersects the
branch of the corresponding gnarl. 
Lemma~\ref{lem:parallelwrapping} shows that $t-4$ vertices of $G_F$ lie in the 
union of the interiors of these annuli. But there are at least six more vertices
coming from the gnarls and their trunk curve(s), contradicting that $G_F$
has $t$ vertices.
\end{proof}

\section{ Proof of Theorem~\ref{thm:main}.}\label{sec:proofofmain}

In this section we prove our main theorem, which is the following.

\begin{thm}\label{thm:main}
There is a linear function $w : \mathbb N \to \mathbb N$ with the following 
property. 
Let $K'$ be a hyperbolic knot in $S^3$, $M = K'(p/q)$ where $q \ge 2$, 
and $K$ the core of the attached solid torus in $M$. 
Suppose $K$ is in thin position with respect to a genus $g$ Heegaard 
splitting of $M$ and let $S$ be a corresponding thick level surface. 
If $S$ is a strongly irreducible Heegaard surface for
$M$ 
then either
\begin{itemize}
\item[(1)] $|K\cap S| \le 2w(g)$; or

\item[(2)] $M$ is toroidal; or

\item[(3)] $M$ is a Seifert fibered space over the 2-sphere with exactly three exceptional fibers, at least one of which has order 2 or 3. 
\end{itemize}
Furthermore, in cases (2) and (3) $M$ has a genus~2 
Heegaard splitting with respect to 
which $K$ has bridge number~$0$ in case (2) and at most $w(2)$ in case (3). 
\end{thm}

Theorem~\ref{thm:main} is proved by establishing the existence of a collection
of annuli in $S=\hatF$ whose boundary components are, roughly speaking, elements
of $\ao(\Sigma)$, and which capture, in their interiors, vertices of $G_F$
that do not correspond to labels in $\calL$. This is stated precisely in 
Lemma~\ref{lem:buildA}, which is proved in subsection 8.1 below. Before 
giving the
proof of the lemma, we show how it leads to the proof of Theorem~\ref{thm:main}.

\begin{proof}[Proof of Theorem~\ref{thm:main}.] 

Let $K'$ be a knot in $S^3$ and $M=K'(p/q)$ where $q \geq 2$.
Let $K$ be the core of
the attached solid torus in $K'(p/q)$. We are given a genus $g$, strongly
irreducible Heegaard
surface $S$ for $M$, and we may assume
$K$ cannot be isotoped onto this surface, else we may take $(1)$ to hold. 
In the special case that $M$ is a Seifert fibered space as in $(3)$, we 
further assume that $S$ comes from a genus $2$ 
Heegaard splitting of $M$ with respect to which $K$ has the smallest
bridge number among all genus $2$ splittings of $M$. 
Note that as $M$ is neither a lens space, $S^3$, $S^1 \times S^2$, 
nor a connected sum, 
any genus $2$ splitting
of $M$ is strongly irreducible. 
Then there are surfaces $F,Q$ 
in the exterior of $K'$ as described by Lemma~\ref{lem:FandQ}, 
where $\hatF=S$ is a thick level surface in a thin
presentation of $K$ with respect to the given splitting. Let $G_F,G_Q$
be the corresponding graphs of intersection. Let $t=|\partial F|=
|K \cap \hatF|$. 
We assume that $M$ is atoroidal so that $(2)$ of Theorem~\ref{thm:main} does
not hold. We then show that 
either conclusion $(1)$ (for $w(g)=10,581(g-1)+394$) or $(3)$ 
(with the bridge number bound on minimal bridge number,
genus $2$ splitting) of Theorem~\ref{thm:main} holds.

We may assume $t \geq 2g-2$, so Lemma~\ref{lem:2} guarantees that $G_Q$ has a
great web $\Lambda$. Let $\calL$ be the collection of labels of
$\Lambda$ given in Definition~\ref{def:calL}, and $-\calL$ the complement
of $\calL$ among all labels of $G_Q$. Let $\Sigma$ be the minimal collection
of extended Scharlemann cycles of $\Lambda$ given by 
Definition~\ref{def:Sigma}. Let $\Sigma_\II$ and $\Sigma_\III$ be as in
Definitions~\ref{def:cornersI&II} and \ref{def:SigmaIII}. 
Then for every $x \in \calL$, either $x \in L(\Sigma)$ or there is 
a $\sigma \in \Sigma_\II \cup \Sigma_\III$ such that $x \in c(\sigma)$.
Let $\G$ be the set of labels of $G_Q$ that correspond to vertices of 
simple gnarls of $\Sigma_\III$ 
that are not in $L(\Sigma)$ (section~\ref{sec:typeIII}). 
A gnarl hereafter will always be a gnarl of $\Sigma_\III$.

Every element of $\calL$ belongs to either $L(\Sigma)$, $\calL_\II$ 
(Definition~\ref{def:L2}), or corresponds to a vertex of
a gnarl (Definition~\ref{def:simpleandwrapping}). 
A gnarl may be either wrapping or simple.
By Theorem~\ref{thm:Type12boundN} and Lemma~\ref{lem:boundwrapping}, 
\begin{equation} 
|\calL| \leq |L(\Sigma)|+|\G|+960(g-1)+14F_2(g) \tag{$*$} 
\end{equation}
Let $\calA$ be the collection of annuli given by Lemma~\ref{lem:buildA}.
We can write $\G=\G' \coprod \G''$, where $\G'$ are the elements
of $\G$ corresponding to vertices that lie outside $\cup_{B \in \calA} B$, and
$\G''$ are those corresponding to vertices in the interiors of the annuli
in $\calA$. 

Using Lemma~\ref{lem:nonemptygnarl} to remove from consideration those 
gnarls with no vertices in their interior, Corollary~\ref{3ingnarl} and
Lemma~\ref{lem:depth<2} imply that
there are at least $|\G'|-2F_2(g)$ vertices of $G_F$ outside 
$\cup_{B \in \calA}B$ that lie in the interior of simple gnarls but are
not vertices of gnarls. Such vertices cannot belong to $L(\Sigma)$, so
they are either in $\calL_\II$ or in $-\calL$. 
\begin{claim}\label{clm:G'bound}
There are at least $|\G'|-2F_2(g)$ vertices of $G_F$ that lie outside of
$\cup_{B \in \calA} B$ and that correspond to elements of $-\calL \cup \calL_\II$.
\end{claim}

\begin{proof}
Let $\overline{\G'}$ be the set of elements of $\G'$ that do not belong to 
simple gnarls that contain no vertices in their interior. By 
Lemma~\ref{lem:nonemptygnarl}, $|\overline{\G'}| \geq |\G'|-2F_2(g)$.
To the elements of $\overline{\G'}$ we associate distinct elements of 
$-\calL \cup \calL_\II$ that lie outside $\cup_{B \in \calA}B$. 

Consider a depth $1$ gnarl $\g$ that contains an element of $\overline{\G'}$
(as one of its vertices). By Corollary~\ref{3ingnarl} 
and Lemma~\ref{lem:depth<2},
there are at least as many vertices in the interior of $\g$ that do not belong
to gnarls as there are elements of $\overline{\G'}$ that lie in $\g$ or its
interior ($\g$ contains at most two depth $0$ gnarls, each of which shares
a vertex with $\g$, and two such gnarls have disjoint interiors). These elements
in the interior of $\g$ that do not belong to gnarls cannot be in $L(\Sigma)$,
hence must correspond to elements of $-\calL \cup \calL_\II$. Furthermore,
as they are interior to $\g$, they must lie outside of $\cup_{B \in \calA} B$. 
To the elements of $\overline{\G'}$ in $\g$ or its interior we associate
distinct elements of $-\calL \cup \calL_\II$ among these. 

In sequence, consider all elements of $\overline{\G'}$ that lie in depth $1$
gnarls that have not previously been assigned an element of $-\calL \cup \calL_\II$.
Apply the above procedure to assign elements of $-\calL \cup \calL_\II$ to that
element of $\overline{\G'}$ along with any others that lie in that gnarl
or its interior. 

Each of the elements of $\overline{\G'}$ not assigned elements of 
$-\calL \cup \calL_\II$ by the above process must lie in a depth $0$ 
gnarl (with interior vertices) that does not itself lie in a depth $1$ gnarl
containing an element of $\overline{\G'}$. Consider such a depth $0$ gnarl.
By Corollary~\ref{3ingnarl}, we may associate distinct vertices in the interior
of $\g$ to the elements of $\overline{\G'}$ that belong to $\g$. These
interior vertices will again belong to $-\calL \cup \calL_\II$ and lie outside
of $\cup_{B \in \calA} B$. By looking at such depth $0$ gnarls we may thus sequentially
assign elements of $-\calL \cup \calL_\II$ to the remaining elements of 
$\overline{\G'}$. 

In the above procedure, no two elements of $\overline{\G'}$ can be assigned the 
same element of $-\calL \cup \calL_\II$. For each such element of 
$-\calL \cup \calL_\II$ comes from the interior of a certain gnarl that contains
an element of $\overline{\G'}$. By Lemma~\ref{lem:depth<2}, 
if two gnarls
share an interior vertex, one must be nested in the other. So one associated
gnarl must be depth $1$, the other depth $0$. This is prohibited by above
procedure.
\end{proof}

First assume $\calA$ is non-empty. By Claim~\ref{clm:G'bound} and
part $(5)$ of Lemma~\ref{lem:buildA} 
\begin{align*}
|-\calL | + |\calL_\II| &\ge |\G'|-2F_2(g)+(3/7) \sum_{B \in \calA}(n_B+4)  \tag{$**$}\\
         &= |\G'|+(3/7)|\G''|+(3/7) \, 4 \, |\calA| -2F_2(g) \\
         &\ge (3/7)(|\G|+|L(\Sigma)| - 4l(g))-2F_2(g) \\
         &\ge (3/7)(|\calL|-960(g-1)-14F_2(g)-4l(g))-2F_2(g)   
\end{align*}
where $l(g)$ is the function defined in Lemma~\ref{lem:buildA}, and where
the last line uses $(*)$. Thus by Theorem~\ref{thm:Type12boundN}
\[|-\calL | \ge (3/7)(|\calL|-960(g-1)-14F_2(g)-4l(g))-2F_2(g)-960(g-1)\]

By Proposition~\ref{prop:webcount}, $|\calL| \ge (3/4)t - (g-1)/2$.
Therefore
\begin{align*}
t &= |-\calL| + |\calL|  \\
  &\ge (1+3/7)((3/4)t-(g-1)/2)-(3/7)(960(g-1)+14F_2(g)+4l(g)) \\ 
  &\, \, \, \, \, \,  - 2F_2(g) - 960(g-1) \\
  &=(15/14)t - (5/7)(g-1) - (3/7) (960(g-1) + 14F_2(g) + 4l(g)) \\ 
  &\, \, \, \, \, -2F_2(g) - 960(g-1) 
\end{align*}
Thus  $t \le 2w(g)$ where
\begin{align*}
  2w(g) &=10(g-1)+6 \cdot 960 (g-1) + 24 l(g) + 84 F_2(g) + 28 F_2(g) + 14 \cdot 960 (g-1) \\
       &=19,210(g-1) + 24 l(g) + 112 F_2(g) \\
       &=19,210(g-1) + 24 (58(g-1) + 47/2) + 112 (5(g-1)+2) \\
       &=21,162(g-1) + 788 
\end{align*}
proving Theorem~\ref{thm:main} when $\calA$ is non-empty.

Now assume $\calA$ is empty. Then Lemma~\ref{lem:buildA}(1) implies
$|L(\Sigma)| \le 4 l(g)$. Furthermore, $\G = \G'$, and Claim~\ref{clm:G'bound} along
with Theorem~\ref{thm:Type12boundN}
yields $|-\calL| \ge |\G| -2F_2(g) - 960(g-1)$ (this is the analog of $(**)$). Then  $(*)$ gives
$|-\calL| \ge |\calL| - 1920(g-1) - 16F_2(g) - 4l(g)$.
Using Proposition~\ref{prop:webcount}, we get
\begin{align*}
t &= |-\calL| + |\calL|  \\
%
  &\ge (3/2)t - 1921(g-1) - 16F_2(g) - 4l(g)  
\end{align*}
This gives $t \le 2 (1921(g-1) + 16F_2(g) + 4l(g)) < 2w(g)$ where
$w(g)$ is as above.
\end{proof}

\subsection{The collection of annuli, $\calA$}

If $\theta$ is a $\theta$--curve, a {\em curve} in $\theta$ is a circle
$\gamma$ obtained by removing from $\theta$ the interior of an edge.
If $\sigma$ is an extended Scharlemann cycle of length $3$,
$\gamma(\sigma)$ will denote the set of curves in $\theta$--curves
belonging to $\theta(\sigma)$. If $\sigma$ is an extended Scharlemann
cycle of length $2$ or $3$, $a\gamma(\sigma)$ will denote $a(\sigma)
\cup \gamma(\sigma)$ (where $\gamma(\sigma)=\emptyset$ if $\sigma$ has
length $2$).

\begin{lemma}\label{lem:buildA}
There is a (possibly empty) collection $\calA$ of annuli in $\hatF$ such that 
\begin{enumerate}
\item $|\calA| \geq |L(\Sigma)|/4 - l(g)$ where $l(g)=58(g-1)+47/2$;
\item for any $B \in \calA$ there is a $\sigma \in \Sigma$ 
such that each component of $\partial B$ belongs to $a\gamma(\sigma)$;
\item the interiors of any two distinct annuli in $\calA$ are disjoint;
\item no vertex in the interior of an annulus in $\calA$ belongs
to $L(\Sigma)$, and any vertex in the interior of an annulus in $\calA$
that belongs to $\calL$ is either a vertex of a simple 
gnarl or belongs to $\calL_\II$;
\item if $B \in \calA$ and $n_B$ is the number of vertices in the
interior of $B$ that belong to simple gnarls then there are at least
$(3/7)(n_B+4)$ vertices in $B$ that belong to $-\calL \cup \calL_\II$. 
\end{enumerate}
\end{lemma}
\begin{remark}
If $\calA$ is empty ($|\calA|=0$), then (1) says that 
$|L(\Sigma)| \leq 4l(g)$.
\end{remark}


\begin{proof}
By Lemma~\ref{lem:parallelthetas} there are at least $|\theta(\Sigma)|-3F_2(g)$ $\theta$--curves of $\theta(\Sigma)$ that lie in essential annuli.   Since $|L(\Sigma)|/2 \leq |a(\Sigma)| + |\theta(\Sigma)|$ and each element of $a(\Sigma)$ lies in an essential annulus, at least $|L(\Sigma)|/2 - 3F_2(g)$ of the elements of $\ao(\Sigma)$ lie in essential annuli. Let $\calC$ be this subcollection of 
$\ao(\Sigma)$. 
Call two elements of $\calC$ {\em
isotopic} if the cores of the corresponding annuli are isotopic in $\hatF$.  
Then any two elements of $\calC$ which are isotopic are either disjoint
or intersect non-transversely in a single vertex. Furthermore, the
elements of $\calC$ form at most $F_2(g)$ isotopy classes
(using `core' curves from the $\theta$--curves and perturbing 
non-transverse intersections, Lemma~\ref{lem:atmostoneintersection} shows
that a subcollection of non-isotopic elements of $\calC$ have property $P(2)$
-- apply Lemma~\ref{lem:Fkg}). Thus all but
at most $3 F_2(g)$ of these curves lie in isotopy classes with at least 
four elements
of $\calC$. Let $\calC_1, \dots, \calC_m$
be the distinct isotopy classes of $\calC$, each of which contains at least four elements of $\calC$. 

A {\em pinched annulus} $A \subset \hatF$ is a disk $D$ with two points
on its boundary identified to a single point $v$. The {\em interior}
of $A$ is the interior of $D$. Note that under the identification
$\partial D$ becomes the union of the two simple closed curves $\gamma$
and $\gamma'$ that intersect non-transversely at $v$. We write 
$\partial A= \gamma \cup \gamma'$. 

Consider $\calC_i=\{c_1,c_2,\dots,c_n\}$, say, where $c_r \in a\theta(\Sigma), 1 \le r \le n$, and the $c_r$'s are ordered 
sequentially as they lie on $\hatF$. Then we have 
\begin{enumerate}
\item[(R1)] $c_r$ and $c_s$ are disjoint unless $|r-s|=1$ or $\{r,s\}=\{1,n\}$,
and $c_r$ and $c_s$
share a single vertex (Lemma~\ref{lem:atmostoneintersection});
\item[(R2)] for $1 \le r < n$ there exists $A_r \subset \hatF$ such that
 \begin{enumerate} 
    \item $A_r$ is an annulus or pinched annulus according as $c_r \cap c_{r+1}$ is empty or non-empty;
     \item $\partial A_r=\gamma_r \cup \gamma_{r+1}$ where $\gamma_s$ is a
curve in $c_s, s=r,r+1$;
     \item $A_r$ contains $c_r \cup c_{r+1}$;
     \item there are no vertices of any $a\theta(\Sigma)$ in the interior of $A_r$.
 \end{enumerate}
\end{enumerate}

We have $c_r \in a\theta(\sigma_r), \sigma_r \in \Sigma$.

Note that if $A_r$ is pinched then $\sigma_r \neq \sigma_{r+1}$.

\begin{claim}\label{clm:intAr}
If $\sigma_r \ne \sigma_{r+1}$ then the number of vertices of $G_F$
in the interior of $A_r$ is at least $t/2-3$, and at least $t/2-2$
if $A_r$ is pinched.
\end{claim}

\begin{proof} 
We may assume $t > 32(g-1)$, hence Lemmas~\ref{lem:parallelesc} and 
\ref{notwo} imply that $M$ is a Seifert fibered space over $S^2$ with 
three exceptional fibers at least one of which has order $2$ or $3$.
This puts us in the context of Theorem~\ref{thm:main}(3).
Thus $\hatF$ is a genus $2$ splitting surface for $M$ 
with respect 
to which $K$ has the smallest bridge number. Again by 
Lemmas~\ref{lem:parallelesc} and \ref{notwo}, 
$K$ has bridge number at most $k-1$, where
$k=|K \cap A_r|$. Thus $t/2 \le k-1$ (twice the bridge number is an 
upper bound for the number of intersections with a thick surface). Let $k'$ be the number of vertices
in the interior of $A_r$. Then $k=k'+4$ or $k'+3$ according as $A_r$ is an
annulus or a pinched annulus, giving $k' \ge t/2-3$ or $t/2-2$, respectively.
\end{proof}

\begin{claim}\label{clm:countingC}
If the elements of $a\theta(\Sigma)$ are incident to at most $v$ vertices of $G_F$ then 
$|a\theta(\Sigma)| \leq v$. In particular, we may assume 
$v > 2l(g)=116(g-1)+47$.  
\end{claim}

\begin{proof}
Each vertex of $G_F$ belongs to at most two different elements of 
$a\theta(\Sigma)$ by Lemma~\ref{lem:atmostoneintersection}. 
On the other hand, every element of $a\theta(\Sigma)$ involves two vertices of 
$G_F$. Thus $v \geq |a\theta(\Sigma)| = |a(\Sigma)|$ + $|\theta(\Sigma)| \geq
|L(\Sigma)|/2$. If $v \leq 2l(g)$, then $4l(g) \geq |L(\Sigma)|$, and 
taking $\calA = \emptyset$ satisfies Lemma~\ref{lem:buildA}. 
\end{proof}

Let $\calA_i^+$ be the collection $\{A_r| 1 \le r < n\}$ coming from $\calC_i$.

\begin{claim}\label{clm:disjointannuli} For any $i \neq j$, any element of $\calC_i$ is disjoint from some 
element of $\calC_j$. Thus any element of $\calC_i$ is either disjoint from 
a given element of $\calC_j$ or intersects it non-transversely in a 
single vertex.
No vertex in the interior of an annulus of $\calA_i^+$ lies in the 
interior of an annulus of $\calA_j^+$, $i \neq j$.
\end{claim}

\begin{proof}
Assume that $\calA_i^+$ contains two non-adjacent pinched
annuli. Then these annuli have disjoint interiors. By
Claim~\ref{clm:intAr} there are at least $t-4$ vertices of $G_F$
in the interiors of these annuli. Then (R2)(d) implies there are at most 
four vertices
of $G_F$ belonging to $|a\theta(\Sigma)|$, contradicting 
Claim~\ref{clm:countingC}. Thus any pinched annuli in $\calA_i^+$ are 
adjacent and there are at most 
two pinched annuli in $\calA_i^+$.
Since $|\calC_i| \geq 4$, there are at least three disjoint elements
of $\calC_i$. The middle of these is then disjoint from any element of $\calC_j$
as proposed, and any element of $\calC_j$ can intersect an element of $\calC_i$
only non-transversely. Let $A \in \calA_i^+, B \in \calA_j^+$. 
By the above,
the boundary components of $A,B$ can be perturbed to be disjoint. If their
interiors intersected, then some component of $\partial A$ would be isotopic to
one of $\partial B$, a contradiction. Thus $A,B$ share no interior vertices.
\end{proof}

Let $\calA^+ = \cup_{i=1}^m \calA_i^+$. It follows from  
Claims~\ref{clm:intAr}, ~\ref{clm:countingC}, and ~\ref{clm:disjointannuli},
that there are at most two elements $A \in \calA^+$
such that $\partial A = \gamma \cup \gamma', \gamma \in a\gamma(\sigma),
\gamma' \in a\gamma(\sigma')$, where $\sigma \ne \sigma'$, and that,
if there are two, then they are adjacent (and hence belong
to the same $\calA_i^+$). Let $\widetilde{\calA_i^+}$ be $\calA_i^+$
with any such elements removed (only one family is changed).  
Then for each $i$ there is a 
$\sigma \in \Sigma$ such that for any annulus
$B \in \widetilde{\calA_i^+}$, 
$\partial B \subset a\gamma(\sigma)$. In particular, there are no pinched
annuli in $\widetilde{\calA_i^+}$.

If $A^+,B^+ \in \widetilde{\calA_i^+}$, their interiors will be 
disjoint unless they share a $\theta$--curve $\theta$, in which
case two of the edges of $\theta$ will cobound a disk, $D$, of
parallelism in $\hatF$ with $D \subset A^+ \cap B^+$. Since
the cores of the annuli in $\widetilde{\calA_i^+}$ are all isotopic
in $\hatF$, assigning an orientation to this isotopy class allows
us to talk about  one element of $\widetilde{\calA_i^+}$ being 
to the {\em right} or {\em left} of another. If $A^+,B^+ \in 
\widetilde{\calA_i^+}$ share a disk of parallelism $D$, with 
$A^+$ on the left of $B^+$, we define $A=A^+ - \Int D$. Doing this
for all pairs in $\widetilde{\calA_i^+}$ that share a disk of parallelism,
we get a collection $\widetilde{\calA_i^+}$ of annuli in $\hatF$
whose interiors are disjoint. 

Let $\tilde{\calA}=\cup_{i=1}^m \widetilde{\calA_i^+}$. Then 
$\tilde{\calA}$ satisfies conditions $(2)$ and $(3)$ of the lemma.
Also, recall from the first paragraph of the proof that 
$|\calC| \ge |L(\Sigma)|/2 - 3F_2(g)$ and
$\sum_{i=1}^m|\calC_i| \ge |\calC| - 3F_2(g)$.
Hence 
\begin{align}\label{eq:Atilde}
|\tilde{\calA}|=\sum_{i=1}^m |\widetilde{\calA_i^+}| &\ge \sum_{i=1}^m (|\calC_i|-1) - 2 \notag \\
   &\ge |L(\Sigma)|/2 - m -2 - 6F_2(g) \tag{$***$}
\end{align}

\begin{claim}\label{clm:intofA}
There are at most $|\tilde{\calA}|/2 + (m+1)/2$ annuli in $\tilde{\calA}$
with no vertices of $G_F$ in their interior.
\end{claim}

\begin{proof}
The elements of $\widetilde{\calA_i^+}$ appear sequentially along
$\hatF$. Since $\tilde{\calA}=\cup \widetilde{\calA_i^+}$ is obtained from 
$\calA^+=\cup_{i=1}^m \calA_i^+$
by removing either $0,1$ or $2$ adjacent elements from some
$\calA_i^+$, the annuli in $\tilde{\calA}$ appear
in at most $(m+1)$ sequential groups, with, say $n_r$ annuli
in the $r$th group, $1 \le r \le m+1$.

We show that no two consecutive annuli in $\widetilde{\calA_i^+}$ have empty interiors. So assume $A,B$ are consecutive annuli in 
$\widetilde{\calA_i^+}$, neither of which contains vertices in its interior. 
By $(R2)$ and the construction of $\widetilde{\calA_i^+}$, 
$\partial A$ and $\partial B$
belong to $a\gamma(\sigma)$ for some $\sigma \in \Sigma$. But this
contradicts Lemma~\ref{lem:triplets} or Lemma~\ref{lem:thetatriplets}
(or Theorem~\ref{thm:main}$(3)$ holds with the bridge number bound).
Note that if $\sigma$ were of length $3$, letting $\theta \in \ao(\sigma)$
be the $\theta$--curve which meets both $A$ and $B$, then, by the way
we defined the elements of $\tilde{\calA}$, the disk of parallelism
of $\theta$ is contained in either $A$ or $B$. So 
Lemma~\ref{lem:thetatriplets} applies.

It follows that the number of annuli in $\tilde{\calA}$ with no vertices
in their interior is at most 
\[\sum_{r=1}^{m+1} (\frac{n_r+1 }{ 2}) = \frac{|\tilde{\calA}| }{ 2}+\frac {(m+1) }{ 2} \]
%
as claimed.
\end{proof}

Since any label in $L(\Sigma)$ corresponds to a vertex in an element
of $\ao(\Sigma)$, it is clear from the definition of the original
collection $\calA^+$ that no annulus in $\tilde{\calA}$ contains
vertices in its interior corresponding to elements of $L(\Sigma)$.

Finally, we define $\calA$ to be the collection of annuli obtained
by discarding from $\tilde{\calA}$ any annulus whose interior either 
\begin{itemize}
\item has no vertices of $G_F$; or
\item contains the vertex of a wrapping gnarl; or
\item contains the vertex of a simple gnarl with no vertices in its interior.
\end{itemize}
Then $\calA$ satisfies condition $(2),(3),$ and $(4)$ of Lemma~\ref{lem:buildA}.

By Claim~\ref{clm:intofA}, Lemma~\ref{lem:boundwrapping}, and
Lemma~\ref{lem:nonemptygnarl}
\begin{align*}
|\calA|&\ge |\tilde{\calA}| - (|\tilde{\calA}| + m+1))/2 - 7F_2(g)- F_2(g) \\
   &= (|\tilde{\calA}| -  (m+1))/2 - 8F_2(g)  
\end{align*}
Hence, using inequality~(\ref{eq:Atilde}) above, we have
\[  |\calA| \ge |L(\Sigma)|/4 - m - 3/2 -11F_2(g)  \]
By Claim~\ref{clm:disjointannuli} and Lemma~\ref{lem:Fkg}, $m \le F_0(g) =3g-3$.
Setting $F_2(g)=5(g-1)+2$, we get  
\[  |\calA| \ge |L(\Sigma)|/4 - l(g)  \]
where $l(g)=58(g-1)+47/2$.
Thus $\calA$ satisfies condition $(1)$.

It remains to show that $\calA$ satisfies condition $(5)$. 

\begin{claim}\label{clm:atleast2}
Each element of $\calA$ contains at least two vertices in its interior.
\end{claim}

\begin{proof}
Assume the annulus $B \in \calA$ contains a single vertex in its interior.
By (2), the components of $\partial B$
belong to $a\gamma(\sigma)$ for the same element, $\sigma$, of $\Sigma$.

First assume that $\sigma$ is an extended Scharlemann cycle of length
$2$. Then there is an annulus $A \subset A(\sigma)$ that runs between 
the components of $\partial B$. $A$ intersects $B$ 
only along their mutual boundary, forming a 2-torus, $T$. 
Furthermore, $K$ may be perturbed off $A$ so that it intersects $T$ either
once or three times. But then $T$ is a non-separating
surface in $M$, which is impossible. 

When $\sigma$ has length $3$, we arrive at an analogous contradiction.
$B$ contains a $\theta$--curve $\theta \in \theta(\sigma)$. 
As in the proof of Corollary~\ref{3ingnarl}, if the unique vertex of
$G_F$ in the interior of $B$ 
lay between the two parallel edges of $\theta$ in $B$, then $G_F$
would have a 1-sided face (the number of edges interior to the parallelism
is less than the valence of the interior vertex). Hence this vertex cannot lie
in this disk of parallelism. 
The proof of Lemma~\ref{lem:annulusintheta}
shows how to construct an annulus $A$ within
$\Theta(\sigma)$ that runs between components of $\partial B$ such that 
$K$ can
be perturbed off of $A$ to intersect $B$ transversely in either one or
three points. But then $A \cup B$ is a non-separating torus in $M$, a
contradiction. 
\end{proof}

Let $B \in \calA$, let $n_B$ be the number of vertices in the interior of
$B$ that belong to simple gnarls, and $\overline{n_B}$ the number 
that correspond to labels in $-\calL \cup \calL_\II$. 
We must show that $\overline{n_B} \ge (3/7)(n_B+4)$. Note that if
a vertex in the interior of $B$ belongs to $\calL$, it must in fact
belong to $\calL_\II$ by condition $(4)$ of Lemma~\ref{lem:buildA}.

If $n_B=0$ then $\overline{n_B} \ge 2 > (3/7) 4$ by Claim~\ref{clm:atleast2}.

If $n_B > 0$ then the corresponding vertices belong to simple gnarls
where each such gnarl contains vertices in its interior.
By Lemma~\ref{lem:depth<2}, simple gnarls are of depth $0$ or $1$.
By Corollary~\ref{3ingnarl}, each 
depth $0$ gnarl contains in its interior at least three vertices, which 
must belong to $-\calL \cup \calL_\II$ and hence
contribute $3$ to $\overline{n_B}$. By 
Lemma~\ref{lem:depth<2}, a depth $1$ gnarl must share a vertex with an
internal depth $0$ gnarl. By Lemma~\ref{lem:samelabels}, a vertex
of $G_F$ belongs to at most two different gnarls.

Let $n_B'$ be the vertices
that belong to gnarls lying in $B$ (possibly with vertices on $\partial B$).
Then $n_B \le n_B'$. Let $n$ be number of
depth $0$ gnarls contained in $B$, $k$ be the number of vertices in depth $1$ gnarls in $B$
that are not in depth $0$ gnarls, and $r$ the number of vertices that are contained
in two depth $0$ gnarls in $B$. From the observations in the 
preceding paragraph, $n_B' = 2n-r+k$.
Furthermore, there are at least $3n$ vertices in $-\calL \cup \calL_\II$ in the 
interior of $B$. 
Thus $|-\calL \cup \calL_\II| \ge 3n = n_B+[(n_B'-n_B)+r+(n-k)]$. Noting
that $n \ge k$, the bracketed quantity is non-negative.  
Therefore, 
\[
\overline{n_B} \quad \ge \quad
\begin{cases}
n_B, & \text{if} \, n_B \equiv 0 \, (\text{mod} \, 3) \\
n_B+2, & \text{if} \, n_B \equiv 1 \, (\text{mod} \, 3) \\
n_B+1, & \text{if} \, n_B \equiv 2 \, (\text{mod} \, 3)
\end{cases}
\]
It follows that $\overline{n_B} \ge (3/7)(n_B+4)$.
\end{proof}

\section{Appendix, intersections with incompressible surfaces}\label{sec:appendix}

\begin{thm}\label{thm:appendix}
There is a linear function $w_I \colon \mathbb N \to \mathbb N$ with the following
property.
Let $K'$ be a hyperbolic knot in $S^3$, $M = K'(p/q)$ where $q \ge 2$,
and $K$ the core of the attached solid torus in $M$.
Let $S$ be an orientable, incompressible surface in $M$ of
genus $g$.
Then $K$ can be isotoped to intersect $S$ at most $w_I(g)$ times.
\end{thm}

\begin{remark} Compare the above with Theorem 4 of \cite{verjovsky} which says the following.  Let $K'$ be an hyperbolic knot in $S^3$ that does not contain a genus $g$, closed incompressible surface in its exterior and such that $M=K'(p/q)$ contains an incompressible surface of genus $g$. Let $K$ be
the core of the attached solid torus in $M$. If $q \ge 4$, then there is an incompressible surface of genus $g$ in $M$ which $K$ intersects at most $\frac {36} {q-3} (g-1)$ times.
\end{remark}

\begin{proof}
The proof is just that of Theorem~\ref{thm:main} given for strongly irreducible Heegaard
surfaces, except basically that the proof of Claim~\ref{clm:stronglyirred}
must be altered to replace the assumption of strong irreducibility
with incompressibility. In fact the arguments for an incompressible surface tend to be easier than in the context of an Heegaard surface,
but as those arguments have been already been made, we will make use
of them. As $M$ is a non-integral surgery on a
knot in $S^3$, it is irreducible (\cite{gl:oidscyrm}).

We may assume that $M$ is atoroidal. Otherwise by \cite{gl:nitds}
$K'$ is an Eudave-Mu\~noz knot and, consequently, $M$ is the union
of two Seifert fibered spaces over the disk, each with two exceptional
fibers, along a torus $S$. The Seifert fibers of the two halves intersect once along
$S$. Hence $S$ is the unique orientable incompressible surface. (Any other
connected incompressible surface would
have to be the union of  horizontal surfaces in the two Seifert fibered
spaces.   
Consequently such a surface would be non-separating, contradicting that $M$ is a rational homology sphere.)  
Again \cite{gl:nitds} shows
that $K$ can be isotoped to intersect $S$ twice. Note that this also
shows that
$M$ cannot be a Seifert fibered space as any such which is atoroidal
contains no incompressible surface. Finally, recall that $M$ cannot
contain a Klein bottle by Theorem 1.3 of \cite{gl:dsokcetI} (see also
\cite{boyerzhang}).

So let $\hatF$ be an incompressible surface in $M$ of genus at least $2$
and isotop $K$ to intersect $\hatF$ minimally. Let $F$ be the punctured surface, $\hatF \cap X$, where $X$ is the exterior of $K$ in $M$ (which is the exterior of $K'$ in $S^3$). The thin position argument given in Lemma 4.4 of \cite{gabai:fatto3mIII} (in place of, and simpler than, that of
\cite{rieck}) shows that there is a level surface
$\hatQ$ in a thin presentation of $K'$ in $S^3$ such that the
corresponding punctured surface $Q = \hatQ \cap X$ intersects
$F$ in arcs which are essential in both $Q$ and $F$. That is,
no arcs of $F \cap Q$ are boundary parallel in either $F$ or $Q$.

The combinatorial arguments of Section~\ref{sec:graphcombo} go through in this context.
In particular,  Proposition~\ref{prop:webcount} still holds for
$\calL$ as defined in that section.
For each label $x \in \calL$, there is a bigon or trigon in $\Lambda_x$,
and the classification of such as either an extended Scharlemann cycle
or as a trigon of Type I, I\!I, or I\!I\!I is the same.

One checks that all of the results of section~\ref{sec:thinning} hold;
however, with the helpful addition now that $M$ cannot be a Seifert fibered space.
Note that an isotopy which gave a thinning in this section now
reduces the intersection number of $K$ and $\hatF$. Indeed this is overkill. For example,
neither conclusion of Lemma~\ref{lem:SFS} can hold in this context:
conclusion $(1)$ allows us to reduce the intersection of $K$ with
$\hatF$, and conclusion $(2)$ says that $M$ is a Seifert fibered space.
There is one place in section~\ref{sec:thinning} where the fact that
$\hatF$ is a strongly irreducible Heegaard surface is used. That is
in the proof of Claim~\ref{clm:stronglyirred}. Strong irreducibility
is used there to show that the Seifert fibered space $N$ does not
lie in a 3-ball. In the current context, this does not occur as it
would imply that a curve that is homotopically essential in $\hatF$
lies in this 3-ball, and consequently is homotopically trivial in $M$
-- thereby violating the incompressibity of $\hatF$.

The results of sections~\ref{sec:typeIorII} and \ref{sec:typeIII} follow from the
supporting Lemmas in section~\ref{sec:thinning} and the fact that $M$
is irreducible, atoroidal, and contains no Klein bottle.

Finally, section~\ref{sec:proofofmain} pulls together these supporting Lemmas to give
the bound $w_I(g)=2w(g)$ on $|K \cap \hatF|$.
\end{proof}


\providecommand{\bysame}{\leavevmode\hbox to3em{\hrulefill}\thinspace}
\providecommand{\MR}{\relax\ifhmode\unskip\space\fi MR }
\providecommand{\MRhref}[2]{%
  \href{http://www.ams.org/mathscinet-getitem?mr=#1}{#2}
}
\providecommand{\href}[2]{#2}






\end{document}